\documentclass{amsart}
\allowdisplaybreaks[4]

\newcommand{\R}{\mathbf{R}}

\newcommand{\sig}{\sigma}

\newcommand{{\ba}}{\bf a}
\newcommand{\ve}{\varepsilon}
\newcommand{\la}{\lambda}
\newcommand{\La}{\Lambda}
\newcommand{\ga}{\gamma}

\newcommand{\pa}{\partial}
\newcommand{\ra}{\rightarrow}

\newcommand{\del}{\delta}
\newcommand{\vp}{\varphi}

\newcommand{\cd}{\cdot}

\newcommand{\De}{\Delta}
\newcommand{\al}{\alpha}

\newcommand{\be}{\begin{equation}}
\newcommand{\ee}{\end{equation}}

\newcommand{\nn}{\nonumber}

\newtheorem{lem}{Lemma}{\bf}{\it}
\newtheorem{rem}{Remark}{\it}{\rm}

\newtheorem{theorem}{Theorem}
\newtheorem{proposition}{Proposition}
\newtheorem{corollary}{Corollary}
\numberwithin{theorem}{section}
\numberwithin{lem}{section}
\numberwithin{rem}{section}
\numberwithin{equation}{section}
\numberwithin{proposition}{section}
\numberwithin{corollary}{section}
\title[Stochastic Variational Formulas ]{Stochastic Variational Formulas for Solutions to Linear Diffusion Equations}

\author{Joseph G. Conlon and Mohar Guha}

\address{University of Michigan\\ Department of Mathematics\\ Ann Arbor,
  MI 48109-1109}
\email{conlon@umich.edu, mguha@umich.edu}


\begin{document}

\maketitle

\begin{abstract}
This paper is concerned with   solutions to a one dimensional linear diffusion equation and their relation to some problems in stochastic control theory.  A stochastic variational formula is obtained for the logarithm of the solution to the diffusion equation, with terminal data which is the characteristic function of a set. In this case the terminal data for the control problem is singular, and hence standard theory does not apply. The variational formula is used to prove convergence in the zero noise limit of the cost function for the stochastic control problem and its first derivatives,  to the corresponding quantities for a classical control problem.
\end{abstract}

\section{Introduction.}
In this paper we shall be concerned with solutions to a linear diffusion equation and their relation to some problems in stochastic control theory.  Let $T > 0$ and $b(y,t), \ y \in \R$, $t \le T$, be a function differentiable in $y$ with derivative continuous in $(y,t)$ which satisfies the uniform bound
\be \label{A1}
\sup \Big\{ |\pa b(y,t)/\pa y| : y \in \R,  \ t \le T \} \le A,
\ee
for some constant $A \ge 0$.  We shall be interested in solutions $u_\ve(x,y,t)$ to the equation
\be \label{B1}
\frac{\pa u_\ve}{\pa t} + b(y,t) \; \frac{\pa u_\ve}{\pa y} + \frac \ve 2 \frac{\pa^2 u_\ve}{\pa y^2} = 0, \quad y \in \R, \; t < T,
\ee
with terminal condition
\be \label{C1}
\lim_{t \ra T} u_\ve (x,y,t) = 0 \ \ {\rm for} \ \ y < x,
\ee
\[  \lim_{t \ra T} u_\ve(x,y,t) = 1 \ \ {\rm for} \ \ y > x.  \]
It follows from standard methods \cite{fried} that $u_\ve(x,y,t)$ is a continuous function of $(x,y,t)$ for $x,y \in \R, \; t < T$, and that also the first derivative $u_\ve(x,y,t)$ in $t$ and second derivatives in $(x,y)$ exist and are continuous in $(x,y,t)$.  Evidently $u_\ve(x,y,t)$ is given in terms of the fundamental solution $G_\ve(y,y',t,T)$ for (\ref{B1}) by the formula
\be \label{D1}
u_\ve(x,y,t) = \int^\infty_x \ G_\ve(y,y',t,T)dy'.
\ee
It is well known \cite{kar} that if $b(\cd,\cd)$ satisfies (\ref{A1}) then the stochastic differential equation
\be \label{E1}
dY_\ve(s) = b(Y_\ve(s), s)ds + \sqrt{\ve} \; dW(s),
\ee
where $W(\cdot)$ is Brownian motion, is uniquely solvable in the interval $t \le s \le T$ with given initial condition $Y_\ve(t) = y$.  Furthermore, $u_\ve(x,y,t)$ is related to solutions of (\ref{E1}) by the identity,
\be \label{F1}
u_\ve(x,y,t) = P \left( Y_\ve(T) > x \;|\; Y_\ve(t) = y \right), \quad  t < T.
\ee

The connection between solutions of (\ref{B1}), (\ref{C1}) and control theory comes via the function $q_\ve(x,y,t)$ defined by
\be \label{G1}
u_\ve(x,y,t) = \exp[-q_\ve(x,y,t)/ \ve ].
\ee
In view of (\ref{F1}) the function $q_\ve$ is positive, and by virtue of (\ref{B1}), (\ref{C1}) it satisfies the PDE
\be \label{H1}
\frac{\pa q_\ve}{\pa t} + b(y,t) \; \frac{\pa q_\ve}{\pa y} - \frac 1 2 \left(\frac{\pa q_\ve}{\pa y}\right)^2 +  \frac \ve 2 \;  \frac{\pa^2 q_\ve}{\pa y^2} = 0, \quad y \in \R, \; t < T,
\ee
with terminal condition
\be \label{I1}
\lim_{t \ra T} q_\ve(x,y,t) = \infty \ \ {\rm for} \ \ y < x,
\ee
\[	\lim_{t \ra T} q_\ve(x,y,t) = 0 \ \ {\rm for} \ \ y > x.  \]
If we let $\ve \ra 0$ in (\ref{H1}) we obtain a Hamilton-Jacobi equation, and therefore should expect that the limit of $q_\ve(x,y,t)$ as $\ve \ra 0$ is given by the solution of a variational problem.  This turns out to be the case.  Let $q(x,y,t)$ be defined by 
\be \label{J1}
q(x,y,t) = \min \left\{ \frac 1 2 \int^T_t \left[ \frac{dy(s)}{ds} - b( y(s), s) \right]^2 \; ds \  \Big|  \ y(t) = y, \  y(T) > x\right\}.
\ee
Thus the functional in (\ref{J1}) is minimized over all paths $y(s), \ t \le s \le T$, with initial point $y(t) = y$ and terminal point $y(T) > x$.  Define the function $F(x,t), \; x \in \R, \; t \le T$, by $F(x,t) = y(t)$ where $y(\cd)$ is the solution to the terminal value problem,
\be \label{K1}
\frac{dy(s)}{ds} = b( y(s), s), \quad  s \le T, \ y(T) = x.
\ee
Then one easily sees that $q(x,y,t) = 0 $ if $y \ge F(x,t)$, whence the function $q(x,y,t)$ is nontrivial only for sufficiently large negative values of $y$.  In $\S 3$ we prove the following theorem showing that $q_\ve$ converges to $q$ as $\ve \ra 0$:
\begin{theorem} Assume $b(\cd,\cd)$ satisfies (\ref{A1}).  Then for $x,y \in \R, \; t < T, \; 0 < \ve < 1$, there is a constant $C$ depending only on $x,y,t,T,A$ such that
\be \label{L1}
|q_\ve(x,y,t) - q(x,y,t)| \le C\sqrt{\ve}.
\ee
\end{theorem}
Inequalities of the type (\ref{L1})  for terminal data which is not singular- unlike in the case of (\ref{I1})-  have been known for many years \cite{cl, flem}.  A short elegant proof of this has recently been given in \cite{evans}.  The inequality (\ref{L1}) implies via (\ref{F1}) the large deviation result for solutions to the stochastic equation (\ref{E1}),
\be \label{M1}
\lim_{\ve \ra 0} \; \ve \; \log \Big[ P \left( Y_\ve(T) > x \; |\; Y_\ve(t) = y \right)\Big] = -q(x,y,t),
\ee
a result which also follows from  Theorem 1.1 of Chapter 4 of \cite{fw}.

In proving Theorem 1.1 we take the approach of showing that in some sense $q_\ve(x,y,t)$ is the cost function of a stochastic control problem.  The formal limit as $\ve \ra 0$ of this stochastic control problem is a classical control problem with cost function $q(x,y,t)$ given by (\ref{J1}).  The stochastic control problem can be described as follows:  Let $y_\ve(\cd)$ be the solution to the stochastic differential equation,
\be \label{N1}
dy_\ve(s) = \la_\ve(\cd,s)ds + \sqrt{\ve} \; dW(s),
\ee
where $\la_\ve(\cd,s)$ is a non-anticipating function.  The cost function for the problem is given by the formula,
\be \label{O1}
q_\ve(x,y,t) = \min_{\la_\ve} E \left[ \frac 1 2 \int^T_t \left[ \la_\ve(\cd,s) - b( y_\ve(s), s) \right]^2 \; ds \  \Big| \  y_\ve(t) = y,\; y_\ve(T) > x\right].
\ee
Thus the minimum in (\ref{O1}) is to be taken over all non-anticipating $\la_\ve(\cd,s)$, $t \le s < T$, which have the property that the solutions of (\ref{N1}) with initial condition $y_\ve(t) = y$ satisfy the terminal condition $y_\ve(T) > x$ with probability 1.  One expects that the function $q_\ve(x,y,t) $ of (\ref{O1}) is identical to the function $q_\ve(x,y,t) $ of (\ref{G1}), but this is not so easy to prove.  An immediate question that arises is how to define a suitable space of non-anticipating functions $\la_\ve(\cd,s), \; t \le s < T$, which have the property that solutions of (\ref{N1}) with initial condition $y_\ve(t) = y$ satisfy $y_\ve(T) > x$ with probability 1.

Instead of attempting to establish the formula (\ref{O1}) with $q_\ve(x,y,t)$ given by (\ref{G1}), we shall confine ourselves to the simpler problem of showing that the expectation on the RHS of (\ref{O1}) is greater than or equal to $q_\ve(x,y,t)$ for certain non-anticipating functions $\la_\ve(\cd,s),  \ t \le s < T$, and that there is equality when $\la_\ve(\cd,s)$ is given by the formula
\be \label{P1}
\la_\ve(\cd,s) = \la^*_\ve  ( x, y_\ve(s), s) =  b( y_\ve(s), s) - \frac{\pa q_\ve}{\pa y} \; ( x, y_\ve(s), s).
\ee
In order to prove Theorem 1.1 it is actually only necessary to prove equality in (\ref{O1}) in the approximate sense
\be \label{Q1}
q_\ve(x,y,t) =  E \left[ \frac 1 2 \int^{T-\sqrt{\ve}}_t \left[ \la^*_\ve ( x,y_\ve(s), s) - b( y_\ve(s), s) \right]^2 \; ds \; \Big| \; y_\ve(t) = y\right] + O (\sqrt{\ve} ).
\ee
The identity (\ref{Q1}) turns out to be much easier to establish than the equality in (\ref{O1}) when $\la_\ve(\cd,s), \  t \le s < T$, is given by (\ref{P1}).

We turn to the proof of this equality in $\S 4$ and $\S 6$.  In $\S 4$ we show that the solution $y_\ve(s), \  t \le s < T$, of (\ref{N1}) with initial condition $y_\ve(t) = y$ and $\la_\ve(\cdot,s)$ given by the optimal controller (\ref{P1}), has the property that
\be \label{R1}
\lim\inf_{t \ra T} y_\ve(t) > x \ \ {\rm with \ probability} \ \ 1 .
\ee
The proof of (\ref{R1}) depends crucially on obtaining a lower bound on the derivative of the function $q_\ve$ of (\ref{G1}),
\be \label{S1}
-\frac{ \pa q_\ve}{\pa y} \; (x,y,t) \ge \frac{x-y}{T-t} \; [1-\eta(\del)], \quad  0 < T-t< \del, \; x-y < \ga,
\ee
where $\ga$ is independent of $\del$ and $\lim_{\del \ra 0} \eta(\del) = 0$.  Observe that the inequality (\ref{S1}) is only non-trivial for $y < x$ since $-\pa q_\ve(x,y,t) / \pa y \ge 0$, $y \in \R$, by the maximum principle.  The proof of (\ref{S1}) relies on the use of the Cameron-Martin formula \cite{simon} applied to the diffusion $Y_\ve(\cd)$ of (\ref{E1}).  One can see from (\ref{D1}) that the inequality (\ref{S1}) gives some information about the short time asymptotics of fundamental solutions to diffusion equations.  There has been much research over several decades \cite{kan, min, mol, var} devoted to this subject.  In particular, Molchanov \cite{mol} has obtained short time asymptotic formulas for diffusions with bounded drift.  These results have been used by Fleming and Sheu \cite{fs} to prove a representation formula analogous to (\ref{O1}) for the logarithm of the fundamental solution.

In order to establish that the expectation on the RHS of (\ref{O1}) with $\la_\ve(\cd,s), \  t \le s < T$, given by (\ref{P1}) is equal to the LHS, one needs to prove that the inequality (\ref{S1}) holds uniformly for $y \in \R$ i.e. $\ga = \infty$.  This turns out to be a considerably more difficult task than proving (\ref{S1}) for some $\ga > 0$.  It is not possible to obtain estimates by means of the Cameron-Martin formula, and instead one uses an induction argument.  The problem of obtaining a uniform lower bound (\ref{S1}) is closely related to the problem of estimating probabilities for the diffusion $Y_\ve(\cd)$ of (\ref{E1}) tied at 2 different times.  In $\S 5$ we prove the following :

\begin{theorem}
Suppose $b(\cd, s), \; 0 \le s \le T$, satisfies (\ref{A1}) and in addition $b(0,s) = 0, \; 0 \le s \le T$.  Then there exist positive universal constants, $\eta, C_1, C_2, \ga_1, \ga_2$ such that
\begin{multline} \label{T1}
P \left( Y_\ve(t) < \frac{C_1(T-t)y} T \ \Big| \ Y_\ve(0) = y, \ Y_\ve(T) = 0 \right) \\
\le \exp \left[ - \ \frac{\ga_1(T-t)y^2}{\ve T^2} \right], \quad y < -T\sqrt{\ve/(T-t)},  
\end{multline}
\begin{multline} \label{U1}
P \left( Y_\ve(t) > \frac{C_2(T-t)y} T \ \Big| \ Y_\ve(0) = y, \ Y_\ve(T) = 0 \right)
\\
\le \exp \left[ - \ \frac{\ga_2(T-t)y^2}{\ve T^2} \right], \quad  y < -T\sqrt{\ve/(T-t)},  
\end{multline}
provided $AT < \eta, \ T-t < T/2$.
\end{theorem}

In $\S 6$ we not only show that the expectation on the RHS of (\ref{O1}) with $\la_\ve(\cd,s)$ given by (\ref{P1}) equals the LHS.  We also obtain corresponding formulas for the first derivatives of $q_\ve(x,y,t)$ in $x$ and $y$.  An immediate consequence of this-Corollary 6.1- is that the fundamental solution $G_\ve$ for (\ref{B1}) satisfies the inequality
\be \label{V1}
G_\ve(y,x,t,T) \le [1+(T-t)A]u_\ve(x,y,t) \left[ \frac{-2\log u_\ve(x,y,t)}{\ve(T-t)} \right]^{1/2},
\ee
where $A$ is the constant in (\ref{A1}) and $u_\ve(x,y,t)$ is given by (\ref{D1}).  The inequality (\ref{V1}) appears to be nontrivial even in the case $b \equiv 0$, where it states that the cumulative distribution function $N(\cdot)$  for the standard normal variable,
\be \label{W1}
N(z) = \frac 1 {\sqrt{2\pi}} \ \int^z_{-\infty} \ \exp(-\rho^2/2) \ d\rho=
\frac{1}{2}+\frac{1}{2}  {\rm sign}(z){\rm  \ erf}\left(\frac{|z|}{\sqrt{2}}\right),
\ee
satisfies the inequality
\be \label{X1}
\exp(-z^2/2) \le 2\sqrt{\pi} N(z) \left[ -\log N(z) \right]^{1/2}, \ z \in \R.
\ee

Let us assume now that the function $b(y,t)$, in addition to satisying (\ref{A1}), is also concave in $y$ for each $t \le T$.  In $\S 2$ we show that in this case the function $q(x,y,t)$ of (\ref{J1}) is $C^1$ in $(x,y,t)$ and is a classical solution of the $\ve = 0$ Hamilton-Jacobi equation (\ref{H1}).  Furthermore, for any $t < T$ the function $q(x,y,t)$ is convex in $(x,y)$ and its second derivatives in $(x,y)$ exist and are continuous on the set $\{ (x,y,t) : x,y \in \R,  \ t < T,  \ y \not= F(x,t)\}$, where $F(x,t)$ is the function defined by (\ref{K1}).  In the Appendix we prove using the method of Korevaar \cite{gk, gp, kor} that the function $q_\ve(x,y,t)$ defined by (\ref{G1}) is also convex in $(x,y)$ for any $t < T$.  Although Korevaar's method is simple in concept, considerable difficulty arises here in its implementation due to the fact that we need to approximate  solutions of the {\it linear} equation (\ref{B1}) by solutions of a {\it quasi-linear } equation 
(\ref{V6}). Hence we need regularity theory-Proposition  A2- for solutions to quasi-linear equations \cite{fried, lieb}. Alternative approaches to Korevaar's method \cite{all, lm} seem to also give rise to comparable technical difficulties in the implementation.

The proof that for fixed $(x,t)$ the function $q_\ve(x,y,t)$  is convex in $y$ -Theorem A1-is much easier to establish than the joint convexity in $(x,y)$. Using this fact and the representation theorem of $\S 6$ we prove in $\S 7$ convergence of first derivatives of $q_\ve(x,y,t)$ in $(x,y)$ to first derivatives of $q(x,y,t)$ as $\ve \ra 0$.

\begin{theorem}  Assume $b(\cd,\cd)$ satisfies (\ref{A1}) and in addition that $b(y,t)$ is concave in $y$ for each $t \le T$.  Then $q(x,y,t)$ is $C^1$ in $(x,y,t)$ for $t < T$ and
\be \label{Y1}
\lim_{\ve \ra 0} \ \frac{\pa q_\ve}{\pa x} (x,y,t) = \frac{\pa q}{\pa x} (x,y,t), \ \ x,y \in \R, \ t < T,
\ee
\[ \lim_{\ve \ra 0} \ \frac{\pa q_\ve}{\pa y} (x,y,t) = \frac{\pa q}{\pa y} (x,y,t), \ \ x,y \in \R, \ t < T. \]
\end{theorem}

Theorem 1.3 gives no rate of convergence as $\ve \ra 0$ like in Theorem 1.1, but if one assumes some H\"{o}lder continuity of $\pa b(y,t)/\pa y$ in $y$,  then the proof of the theorem yields a rate of convergence which is a power of $\ve$.  It is of some interest to compare Theorem 1.3 to the results of Kifer \cite{kif} on the asymptotics of the fundamental solution $G_\ve(y,x,t,T)$ defined by (\ref{D1}) as $\ve \ra 0$.  In that paper asymptotic formulas are established by using the fact that $G_\ve(y,\cdot,t,T)$ is the probability density function for the random variable $Y_\ve(T)$ conditioned on $Y_\ve(t) = y$.  Estimates on the probability density are then obtained by using large deviation techniques \cite{fw}.  Emphasis in the paper is placed on the local nature of the result.  Thus the behavior of the drift $b(\cd,\cd)$ far away from the minimizing trajectory in (\ref{J1}) is shown to be largely irrelevant.

\section{A Classical Control Problem}
 Let $b(y,s), \ y \in \R, \ s \le T$, satisfy (\ref{A1}) and consider the control dynamics
\be\label{A2}
\frac{dy}{ds} = \la(s), \quad  t \le s \le T, \ \ \ y(t) = y,
\ee
where the controller $\la(s), \ t \le s \le T$, is assumed to be piece-wise continuous.  We shall be interested in the optimal control problem with cost function $q(x,y,t), \ x,y \in \R, \ t < T$,  defined by
\be\label{B2}
q(x,y,t) = \min_{\la(\cd)} \left\{ \frac 1 2 \int^T_t \left[\la(s) - b(y(s),s)\right]^2ds \ \Big| \ y(t) = y, \ y(T) > x \right\}.
\ee
Formally the function $q(x,y,t)$ of (\ref{B2}) satisfies the Hamilton-Jacobi equation,
\be\label{C2}
\frac{\pa q}{\pa t} + b(y,t) \ \frac{\pa q}{\pa y} - \frac 1 2 \; \left( \frac{\pa q}{\pa y}\right)^2 = 0.
\ee
Since the minimum in (\ref{B2}) is over paths $y(s), \ t \le s\le T$, satisfying $y(T)>x$, the terminal condition on the PDE (\ref{C2}) is given by
\begin{eqnarray}\label{D2}
\lim_{t\ra T} \ q(x,y,t) &=& \infty, \quad y < x, \\
\lim_{t\ra T} \ q(x,y,t) &=& 0, \quad y > x, \nn
\end{eqnarray}
The optimal controller $\la(\cd)$ for (\ref{B2}) is given by the formula
\be\label{Z2}
\la(s) =\la^*(x,y(s),s)= b(y( s),s) - \pa q(x,y(s),s)/\pa y, \quad  t \le s \le T,
\ee
and the Euler-Lagrange equation for the minimizing trajectory  by
\be \label{H2}
\frac d{ds} \left[ \frac{dy}{ds} - b(y(s),s) \right] + \frac {\pa b}{\pa y}(y(s),s) \left[ \frac{dy}{ds} - b(y(s),s) \right] = 0,  \quad t \le s \le T.
\ee

Our first goal is to prove that there exists a minimizer for the variational problem.  We have already observed  that if $F(\cdot,\cdot)$ is the function defined by (\ref{K1}), then $q(x,y,t)=0$ if $y\ge F(x,t)$. Evidently in this case there is a unique minimizer  $y(\cdot)$ for (\ref{B2}), which is the solution to the differential equation (\ref{K1}) with initial condition $y(t)=y$.
 For $y < F(x,t)$ we need to define a space of functions $y(s), \ t \le s \le T$, over which to minimize the expression in (\ref{B2}).  For any $f \in L^2[t,T]$ let $y(\cd)$ be determined from $f$ by
\be \label{G2}
y(s) = \ \ y \ \ + \int^s_t \ f(s')ds'.
\ee
Thus $y(\cd)$ is Holder continuous of order 1/2 on $[t, T]$ and $y(t) = y$.  We define $E_{x,y,t}$ to be the space of all such functions $y(\cd)$ with $f \in L^2[t,T]$ and $y(T) \ge x$.  The distance between 2 functions $y_1, y_2 \in E_{x,y,t}$ is given by the norm $\|y_1 - y_2\| = \|f_1 -f_2\|_2$, where $y_1$ corresponds to $f_1$ and $y_2$ to $f_2$ in (\ref{G2}).  Evidently the space $E_{x,y,t}$ is complete under this distance function. Now (\ref{H2}) indicates that on a minimizer $y(s), \ t \le s \le T$, for (\ref{B2}) the expression $y'(s) - b(y(s), s)$ does not change sign for $s$ in the interval  $[t,T]$.  We shall show that if $y < F(x,t)$ the sign is in fact positive.
\begin{proposition} Assume the function $b(\cd, \cd)$ satisfies (\ref{A1}).  Then there exists a minimizer $y(\cd) \in E_{x,y,t}$ of the variational problem (\ref{B2}).  Any minimizer $y(\cd)$ has the property that $y(\cd)$ is $C^1$ in $[t,T]$.  If $y < F(x,t)$ then $y'(s) > b(y(s),s), \ t \le s \le T$, and $y(T) = x$.  The function $q(x,y,t)$ of (\ref{B2}) is continuous for $(x,y) \in \R^2, \ t < T$.
\end{proposition}
\begin{proof} We define a functional ${\mathcal F}[y(\cd)]$ on $E_{x,y,t}$ by
\be \label{R2}
{\mathcal F}[y(\cd)] = \frac 1 2 \ \int^T_t \left[ \frac{dy}{ds} - b(y(s),s) \right]^2 \;ds.
\ee
Following the standard method \cite{rs} we show that ${\mathcal F}[\cd]$ is weakly lower semi-continuous on $E_{x,y,t}$.  Thus let $y_N(\cd), N \ge 1$, be a sequence in $E_{x,y,t}$ converging weakly to $y_\infty(\cd) \in E_{x,y,t}$.  Hence if $f_N, \; N \ge 1, \ f_\infty$ in $L^2[t,T]$ are associated with $y_N(\cd), N\ge 1$, and $y_\infty(\cd)$ respectively, we have that
\be \label{I2}
\lim_{N\ra \infty} \left< f,f_N \right> = \left< f,f_\infty \right>, \quad  f \in L^2[t,T].
\ee
From the uniform boundedness principle \cite{rs} it follows that $\displaystyle{\sup_{N\ge 1}}\|f_N\|_2 < \infty$.  It also follows from (\ref{I2}) that $\displaystyle{\lim_{N\ra \infty}} y_N(s) = y_\infty(s)$, $t \le s \le T$, and $\sup\{|y_N(s)| \ : \ N\ge 1, \; t\le s \le T\} < \infty$.  Hence by the dominated convergence theorem one has that
 \be \label{J2}
\lim_{N\ra \infty} \ \int^T_t \ b( y_N(s),s)^2\; ds = \int^T_t \ b( y_\infty(s),s)^2\; ds.
\ee
Using the uniform boundedness of the $f_N, \ N \ge 1$, we also have that
\[ \lim_{N \ra \infty} \int^T_t \ \left[b( y_N(s),s) - b(y_\infty (s),s)\right] f_N(s)ds = 0.  \]
Hence using (\ref{I2}) again we conclude that
\be \label{K2}
\lim_{N \ra \infty} \int^T_t \ b( y_N(s),s)f_N(s)ds = \int^T_t \ b( y_\infty(s),s)f_\infty(s)ds.
\ee
Now (\ref{J2}), (\ref{K2}) imply that
\begin{multline*}
\liminf_{N \ra \infty} {\mathcal F}[y_N(\cd)] = \frac 1 2 \ \liminf_{N \ra \infty} \int^T_t \left[ \frac{dy_N(s)}{ds}\right]^2ds \\
-\int^T_t b(y_\infty(s),s) \frac{dy_\infty(s)}{ds} \ ds + \frac 1 2 \; \int^T_t b(y_\infty(s),s)^2 \; ds.
\end{multline*}
The lower semi-continuity of ${\mathcal F}[\cd]$ on $E_{x,y,t}$ follows from the inequality,
\[ \frac 1 2 \  \int^T_t \left[ \frac{dy_\infty(s)}{ds}\right]^2 ds \le \frac 1 2 \ \liminf_{N \ra \infty} \int^T_t \left[ \frac{dy_N(s)}{ds}\right]^2  ds,   \]
which is a consequence of the convexity of the Dirichlet form \cite{rs}.  One easily concludes from the lower semi-continuity of ${\mathcal F}[\cd]$ the existence of a minimizer $y(\cd) \in E_{x,y,t}$.

Suppose now $y(\cd) \in E_{x,y,t}$ is a minimizer for ${\mathcal F}[\cd]$.  Then the first variation of ${\mathcal F}[\cd]$ about $y(\cd)$ must be 0, whence
\be \label{L2}
\int^T_t \left[ \frac{d\varphi(s)}{ds} - \frac{\pa b}{\pa y}(y(s),s) \ \varphi(s)\right] 
\left[ \frac{dy(s)}{ds} - b(y(s),s)\right]ds = 0,
\ee
provided $\vp(\cd)$ is a $C^1$ function satisfying $\vp(t) = 0$, $\vp(T) = 0$.  Setting
\[    \vp(s) = \psi(s) \exp \left[ \int^s_t \frac{\pa b}{\pa y} \big( y(s'), s'\big) ds' \right] = \psi(s) V(s),  \]
it follows from (\ref{L2}) that 
\be \label{M2}
\int^T_t \ \frac{d\psi}{ds} \left[ \frac{dy}{ds} - b( y(s), s) \right] V(s)ds = 0,
\ee
for all $C^1$ functions $\psi : [t, T] \ra \R$ with $ \psi(t) = \psi(T) = 0$.  Equation (\ref{M2}) implies that
\be \label{X2}
\left[ \frac{dy}{ds} - b(y(s), s)\right] V(s) = \ {\rm constant}, \ \ t \le s \le T,
\ee
from which we may conclude that if $y < F(x,t)$ then $y'(s) >  b(y(s), s) $ for all $ s, \; t \le s \le T,$ and $y(\cd) $ is $C^1$.  It also follows that $y(T) = x $, for if $y(T) > x $ then there exists $t_1 < T$ such that if $y_1(s), \; t_1 \le s \le T $, satisfies $y_1(t_1) = y(t_1), y'_1(s)  = b(y_1(s), s)$, $t_1 \le s \le T$, then $y_1(T) > x$.  Evidently the function $y^*(s), \; t \le s \le T$, defined by $y^*(s) = y(s), \; t \le s \le t_1$, $y^*(s) = y_1(s),\ t_1 \le s \le T$, is in $E_{x,y,t}$ and satisfies ${\mathcal F} \Big[ y^*(\cd)\Big] < {\mathcal F}\Big[ y(\cd) \Big]$, yielding a contradiction.  One can argue in a similar way to prove the continuity of the function $q(x,y,t), (x,y) \in \R^2, \ t < T$. 
\end{proof}
 We have already observed that for $y \ge F(x,t)$ there is a unique minimizer $y(\cd)  \in E_{x,y,t}$ for the variational expression (\ref{B2}) and it is given by the solution $y(\cd)$ of equation (\ref{K1}) with initial condition $y(t) = y$.  For $y < F(x,t) $ we need to impose some condition on the function $b(\cd, \cd)$ beyond (\ref{A1}) to guarantee a unique minimizer.  To see what such a condition should be let us suppose that $y(s), \ t \le s \le T $, is a solution of the Euler-Lagrange equation (\ref{H2}) with initial conditions satisfying 
\be \label{N2}
y(t) = y, \ \ y'(t) > b(y,t).
\ee
Hence (\ref{H2}) implies that $y'(s) > b(y(s), s), \; t \le s \le T$.  Suppose now that $y(s) + \vp(s), \ t \le s \le T$, is also a solution to (\ref{H2}) with $\vp(t) = 0, \ \vp'(t) = \ve $.  Then to first order in $\ve $ the function $\vp(s), \; t \le s \le T $, satisfies the linear equation
\begin{multline} \label{O2}
\frac{d^2\vp}{ds^2} - \frac d{ds} \left[ \frac{\pa b}{\pa y}(y(s), s) \vp(s) \right] + \frac{\pa b}{\pa y} (y(s), s)\frac{d\vp(s)}{ds} \\
- \left[ \frac{\pa b}{\pa y} (y(s), s)\right]^2\; \vp(s) + \frac{\pa^2 b}{\pa y^2} (y(s), s) 
\left[ \frac{dy}{ds} - b(y(s), s)\right] \vp(s) = 0. 
\end{multline}
Suppose now that $\vp(\tau) = 0 $ for some $\tau, \; t < \tau \le T$.  Then on multiplying (\ref{O2}) by $\vp(s) $ and integrating over the interval $t \le s \le \tau $ we get
\begin{multline} \label{P2}
- \int^\tau_t \left[ \frac{d\vp(s)}{ds}\right]^2 \; ds + 2 \int^\tau_t \frac{\pa b}{\pa y} (y(s), s) \vp(s) \frac{d\vp(s)}{ds} ds \\
- \int^\tau_t \left[ \frac{\pa b}{\pa y}( y(s), s)\right]^2 \vp(s)^2 ds - \int^\tau_t V(s) \vp(s)^2 \;ds, 
\end{multline}
where $V(s) $ is given by the formula
\be \label{Q2}
V(s) = - \ \frac{\pa^2 b}{\pa y^2}( y(s), s) \left[ \frac{dy}{ds} - b(y(s), s)\right]. 
\ee 
Observe that by the Schwarz inequality we have 
\[  2 \int^\tau_t \frac{\pa b}{\pa y}( y(s), s)\vp(s) \frac{d\vp(s)}{ds}\; ds \le \int^\tau_t \left(\frac{d\vp(s)}{ds}\right)^2 
+ \int^\tau_t \left[ \frac{\pa b}{\pa y}( y(s), s) \right]^2 \; \vp(s)^2\; ds,   \]
with strict inequality in general.  Thus if $V(\cd)$ in (\ref{Q2}) is non-negative the expression (\ref{P2}) is strictly negative in general.  Since $V(\cd) $ is non-negative if the function $b(y,s) $ is concave in $y $, it appears that one gets a contradiction to the fact that (\ref{P2}) is zero when one assumes that $b(y,s)$ is concave in $y, \; t \le s \le T$.  We conclude therefore that the trajectories $y(\cd)$ of the Euler-Lagrange equation (\ref{H2}) which satisfy (\ref{N2}) are non intersecting.  In particular, for $y < F(x,t) $ there is exactly one which has the property that $y(t)=y, \ y(T) = x $ .  We make this argument rigorous in the following:

\begin{proposition}  Assume the function $b(\cd, \cd) $ satisfies (\ref{A1}) and that $b(y,s) $ is concave in $y$ for $y \in \R, \ s \in [t, T] $.  Then the minimizer $y(\cd) \in E_{x,y,t}$ of the variational problem (\ref{B2}) is unique for all $(x,y) \in \R^2$.  Furthermore the function $q(x,y,t)$ of (\ref{B2}) is $C^1$ for $(x,y) \in \R^2$, $t < T$.
\end{proposition}
\begin{proof} Since the minimizer is clearly unique for $y \ge F(x,t)$ we assume  $y < F(x,t)$.  We show that the functional ${\mathcal F}[\cd] $ of (\ref{R2}) has a convexity property provided $b(y,s)$ is concave in $y,\; t \le s \le T$. Let $E$ be the set of $C^1$functions $y(\cd)$ on $[t,T]$ which satisfy $y'(s) \ge b(y(s), s)$, $t \le s \le T$.  It is evident that $E$ is convex in $y$ for $t \le s \le T$, in the following sense:
\be \label{S2}
y_1(\cd), \  y_2(\cd), \  \la y_1(\cd) + ( 1 - \la) y_2(\cd) \in E, \quad  0 \le \la \le 1,
\ee
implies
\[   {\mathcal F} \left[ \la y_1(\cd) + ( 1 - \la) y_2(\cd)\right] \le \la {\mathcal F}  \left[ y_1(\cd)\right] + ( 1 - \la){\mathcal F}\left[ y_2(\cd) \right].  \]
To prove (\ref{S2}) we write
\begin{multline*}
{\mathcal F} \left[ \la y_1(\cd) + ( 1 - \la) y_2(\cd)\right]  
= \frac 1 2 \int^T_t \bigg[ \la \left\{ \frac{dy_1}{ds} - b(y_1(s), s)\right\}  + (1-\la) \left\{ \frac{dy_2}{ds} - b(y_2(s), s)\right\}   \\
- \left\{ b( \la y_1(s) + ( 1 - \la) y_2(s),s ) - \la  b(y_1(s), s) - (1-\la)  b(y_2(s), s) \right\} \bigg]^2 ds.
\end{multline*}
Since $y_1(\cd) \ y_2(\cd) \in E$ and $b(y,s)$ is concave in $y$, $t \le s \le T$, each term in the last expression inside curly braces is non-negative.  Assuming also that $\la y_1(\cd) + ( 1 - \la) y_2(\cd) \in E$ we have that
\begin{multline*}\
0 \le b( \la y_1(s) + ( 1 - \la) y_2(s),s) - \la  b(y_1(s), s) - (1-\la)  b(y_2(s), s) \\
\le 2\left[ \la \left\{ \frac{dy_1}{ds} - b(y_1(s), s)\right\}  + (1-\la) \left\{ \frac{dy_2}{ds} - b(y_2(s), s)\right\}\right], \quad t \le s \le T.
\end{multline*}
We conclude therefore that 
\begin{multline*}
{\mathcal F} \left[ \la y_1(\cd) + ( 1 - \la) y_2(\cd)\right] 
\le \frac 1 2 \int^T_t \left[ \la \left\{ \frac{dy_1}{ds} - b(y(s), s)\right\}  + (1-\la) \left\{ \frac{dy_2}{ds} - b(y(s), s)\right\} \right]^2 ds  \\
\le \la \;  {\mathcal F} [ y_1(\cd)] + ( 1 - \la) {\mathcal F}[y_2(\cd)] ,
\end{multline*}
and hence (\ref{S2}) holds.

The uniqueness of the minimizer $y(\cd) \in E_{x,y,t}$ follows from the strict convexity of ${\mathcal F}[\cd]$ in the sense of (\ref{S2}).  Let us assume $y_1(\cd) \ y_2(\cd) \in E_{x,y,t}$ are two minimizers where $y < F(x,t)$.  Then by Proposition 2.1 the functions $y_1(\cd), \  y_2(\cd)$ are in the set $E$ and for sufficiently small $\la > 0$ the function $\la y_1(\cd) + (1-\la) y_2(\cd)$ is also in $E$, whence  (\ref{S2}) implies that $\la y_1(\cd) + (1-\la) y_2(\cd)$ is a minimizer.  From the strict convexity of ${\mathcal F}[\cd]$ we have then  that
\[	\frac{dy_1}{ds} - b(y_1(s), s) = \frac{dy_2}{ds} - b(y_2(s), s), \quad t \le s \le T.   \]
Since $y_1(t) = y_2(t) = y$ we conclude from this last identity that $y_1(s) = y_2(s), \ t \le s \le T$, and so the uniqueness of the minimizer.  

To show that the function $q(x,y,t)$ is $C^1$ we consider the optimal control $\la^*(x,y,t) = y'(t)$ where $y(\cd) \in E_{x,y,t}$ is the unique minimizer for the variational problem (\ref{B2}).  Evidently $\la^*(x,y,t) = b(y,t)$ if $y \ge F(x,t)$.  We first prove that $\la^*(x,y,t)$ is continuous in $(x,y,t)$ for $(x,y) \in \R^2, \ t < T$.  To do this let $D_{x,y}(\del) \subset \R^2$ be the disc of radius $\del > 0$ centered at $(x,y)$.  Then there exists a constant $K(\del) > 0$ depending only on $\del$ such that
\be \label{T2}
\int^T_t \left[ \frac{dz(s)}{ds} \right]^2 \le K(\del), \quad  z(\cd) \in E_{x',y',t} \ , \ (x',y') \in D_{x,y}(\del),	
\ee	
where $z(\cd)$ is the minimizer of the variational problem.  To see (\ref{T2}) observe that
\[	{\mathcal F}[z(\cd)] \ge \frac 1 4 \int^T_t \left[ \frac{dz(s)}{ds} \right]^2 ds - \frac 1 2 \int^T_t b( z(s), s)^2 ds.  \]
Now from (\ref{A1}) one has that 
\[    |b( z(s), s)| \le | b( z(t), s)| + A \ \int^T_t \left| \frac{dz}{ds'}\right| ds', \quad  t \le s \le T.	\]
Hence from the Schwarz inequality we have that 
\[	{\mathcal F}[z(\cd)] \ge \frac 1 8 \int^T_t \left[ \frac{dz(s)}{ds} \right]^2 ds - K'(\del),	\]
where $K'(\del)$ is a constant depending on $\del$.  Now (\ref{T2}) follows from this last inequality and the continuity of the function $q(\cd, \cd, t)$ on $D_{x,y}(\del)$.

Next we show that for any $\ve > 0$ there exists $\del > 0$ such that
\be \label{U2}
\int^T_t \left[ \frac{dy}{ds} - \frac{dz}{ds} \right]^2 ds < \ve, \quad z(\cd) \in E_{x',y'}, \ (x',y') \in D_{x,y}(\del),
\ee
where $y(\cd) \in E_{x,y,t}$ is the minimizer for (\ref{B2}) and $z(\cd) \in E_{x',y',t}$ is also the minimizer.  The inequality (\ref{U2}) follows from the convexity (\ref{S2}) of the functional ${\mathcal F}[\cd]$.  We first consider the situation $y \ge F(x,t)$, where the minimizer $y(\cd) \in E_{x,y,t}$ satisfies $y'(s)=b( y(s), s)$ and $q(x,y,t) = 0$.  Thus for $\ve_1 >0$ there exists $\del_1 > 0$ and
\be \label{V2}
{\mathcal F}[z(\cd)] < \ve_1, \quad z(\cd) \in E_{x',y',t}, \ \ (x',y') \in D_{x,y}(\del_1).
\ee
We can restate (\ref{V2}) as $z(\cd)$ satisfies the initial value problem
\[	\frac{dz}{ds} = b( z(s), s) + f(s), \quad  t \le s \le T, \ \ z(t) = y'\; ,\]
where $\| f\|_2 < \sqrt{2\ve_1}$.  Putting now $\vp(s) = z(s) - y(s)$ it follows from (\ref{A1})  that $\vp(s)$ satisfies the initial value problem 
\be \label{WW2}
	\frac{d\vp}{ds} = a(s) \vp(s) + f(s), \quad  t \le s \le T, \  \vp(t) = y' - y\; , 
\ee
where $\displaystyle{\sup_{t\le s \le T}} |a(s)| \le A$.  It follows that there are constants $C_1,C_2 > 0$ such that  
\be \label{W2}
\sup_{t\le s \le T} |z(s) - y(s)| \le C_1 |y' - y| + C_2\sqrt{\ve_1}.
\ee 
We write the LHS of (\ref{U2}) as
\[\int^T_t \left\{ \left[ b( y(s), s) - b( z(s), s) \right] + \left[ b( z(s), s) - \frac{dz}{ds}\right]\right\}^2 ds\]
\[  \le 2 \int^T_t  \left[ b( y(s), s) - b( z(s), s) \right]^2 + 4 {\mathcal F}[z(\cd)].  \]
The inequality (\ref{U2}) follows from this last inequality and (\ref{V2}), (\ref{W2}).

We prove (\ref{U2}) for $y < F(x,t)$.  First let $\del_1 > 0$ be such that closure of $D_{x,y}(\del_1)$ lies in the set $\{(x',y') \in \R^2 : y' < F(x',t)\}$.  Then it follows from (\ref{T2}) that there exists $\la_0, \ 0 < \la_0 < 1$, such that
\[	\la_0 z(\cd) + (1 - \la_0)y(\cd) \in E, \quad z(\cd) \in E_{x',y',t},  \ (x',y') \in D_{x,y}(\del_1), \]
where $y(\cd) \in E_{x,y,t} $ and $z(\cd) \in E_{x',y',t} $ are the minimizers for (\ref{B2}).  Since $z(\cd)$ and $y(\cd)$ are also in $E$ we may use the convexity (\ref{S2}) of the functional ${\mathcal F}[\cd]$.  In particular we have that
\begin{multline*}
 {\mathcal F} [\la_0 z(\cd) + (1 - \la_0)y(\cd)] \le \la_0 {\mathcal F}[z(\cd)] + (1 - \la_0) {\mathcal F}[y(\cd)]	\\\
\frac{-\la_0(1-\la_0)}{2} \int^T_t \left\{ \frac{dy}{ds} - \frac{dz}{ds} + b( z(s), s) - b( y(s), s) \right\}^2 ds.
\end{multline*}
Using the continuity of the function $q(\cd, \cd, t)$ at $(x,y)$ we conclude from the last inequality that there exists $\del_2, \ 0 < \del_2 < \del_1$ such that
\be \label{VV2}
\frac 1 2 \int^T_t \left\{ \frac{dy}{ds} - \frac{dz}{ds} + b( z(s), s) - b( y(s), s) \right\}^2 ds < \ve_2, \quad z(\cd) \in E_{x',y',t} \ , \ (x',y') \in D_{x,y}(\del_2),
\ee
where again $y(\cd) \in E_{x,y,t} $ and $z(\cd) \in E_{x',y',t} $ are the minimizers for (\ref{B2}).  Here $\ve_2 > 0$ can be chosen arbitrarily and $\del_2$ depends on $\ve_2$.  Now we may argue as for the case when $y \ge F(x,t)$.  Thus letting $\vp(s) = z(s) - y(s)$ we have that $\vp(s)$ satisfies the equation (\ref{WW2}) with $\|f\|_2 < \sqrt{2\ve_2}$.  Hence we obtain an inequality analogous to (\ref{W2}), which together with (\ref{VV2}) implies (\ref{U2}).

The continuity of $\la^*(x,y,t)$ in $(x,y)$ follows easily from (\ref{U2}) upon using (\ref{X2}).  Thus for a minimizer of (\ref{B2}), $z(\cd) \in E_{x',y',t}$ one has
\be \label{Y2}
\frac{dz}{ds} - b(z(s),s) = A(x',y',t) \exp \left[ -\int^s_t \frac{\pa b}{\pa y}( z(s'), s')ds' \right], \ t \le s \le T.
\ee
where $\la^*(x',y',t) = b(y',t) + A(x',y',t)$.  Evidently (\ref{U2}) implies that the function $A(\cd, \cd, t)$ is continuous at $(x,y)$.  Finally we observe that the continuity of $\la^*(x,y,t)$ as a function of $(x,y,t)$ for $(x,y) \in \R^2, \ t < T$, follows from (\ref{Y2}).  In fact if $y(\cd) \in E_{x,y,t}$ is the minimizer for (\ref{B2}) then (\ref{Y2}) implies that for fixed $x$ the function $s \ra \la^*(x,y(s), s)$ is continuous, $t \le s < T$.  Hence if we combine this with the previous argument on the continuity of $\la^*(\cd, \cd, t)$ for fixed $t$ we obtain the continuity of $\la^*(\cd,\cd,\cd)$ in all three variables.

We prove the $C^1$ property of the function $q(x,y,t), (x,y) \in \R^2, \ t < T$.  First we observe that there is differentiability of the function $q$ in a least one direction.  Thus
\be \label{AA2}
- \frac d{ds} \; q( x,y(s), s)\Big|_{s=t}  = \frac 1 2 \left[ \la^*(x,y,t) - b(y,t) \right]^2,
\ee
where $y(\cd) \in E_{x,y,t}$ is the minimizer for (\ref{B2}).  We use the continuity of the function $\la^*(\cd, \cd, \cd)$ to show differentiability in other directions.  Let us assume that $y < F(x,t)$ and $\De y$ small enough so that $|\De y| < F(x,t) - y$.  Then
\begin{multline} \label{AB2}
 q(x,y + \De y, t) - q(x,y,t)  \le \ \  - \frac 1 2 \int^T_t \left[ \la^*(s) - b( y(s), s) \right]^2 ds \\ 
+  \frac 1 2 \int^T_t\left[ \la^*(s) - \De y/(T-t) - b( y(s) + (T-s)\De y/(T-t), s) \right]^2 ds,
\end{multline}
where $y(\cd) \in E_{x,y,t}$ is the minimizer for (\ref{B2}) and $\la^*(s) = y'(s), \ t \le s \le T$.  Letting $\De y \ra 0$ in (\ref{AB2}) we conclude that 
\begin{multline}  \label{AC2}
\limsup_{\De y \ra 0} \left[ q(x,y + \De y, t) - q(x,y,t)\right] \Big/  \De y \le \\
- \frac 1{T-t} \int^T_t \left[ 1 + (T-s) \frac{\pa b}{\pa y}( y(s), s ) \right] \left[ \la^*(s) - b( y(s), s)\right]ds.  
\end{multline}
Alternatively let $y_\De(\cd) \in E_{x,y +\De y,t}$ be the minimizer for (\ref{B2}) and $\la^*_\De(s) = y'_\De(s), \ t \le s \le T$.  Then one also has
\begin{multline}  \label{AD2}
q(x,y + \De y, t) - q(x,y,t) \ge \frac 1 2 \int^T_t \left[ \la^*_\De(s) - b( y_\De(s), s)\right]^2\;ds \\
- \frac1 2 \int^T_t \left[ \la^*_\De(s) + \De y/ (T-t) - b( y_\De(s) - (T-s)\De y/ (T-t), s )\right]^2\;ds.  
\end{multline}
It follows from (\ref{AD2}) by using (\ref{U2}), (\ref{Y2}) and the continuity of the function $\la^*(\cd, \cd, \cd)$ that 
\begin{multline}  \label{AE2}
\liminf_{\De y \ra 0} \left[ q(x,y + \De y, t) - q(x,y,t)\right] \Big/ \De y \ge \\
- \frac 1{T-t} \int^T_t \left[ 1 + (T-s) \frac{\pa b}{\pa y}( y(s), s) \right] \left[ \la^*(s) - b( y(s), s)\right]ds.  
\end{multline}
The differentiability of $q(x,y,t)$ w.r. to $y$ follows from (\ref{AC2}), (\ref{AE2}).  Using (\ref{U2}), (\ref{Y2}) again we also see from the formula on the RHS of (\ref{AC2}) that $\pa q(x,y,t)/\pa y$ is continuous in $(x,y,t)$ for $y < F(x,t), \ t<T$.  It is easy to extend this argument to show that $\pa q(x,y,t)/\pa y$ exists for all $y \in \R$ and the derivative is continuous in $(x,y,t)$ for $(x,y) \in \R^2, \; t < T$.  This follows from the fact that the formula on the RHS of (\ref{AC2})  is zero if $y = F(x,t).$

One can see by a similar argument that $q(x,y,t)$ is differentiable w.r. to $x$ and that $\pa q(x,y,t)/\pa x$ is continuous for $(x,y) \in \R^2, \; t < T$.  Finally (\ref{AA2}) and the fact that $\pa q(x,y,t)/\pa y$ is continuous shows that $q(x,y,t)$ is differentiable w.r. to $t$ and $\pa q(x,y,t)/\pa t$  is continuous in $(x,y,t)$, $(x,y) \in \R^2, \; t < T$.  We have shown that the function $q(x,y,t)$ is $C^1$ for $(x,y) \in \R^2, \; t < T$. 
\end{proof}
\begin{corollary}  Assume $b(\cd, \cd)$ satisfies the conditions of Proposition 2.2, $q(x,y,t)$ is the function defined by (\ref{B2}), and $\la^*(x,y,t)$ is the corresponding optimal control, $(x,y) \in \R^2$, $t < T$.  Then there are the identities,
\begin{eqnarray} \label{AF2}
\pa q(x,y,t)/\pa y &=& b(y,t) - \la^*(x,y,t), \\
\pa q(x,y,t)/\pa t &=& \frac 1 2 \left[ \la^*(x,y,t)^2 - b(y,t)^2 \right] .\nn
\end{eqnarray}
Furthermore, for $y < F(x,t)$ there are the inequalities 
\be \label{AG2}
\frac {\pa q(x,y,t)}{\pa y} < 0, \quad  \frac {\pa q(x,y,t)}{\pa x} > 0.
\ee
\end{corollary}
\begin{proof} We first show the identity (\ref{AF2}) for $\pa q(x,y,t)/\pa y$.  We assume $y < F(x,t)$ since it is obvious otherwise.  Using the fact that $q(x,y,t)$ is the minimizer for the variational problem (\ref{B2}) we have that for $\la \in \R$,
\[	q(x,y,t) \le \frac 1 2 \left[ \la - b(y,t)\right]^2 \De t + q(x,y + \la \De t, t + \De t) + O\left[ (\De t)^2\right].  \]
Since $q$ is $C^1$ this implies that
\be \label{AH2}
\frac 1 2 \left[ \la^*(x,y,t)^2 - b(y,t)\right]^2 \le \frac 1 2 \left[ \la - b(y,t) \right]^2 + \left[ \la - \la^*(x,y,t)\right] \frac{\pa q}{\pa y} (x,y,t), \ \la \in \R,
\ee
where we have used (\ref{AA2}).  The inequality (\ref{AH2}) implies the first identity of (\ref{AF2}).  The second identity follows from the first identity and (\ref{AA2}).

The first inequality of (\ref{AG2}) follows from Proposition 2.1.  To show that $\pa q(x,y,t)/\pa x > 0$ we derive a formula for $\pa q(x,y,t)/\pa x$ similar to the formula (\ref{AF2}) for $\pa q(x,y,t)/\pa y$.  We have already seen that $\pa q(x,y,t)/\pa x$ is given by an expression similar to the RHS of (\ref{AC2}),
\be \label{AI2}
\frac {\pa q}{\pa x}(x,y,t) = \frac 1{T-t} \int^T_t \left[ 1 - (s-t) \frac{\pa b}{\pa y}( y(s), s) \right] \left[ \la^*(s) - b( y(s), s)\right]ds. 
\ee
Adding (\ref{AC2}) and (\ref{AI2}) we conclude that
\be \label{AW2}
\frac {\pa q}{\pa y}(x,y,t) + \frac {\pa q}{\pa x}(x,y,t)= - \int^T_t \frac {\pa b}{\pa y}(y(s)s) \left[ \la^*(s) - b( y(s), s)\right]ds. 
\ee
If  we use now the identity (\ref{Y2}) we conclude from the previous expression that
\be \label{AJ2}
\frac {\pa q}{\pa x}(x,y,t)=\left[\la^*(x,y,t) -b(y,t)\right] \exp \left[ - \int^T_t \frac {\pa b}{\pa y}(y(s)s) ds \right],
\ee
where $y(\cd) \in E_{x,y,t}$ is the minimizer for (\ref{B2}).  Proposition 2.1 and (\ref{AJ2}) now imply 
$\pa q(x,y,t)/\pa x > 0$. 
\end{proof}
\begin{rem} Observe that Proposition 2.2 and Corollary 2.1 imply that $q(x,y,t)$ is a classical solution to the $\ve=0$ Hamilton-Jacobi equation (\ref{H1}).
\end{rem}
Next we show that $q(x,y,t)$ is twice differentiable in $(x,y)$.  Since this is obvious for $y > F(x,t)$ we consider $y < F(x,t)$.  Let $\vp(s)$, $t \le s \le T$, be the solution of the first variation equation (\ref{O2}) with terminal data $\vp(T) = 0$, $\vp '(T) = -1$.  Then one should have the identity
\be \label{AK2}
\pa\la^*(x,y,t) /\pa y \ \ = \ \ \vp'(t)/\vp(t).
\ee
We have already given an argument to show $\vp(s) > 0$, $t\le s \le T$, if we assume $b(\cd, s)$ is concave for $t\le s \le T$.  Hence in this case the RHS of (\ref{AK2}) makes sense.   Note also that we may write (\ref{O2}) in the form
\be \label{AL2}
\left[ \frac d{ds} + \frac{\pa b}{\pa y}(y(s),s) \right] \left[ \frac {d\vp}{ds} - \frac{\pa b}{\pa y}(y(s),s)\vp(s) \right] - V(s) \vp(s) = 0,
\ee
where $V(s) \ge 0$ if $b(\cd,s)$ is concave for all $s, \; t\le s \le T$.  Hence it follows from (\ref{AL2}) that if we assume the concavity of $b(\cd,s) \; t\le s \le T$, then $\vp'(t) - \pa b/\pa y(y(t,), t)\vp(t) < 0$.  Thus from (\ref{AF2}), (\ref{AK2}) we conclude that $\pa^2q(x,y,t)/\pa y^2 > 0$.  We make this argument rigorous in the following:

\begin{proposition} Assume the function $b(\cd, \cd)$ satisfies (\ref{A1}) and that $b(y,s)$ is concave in $y$ for $y \in \R, s \le T$.  Then the function $q(x,y,t)$ of (\ref{B2}) is convex in $(x,y)$ for $(x,y) \in \R^2, \ t<T$.  Suppose in addition that  $b(y,s)$ is twice differentiable in $y$ for $y \in \R, s \le T$, and $\pa^2b(y,s)/\pa y^2$ is continuous in $(y,s)$.  Then $q(x,y,t)$ is twice differentiable in $(x,y)$ for $(x,y,t) \in U_T =$ $\{(x,y,t) : (x,y) \in \R^2, t < T, \; y < F(x,t)\}$.  The second derivatives of $q(x,y,t)$ w.r. to $(x,y)$ are continuous in $U_{T}$ and satisfy $\pa^2 q(x,y,t)/\pa x^2 > 0, \pa^2 q(x,y,t)/\pa y^2 > 0$, $\pa^2 q(x,y,t)/\pa x \pa y < 0$.  Furthermore, if $(x_0, y_0, t_0) \in \pa U_{T}$ and $t_0 <T$ then
\be \label{AM2}\
\lim_{(x,y,t) \ra (x_0, y_0, t_0)} \pa^2q(x,y,t)/\pa x^2 > 0, \quad \lim_{(x,y,t) \ra (x_0, y_0, t_0)} \pa^2q(x,y,t)/\pa y^2 > 0.
\ee
\end{proposition}
\begin{proof} Observe that the function $F(x,t),  \ x \in \R, \ t<T$, defined by (\ref{K1}) is a convex function of $x$.  In fact one has 
\be \label{AN2}
\frac{\pa F}{\pa x} (x,t)= \exp \left[ - \int^T_t \frac{\pa b}{\pa y}(y(s),s) ds \right],
\ee
where $y(s), s \le T$, is the solution to (\ref{K1}).  Hence by concavity of $b(\cd, s), \  t \le s \le T$,  one has that $\pa F(x,t)/\pa x$ is an increasing function of $x$.  It follows that the set $V_t = \{(x,y) \in \R^2 : y \ge F(x,t)\}$ on which  $q(\cd, \cd, t)$ vanishes is convex.  We also have from the argument of Proposition 2.2 that $q(x,y,t)$ is locally convex on the not necessarily convex open set $\R^2 \backslash V_t$.  Hence $q(x,y,t)$ is convex in $(x,y)$ for all $(x,y) \in \R^2$.

We assume now $b(y,s)$ is twice continuously differentiable in $y$ for $y \in \R,  \ s \le T$.  We can write (\ref{AL2}) as a system
\begin{eqnarray} \label{AO2}
\frac{d\vp}{ds} &-& \frac{\pa b}{\pa y} \big(y(s), s\big) \vp(s) = -\psi(s),  \quad  t \le s \le T, \\
\frac{d\psi}{ds} &+& \frac{\pa b}{\pa y} \big(y(s), s\big) \psi(s) = -V(s)\vp(s), \quad  t \le s \le T, \nn
\end{eqnarray}
where $y(\cd) \in E_{x,y,t}$ is the minimizer for (\ref{B2}).  Evidently (\ref{AO2}) has a unique solution $[\vp(s), \psi(s)],   \ t \le s \le T$, with terminal data $\vp(T) = 0$, $\psi(T) = 1$.  Multiplying the first equation in (\ref{AO2}) by $\psi(s)$ and the second by $\vp(s)$ we see on integration that 
\be \label{AP2}
\psi(s) \vp(s) = \int^T_s  \psi(s')^2 + V(s')\vp(s')^2 \ ds', \quad  t \le s \le T.
\ee
From the terminal conditions on $[\vp(s), \psi(s)]$ we have that $\vp(s) > 0, \; \psi(s) >0$ for $s$ close to $T$.  It follows then from (\ref{AP2}) that $\vp(s) > 0, \; \psi(s) > 0$ for $t \le s \le T$.

Next we use (\ref{Y2}) to write the equation for the minimizer $y(\cd) \in E_{x,y,t}$ of (\ref{B2}) in a form similar to (\ref{AO2}).  Thus we have
\begin{eqnarray}    \label{AQ2} 
\frac{dy}{ds} &-& b( y(s), s) = -p(s), \quad  t \le s \le T, \\
\frac{dp}{ds} &+& \frac{\pa b}{\pa y}( y(s), s) p(s) = 0, \quad  t \le s \le T. \nn
\end{eqnarray}
In (\ref{AQ2}) the first equation is the definition of the Hamiltonian momentum $p(s)$ while the second equation is equivalent to (\ref{Y2}).  Suppose now $z(\cd) \in E_{x',y',t}$ is also a minimizer for (\ref{B2}) and define $\Phi(s) = z(s) - y(s)$, $\Psi(s) = P(s) - p(s)$, where $P(s)$ is the momentum corresponding to $z(\cd)$.  Then since $z(\cd)$ satisfies an equation similar to $(\ref{AQ2})$ we have that
\begin{eqnarray} \label{AR2} 
 \frac{d\Phi}{ds} -\Phi(s) \int^1_0 \frac{\pa b}{\pa y}( \mu y(s) + (1-\mu)z(s),s) \ d\mu = -\Psi(s), &\quad&  t \le s \le T, \\
\frac{d\Psi}{ds} + \frac{\pa b}{\pa y}(z(s),s) \Psi(s) = - \Phi(s)p(s) \int^1_0 \frac{\pa^2 b}{\pa y^2}( \mu y(s) + (1-\mu)z(s),s) \ d\mu,  &\quad&  t \le s \le T. \nn
\end{eqnarray}
We consider now the situation where $x' = x$ so $\Phi(T) =0$.  Then if $y' = y + \De y$ we may write
\be \label{AS2}
\Phi(t) = \al(\De y)\Psi(T), \ \Psi(t) = \beta(\De y) \Psi(T),
\ee
where the functions $\al(\cd)$ and $\beta(\cd)$ satisfy
\be \label{AT2}
\lim_{\De y \ra 0} \al(\De y) = \vp(t), \ \ \lim_{\De y \ra 0} \beta(\De y) = \psi(t)
\ee
 since the coefficients in the equations (\ref{AR2}) converge as $\De y \ra 0$ to the coefficients in the equations (\ref{AO2}).  Now we have that
\begin{multline*}
\left[ \la^*(x,y + \De y, t) - \la^*(x,y,t) \right] \;\Big/ \De y \\
=  \left[ \Phi(t) \int^1_0 \frac{\pa b}{\pa y}( \mu y(s) + (1-\mu)z(s),s) \ d\mu - \Psi(t)\right] \; \Big/\De y,
\end{multline*}
and $\Phi(t) = \De y$.  Hence it follows from (\ref{AS2}), (\ref{AT2}) that $\la^*(x,y,t)$ is differentiable w.r. to $y$ and
\be \label{AU2}
\pa\la^*(x,y,t)/\pa y = \pa b(y,t)/\pa y - \psi(t)/\vp(t).
\ee
One also sees easily from the representation (\ref{AU2}) that $\pa\la^*(x,y,t)/\pa y$ is continuous in $U_T$ and that the limit exists as $(x,y,t) \ra (x_0,y_0,t_0) \in \pa U_T$ provided $t_0 < T$.  The fact that $\pa^2q(x,y,t) /\pa y^2 > 0$ follows now from (\ref{AF2}) and the fact that $\psi(t) > 0, \ \vp(t) > 0$.

We can similarly see that $\la^*(x,y,t)$ is differentiable w.r. to $x$ and $\pa \la^*(x,y,t)/\pa x$ is continuous in $U_T$ and the limit exists as $(x,y,t) \ra (x_0,y_0,t_0) \in \pa U_T$ provided  $t_0 < T$.  To see that $\pa^2q(x,y,t) /\pa x \pa y < 0$ we note that $\pa^2q(x,y,t) /\pa x \pa y = \psi(t)/\vp(T)$, where $[\vp(s), \psi(s)], \  t \le s \le T$, is the solution of (\ref{AO2}) with initial data $\vp(t) = 0, \ \psi(t) = 1$.  We have in this case
\[	\psi(s) \vp(s) = - \int^s_t \psi(s')^2 + V(s') \vp(s')^2 ds',	\]
whence $\vp(T) < 0$ and so $\pa^2 q(x,y,t)/\pa x \pa y$ is negative.

To prove the twice differentiability of $q(x,y,t)$ w.r. to $x$ we use the representation 
\be \label{AV2}
\pa q(x,y,t)/\pa x \ = \ p(T),
\ee
where $p(s)$ is given by (\ref{AQ2}) for the minimizer $y(\cd) \in E_{x,y,t}$ of (\ref{B2}).  The differentiability of $\pa q(x,y,t)/\pa x$ and the positivity of $\pa^2 q(x,y,t)/\pa x^2$ proceeds as before by representing $\pa^2 q(x,y,t)/\pa x^2$ in terms of a solution to (\ref{AO2}).  Finally we observe that (\ref{AV2}) follows from (\ref{Y2}), (\ref{AF2}) and (\ref{AW2}).
\end{proof}
\begin{rem}  Proposition 2.3 shows that all second derivatives of $q(x,y,t)$ with respect to $(x,y)$ have jump discontinuities across the boundary $y = F(x,t)$.  Hence $q(x,y,t)$ is not $C^2$ in $(x,y)$ for all $(x,y) \in \R^2$.
\end{rem}

\section{Proof of Theorem 1.1}
Our main goal in this section is to show that the function $q_\ve(x,y,t)$ defined by (\ref{G1}) converges as $\ve \ra 0$ to the function $q(x,y,t)$ defined by (\ref{J1}).  The formula (\ref{O1}) for $q_\ve(x,y,t)$ makes this intuitively clear, but it is not obvious under what circumstances the function defined by (\ref{G1}) has the representation (\ref{O1}).  As part of our proof of convergence we shall make use of various situations in which (\ref{O1}) is valid.  First we regularize the terminal data (\ref{I1}).

\begin{lem}  Suppose $b(\cd,\cd)$ satisfies (\ref{A1}) and $q_\ve(x,y,t)$ is given by (\ref{G1}).  Then there exists $\del > 0$ and universal constants $C_1, C_2 > 0$ such that if $T - t < \del, \ \ve < 1$, there is the inequality
\be \label{H3}
C_1(x-y)^2/(T-t) < q_\ve(x,y,t) < C_2(x-y)^2/(T-t),
\ee
for $y$ in the region
\be \label{I3}
x - y > 2 \; \int^T_t |b(x,s)| ds + \sqrt{\ve(T-t)}.
\ee
\end{lem}
\begin{proof}   Since $b(\cd, \cd)$ satisfies (\ref{A1}) one can uniquely solve the stochastic equation (\ref{E1}) with given initial data. The solution  $u_\ve(x,y,t)$ of the terminal value problem (\ref{B1}), 
(\ref{C1})  is then given by the formula (\ref{F1}).
Letting $Z_\ve(s) = Y_\ve(s) - y$, we have then that
\begin{multline} \label{J3}
Z_\ve(s) = \int^s_t   \left[ \int^1_0 d\mu \; \frac{\pa b}{\pa y}\big( \mu Y_\ve(s') + (1-\mu)y, s' \big) \right] 
Z_\ve(s') ds' \\
+ \int^s_t  b(y,s') ds' + \sqrt{\ve} \big[ W(s) - W(t) \big], \quad  s > t. 
\end{multline}
Now applying Gronwall's inequality to (\ref{J3}) we conclude that 
\be \label{K3}
\sup_{t\le s\le T}|Z_\ve(s)| \le A(t, T) \sup_{t\le s \le T} \Big|  \int^s_t  b(y,s') ds' + \sqrt{\ve} \big[ W(s)-W(t) \big] \Big|,
\ee
where $A(t,T)$ is a constant depending only on $t,T$.  The lower bound in inequality (\ref{H3}) follows from (\ref{K3}) and (\ref{F1}), (\ref{G1}). 

To obtain the upper bound we consider the stochastic process $Z'_\ve(s), \ s\ge t$, defined by the equation
\be \label{ZA3}
dZ'_\ve(s)=\left[AZ'_\ve(s)+b(x,s)\right]ds+\sqrt{\ve}dW(s), \quad Z'_\ve(t)=x-y,
\ee
where $A$ is the constant in (\ref{A1}). If $\tau$ is the first hitting time at $x$ for the process $Y_\ve(s)$ of (\ref{E1}) with $Y_\ve(t)=y$, then it is evident that $Z'_\ve(s)\ge x-Y_\ve(s), \ t\le s\le \tau$. It follows that
\be \label{ZB3}
P(\tau<T) \ge P(Z'_\ve(T)<0 \  \big| \ Z'_\ve(t)=x-y).
\ee
Since the stochastic equation (\ref{ZA3}) is exactly solvable, we can estimate the RHS of (\ref{ZB3}). Assuming $x-y$ satisfies  (\ref{I3}) we conclude that
\be \label{ZC3}
P(Z'_\ve(T)<0 \  \big| \ Z'_\ve(t)=x-y) \ge \exp\left[-C(x-y)^2/\ve(T-t)\right],
\ee
for a constant $C$ depending only on the parameter $A$ in (\ref{ZA3}). The upper bound in (\ref{H3}) follows now from (\ref{ZB3}), (\ref{ZC3}), and the inequality
\be \label{ZD3}
P(Y_\ve(T)>x \ \big| \ Y_\ve(t)=y)\ge P(\tau<T) \inf_{t\le s\le T} P( Y_\ve(T)>x \ \big| \ Y_\ve(s)=x),
\ee
since it is clear that for $\delta$ small enough the infimum in (\ref{ZD3}) is larger than $1/4$.
\end{proof}
We consider a controller $\la_\ve(y,s), \ y \in \R, \ s < T$, which is uniformly Lipschitz in $y$ for $t \le s \le T-\del$.  Thus there is a constant $C$ such that
\be \label{M3}
|\la_\ve(y,s) - \la_\ve(y', s) | \le C|y-y'|, \quad  y, y' \in \R, \ \ t \le s \le T-\del.
\ee
Hence we may solve the stochastic differential equation (\ref{N1}) for $t \le s \le T - \del$.  We show that in this case the the expectation on the RHS of (\ref{O1}) is bounded below by the LHS.
\begin{lem}  Suppose $\la_\ve(\cd, \cd)$ satisfies (\ref{M3}) and $b(\cd, \cd)$ satisfies (\ref{A1}).  Then if $ q_\ve(x,y,t) $ is given by (\ref{G1}) there is the inequality
\begin{multline}  \label{N3}
 q_\ve(x,y,t) \le E \bigg\{ \frac 1 2 \int^{T-\del}_t \left[ \la_\ve(y_\ve(s),s) - b( y_\ve(s), s) \right]^2 ds  \\
 + q_\ve( x, y_\ve(T-\del), T-\del)  \ \Big|  \ y_\ve(t) = y \bigg\},  
\end{multline}
where $y_\ve(\cdot)$ is the solution to the SDE (\ref{N1}).
\end{lem}
\begin{proof}  Let $V_\ve(y,s), \ y \in \R, \ s \le T-\del$, denote the RHS of (\ref{N3}).  Arguing as in Lemma 3.1, one sees that
\be \label{O3}
0 < V_\ve(y,s)\le Ay^2 + B, \quad  y \in \R, \ t \le s \le T-\del,
\ee
for some constants $A,B$.  In addition $V_\ve, \pa V_\ve/\pa s$, $\pa V_\ve / \pa y$ and $\pa^2 V_\ve / \pa y^2$ are all continuous functions of $(y,s), \ y \in \R, \ t \le s < T-\del$, and satisfy the equation
\be \label{P3}
\frac {\pa V_\ve}{\pa s} + \la_\ve(y,s)\frac{\pa V_\ve}{\pa y} + \frac \ve 2 \; \frac {\pa^2 V_\ve}{\pa y^2} + \frac 1 2 \big[ \la(y,s) - b(y,s) \big]^2 = 0, \quad   y \in R, \ t \le s < T-\del,
\ee
with terminal condition 
\be \label{Q3}
V_\ve(y, T-\del) = q_\ve(x, y, T-\del), \quad y \in \R.
\ee
Note that the twice differentiability of $V_\ve(y,s)$ with respect to $y$ uses the fact that the function $\la_\ve(\cd, s) - b(\cd,s)$ is Lipschitz continuous for $t \le s \le T-\del$ (see \cite{fried} Chapter 1, Theorem 9).  From (\ref{H1}), (\ref{P3}) we conclude that the function $W_\ve(y,s)=V_\ve(y,s) - q_\ve(x,y,s)$ satisfies the PDE
\[	\frac {\pa W_\ve}{\pa s} + \la_\ve(y,s)\frac{\pa W_\ve}{\pa y}+ \frac \ve 2 \; \frac {\pa^2 W_\ve}{\pa y^2} + \frac 1 2 \left[ \la_\ve(y,s) - b(y,s) + \frac{\pa q_\ve}{\pa y} \right]^2 = 0, \quad  y \in R, \ t \le s \le T-\del,  \]
and all the derivatives $\pa W_\ve / \pa s, \ \pa W_\ve/\pa y, \  \pa^2 W_\ve/\pa y^2 $ are continuous.  Furthermore by (\ref{Q3}) the terminal condition for $W_\ve$ is $W_\ve(y, T-\del) = 0, \; y \in \R$.  It follows then from Lemma 3.1, (\ref{O3}) and the maximum principle (see \cite{fried} Chapter 2, Theorem 9) that $W_\ve(y,t) \ge 0, \ y \in \R$, whence the result follows. 
\end{proof}
\begin{lem} Suppose $b(\cd, \cd)$ satisfies (\ref{A1}).  Then for $x,y \in \R, \ t < T$, and $\ve < 1$, there is the inequality
\be \label{R3}
q_\ve(x,y,t) \le q(x,y,t) + C(x,y,t,T) \sqrt{\ve},
\ee
where $q(x,y,t)$ is given by (\ref{J1}) and $C(x,y,t,T)$ is a constant independent of $\ve$.
\end{lem}
\begin{proof}  Let $y(s), \  t \le s \le T$, be a minimizer for (\ref{J1}), whose existence has been  established by Proposition 2.1.  We set $\la_\ve(y,s) = \la(s) = y'(s), \ y \in \R, \ t \le s \le T$, and apply Lemma 3.2, taking $\del = \sqrt{\ve}$.  We consider first the case $y \le F(x,t)$ so $y(T) = x$.  Hence $x-y(T-\del) < C\sqrt{\ve}$ for some constant $C$.  It follows then from Lemma 3.1 that
\be \label{S3}
E \left\{ q_\ve\big( x,y_\ve(T-\del), T-\del \big) \  \Big|  \ y_\ve(t) = y \right\} \le C_1\sqrt{\ve}
\ee
for some constant $C_1$.  Here we are using the fact that $y_\ve(s) - y(s) = \sqrt{\ve} \; [W(s) - W(t)]$ and that $q_\ve(x,y,T-\del)$ is a decreasing positive function of $y \in \R$.  We can similarly see that
\be\label{T3}
E \left\{ \frac 1 2 \int^{T-\del}_t \left[ \la(s) - b(y_\ve(s),s) \right]^2 ds  \ \Big| \  y_\ve(t) = y \right\} \le q(x,y,t) + C_2\sqrt{\ve}.
\ee
for some constant $C_2$.  Thus (\ref{R3}) follows from (\ref{S3}) (\ref{T3}) in the case $y \le F(x,t)$.  For $y > F(x,t)$ we may use the same argument, noting that $q_\ve(x,\cd,T-\del)$ is a decreasing positive function. 
\end{proof}
To obtain a lower bound for $q_\ve(x,y,t)$ corresponding to the upper bound established in Lemma 3.3 we shall need to use the fact that the function $\pa q_\ve(x,y,s)/\pa y$ is uniformly Lipschitz continuous in $y$ for $(y,s)$ in any region $\{ (y,s) : y \ge y_0, \; t \le s \le T - \del\}$, where $\del > 0, \; y_0 \in \R$ can be arbitrarily chosen.  

\begin{lem}  Suppose $b(\cd,\cd)$ satisfies (\ref{A1}) and $u_\ve(x,y,t),  \ t < T, y \in \R$, is the unique bounded solution to (\ref{B1}), (\ref{C1}).  Then for any $\del > 0,  \ y_0 \in \R,  \ t < T$, there is a positive constant $C(\del, y_0, t)$ such that
\be \label{U3}
u_\ve(x,y,s) \ge 1/C(\del, y_0, t), \quad  y \ge y_0, \ \ t \le s \le T-\del,
\ee
\[  |\pa u_\ve(x,y,s)/\pa y| + |\pa^2 u_\ve(x,y,s)/\pa y^2| \le C(\del, y_0, t), \quad  y \ge y_0, \ t \le s \le T-\del.   \]
\end{lem}
\begin{proof} To prove the first inequality in (\ref{U3}) we proceed as in Lemma 3.1, using the representation (\ref{F1}).  Since the solution $Y_\ve(s)$ of (\ref{E1}) which has initial condition $Y_\ve(t) = y$ satisfies the inequality (\ref{K3}), it follows that there exists $y_1 > x$ with the property that $u_\ve(x,y_1,s) \ge 1/2, \ t \le s \le T$.  We consider now $y$ in the interval $y_0 < y < y_1$.  Let $\al$ be defined by 
\[	\al = \inf \Big\{ b(y', s) : y_0-1 \le y' \le y_1, \  t \le s \le T \Big\},	\]
and $Z_\ve(s)$ satisfy the stochastic equation
\[  dZ_\ve(s) = \al ds + \sqrt{\ve} \; dW(s), \quad  Z_\ve(t) = y . \]
Then $Y_\ve(s) \ge Z_\ve(s), \  t \le s \le \tau $,  where $\tau$ is the first exit time of $Z_\ve(s)$ from the interval $[y_0 - 1, y_1]$.  We can easily estimate from below $P\left( \tau < T, \  Z_\ve(\tau) = y_1\right) $.  Combining this with (\ref{K3}) we see that the first inequality in (\ref{U3}) holds for $y_0 < y < y_1$.

We turn to the problem of estimating the derivatives in (\ref{U3}).  Let $y_1 \in \R$ and $T_1 \le T$.  We shall be interested in constructing the solution to the terminal-boundary value problem
\be \label{V3}
\frac{\pa w}{\pa t} + b(y,t) \; \frac{\pa w}{\pa y} + \frac \ve 2 \; \frac{\pa^2 w}{\pa y^2} = 0, \quad  y_1 - \eta < y < y_1 + \eta, \ t<T_1,
\ee
\[	w(y, T_1) \ = \ w_0(y), \quad y \in [ y_1 -\eta, \ y _1 + \eta],	\]
\[ 	w(y_1 - \eta, s) = w_-(s), \quad w(y_1 + \eta, s) = w_+(s), \ s \le T_1,  \]
where $\eta > 0$ and the functions $w_0(\cdot), w_-(\cdot), w_+(\cdot)$ are assumed to be continuous on their domains.  The solution to (\ref{V3}) can be represented in terms of the Dirchlet Green's function $G(y, y', t, T_1)$ for the problem.  Thus
\begin{multline} \label{X3}
w(y,t) = \int^{y_1 +\eta}_{y_1 -\eta} \ G(y, y', t, T_1)\; w_0(y')dy' + \\
 \ve \int^{T_1}_t ds\; w_-(s) \frac{\pa G}{\pa y'} (y, y_1 -\eta, t, s) - \ve \int^{T_1}_t ds\; w_+(s) \frac{\pa G}{\pa y'} (y, y_1 + \eta, t,s). 
\end{multline}
We shall show that the Green's function may be constructed by perturbation expansion provided $t < T_1$ lies in an interval $t \in [T_1 - \De, T_1]$ where $\De,\eta $ satisfy the inequalities
\be \label{Y3}
\De \le \eta^2/\ve, \quad  \De \le \nu \ve \Big/ \left[ \sup \Big\{ |b(y,s)| : y_1 - \eta \le y \le y_1 +\eta, \; T_1-\De \le s \le T_1 \Big\}\right]^2,
\ee
for some $\nu < 1$ independent of $b(\cd, \cd)$ and $\ve$.

We construct the Green's function by the standard method \cite{fried}.  Thus let $G_D(y, y', t)$ be the Green's function for the heat equation on the interval $[-1, 1]$ with Dirichlet boundary conditions.  The function $G_D$ is given from the method of images as an infinite series,
\be \label{AL3}
G_D(y,y',t) = \sum^\infty_{m=0} (-1)^{p(m)} \ G(y-y'_m, t),
\ee
where $y'_0 = y'$ and $y'_m, \; m \ge 1$, are the multiple reflections of $y'$ in the boundaries $-1, 1$, with $p(m)$ being the parity of the reflection, $p(0) = 0$.  The function $G(y,t)$ is a Gaussian with mean 0 and variance $t$.  We now set $K(y,y',t,s)$ to be
\be \label{Z3}
K(y,y',t,s) = \eta^{-1} G_D \big( [y-y_1]/\eta, [y'-y_1]/\eta, \; \ve(s-t)/\eta^2 \big),  \quad y,y' \in [y_1-\eta, y_1 +\eta], \  t < s.
\ee
The Green's function $G(y,y',t,T_1)$ is formally given by an expansion in terms of the function $K$.  Let ${\mathcal L}_{t,y}$ denote the operator on the LHS of (\ref{V3}), so (\ref{V3}) is ${\mathcal L}_{t,y} w=0$.  Then
\begin{eqnarray} \label{AA3}
G(y,y',t,T_1) &=& K(y,y',t,T_1) - \sum^\infty_{n=0} \; v_n(y,y',t,T_1), \\
v_n(y,y',t,T_1) &=& -\int^{T_1}_t ds \int^{y_1 + \eta}_{y_1 - \eta} dz \  K(y,z,t,s) g_n(z, y', s, T_1), \nn \\
g_0(y,y',t,T_1) &=& {\mathcal L}_{t,y} K(y,y',t,T_1), \nn \\
g_{n+1}(y,y',t,T_1) &=& \int^{T_1}_t ds \int^{y_1 + \eta}_{y_1 - \eta} dz \  {\mathcal L}_{t,y} K(y,z,t,s) g_n(z, y', s, T_1). \nn
\end{eqnarray}
One easily obtains from (\ref{AA3}) the estimate 
\begin{multline} \label{AB3}
|g_n(y,y',t,T_1)| \le C^n \left[ \sup \Big\{ |b(z,s)| : y_1 - \eta \le z \le y_1 + \eta,  \ t \le s \le T_1 \big\} \right]^{n+1} \\
\frac{(T_1-t)^{n/2\;-\;1/2}} {\ve^{(n+1)/2}} \ G\big(y - y', 2\ve(T_1 - t)\big), \quad n = 0,1,2,....,  
\end{multline}
for some universal constant $C$, provided $\eta \ge \sqrt{\ve(T_1 - t)}$.  It follows from (\ref{AB3}) that the series expansion (\ref{AA3}) for the function $G$ converges provided $t \in [T_1 -\De, T]$, where $\De,\eta$ satisfy (\ref{Y3})  for some sufficiently small universal $\nu > 0$.  In that case one has the following estimate on the Green's function:
\be \label{AC3}
G(y,y',t,T_1) \le C  \ G\big(y - y', 2\ve(T_1 - t)\big),
\ee
for a universal constant $C > 0$.

We can obtain estimates for the derivatives of $G$ analogous to (\ref{AC3}) by differentiating the expansion (\ref{AA3}) term by term.  We first consider $\pa G(y,y',t,T_1)/\pa y'$.  For $t \in [T_1 -\De, T_1]$ and $\De,\eta$ satisfying (\ref{Y3}). We have from (\ref{AA3}) that
\be \label{AD3}
\Big| \frac{\pa g_0}{\pa y'} (y,y',t,T_1) \Big| \le \frac C{T_1-t} \; \sqrt{\frac \nu{\ve \De}}  \ G\big(y-y', 2\ve(T_1 - t) \big)
\ee
for some universal constant $C$.  The integral representation in (\ref{AA3}) for $\pa v_0(y,y',t,T_1)/\pa y'$ gives rise to a non-integrable singularity in the integration with respect to $s, \; t \le s < T_1$, if we use (\ref{AD3}).  We therefore need to use the fact that $g_0(z,y',s,T_1) = b(z,s)  \ \pa K(z,y',s,T_1)/\pa z$ and integrate by parts with respect to $z$ in the representation (\ref{AA3}) for $v_0(y,y',t,T_1)$.  We conclude that 
\be \label{AE3}
\Big| \frac{\pa v_0}{\pa y'} \; (y,y',t,T_1)\Big| \le \left[ A\left( \frac{T_1-t}\ve \right)^{1/2} + \left( \frac \nu{\ve \De}\right)^{1/2}\right] C\; G(y-y', 2\ve(T_1-t))
\ee
for some universal constant $C$, where $A$ is the upper bound  in (\ref{A1}) on the derivative of $b(\cd,\cd)$.  We can use a similar method to obtain a bound on the derivative of $g_1$.  Thus we have
\be \label{AF3}
\Big| \frac{\pa g_1}{\pa y'} \; (y,y',t,T_1)\Big| \le \left[ A \left( \frac \nu{\ve \De}\right)^{1/2} + \frac \nu{\De\sqrt{\ve(T_1-t)}}\right] C\; G(y-y', 2\ve(T_1-t))
\ee
for some universal constant $C$.  Choosing $\De$ now to also satisfy $\De < \sqrt{\nu}/A$ we conclude from (\ref{AF3}) and the representation (\ref{AA3}) for $g_n$ that
\be \label{AG3}
\Big| \frac{\pa g_n}{\pa y'} \; (y,y',t,T_1)\Big| \le \frac {\nu^{(n+1)/2}(T_1-t)^{n/2\;-\;1}} {\sqrt{\ve}\ \De^{(n+1)/2}} C^n\; G(y-y', 2\ve(T_1-t)), \quad n=1,2,...,
\ee
where $C$ is a universal constant.  The estimate (\ref{AG3}) gives an estimate on the derivatives of $v_n$, $n \ge 1$,
\be \label{AH3}
\Big| \frac{\pa v_n}{\pa y'} \; (y,y',t,T_1)\Big| \le \frac {\nu^{(n+1)/2}(T_1-t)^{n/2}}{\sqrt{\ve}\ \De^{(n+1)/2}}  C^n\; G(y-y', 2\ve(T_1-t)).
\ee
for a universal constant $C$.  We conclude then from (\ref{AE3}), (\ref{AH3})  that on choosing $\nu > 0 $ sufficiently small in a universal way, the function $G(y,y', t,T_1)$  is differentiable with respect  to $y'$ for $t \in [T_1 - \De, T_1]$ and 
\be \label{AI3}
\Big| \frac{\pa G(y,y',t,T_1)}{\pa y'}\Big| \le \frac C {\sqrt{\ve(T_1-t)}} \ G(y-y', 2\ve(T_1-t)),
\ee
for some universal constant $C$.  Hence the integral representation (\ref{X3}) is well-defined for $\De,\eta$ satisfying (\ref{Y3}) and $t \in [T_1 - \De, T_1]$.

We can obtain estimates on other derivatives of $G$ by a similar method.  Observe that from (\ref{AB3}) we may conclude that $G(y,y', t,T_1)$  is differentiable with respect to $y$ for $t \in [T_1 - \De, T_1]$ and 
\be \label{AJ3}
\Big| \frac{\pa G(y,y',t,T_1)}{\pa y}\Big| \le \frac C{\sqrt{\ve(T_1-t)}} \ G(y-y', 2\ve(T_1-t)),
\ee
for some universal constant $C$.  To obtain an estimate on $\pa^2 G(y,y',t,T_1)/\pa y^2$ we must first obtain estimates on $\pa g_n (y,y',t,T_1)/\pa y$.  Evidently we have that
\be \label{AK3}
\Big| \frac{\pa g_0}{\pa y} \; (y,y',t,T_1)\Big| \le \left[ \frac A {\sqrt{\ve(T_1-t)}} + \frac {\sqrt{\nu}} {(T_1-t)\sqrt{\ve\De}} \right]  \; C \; G(y-y', 2\ve(T_1-t))
\ee
for some universal constant $C$.  To estimate $\pa g_1(y,y',t,T_1)/\pa y$ we write the integral representation (\ref{AA3}) as an integral over $t < s < (T_1 + t)/2$ plus an integral over $(T_1 + t)/2 < s < T_1$.  Since the integral over $(T_1 + t)/2 < s < T_1$, may be estimated using (\ref{AB3}) we concentrate on the integral over $t < s < (T_1 + t)/2$.  Now the kernel ${\mathcal L}_{t,y} K(y,z,t,s)$ which appears in the integral representation (\ref{AA3}) for $g_1$ is a sum of terms generated by the boundary reflections which occur in the representation (\ref{AL3}) for $G_D$.  We consider the principle term in this series, which makes a contribution to the representation for $g_1$ given by
\be\label{AM3}
f(y,y',t,T_1) = \int^{(T_1+t)/2}_t  ds  \int^{y_1+\eta}_{y_1-\eta}  dz \; b(y,t) \frac {\pa}{\pa y} G(y-z, \ve(s-t))g_0(z,y',s,T_1) 
\ee
\begin{eqnarray*}
&=& \int^{(T_1+t)/2}_t  ds \int^{y_1+\eta}_{y_1-\eta} \; dz \; b(y,t) \; G(y-z, \ve(s-t)) \frac {\pa g_0}{\pa z}(z,y',s,T_1) \\
&+& \int^{(T_1+t)/2}_t  ds \   b(y,t) \; G(y-y_1 +\eta, \ve(s-t)) g_0(y_1-\eta,y',s,T_1) \\
&-& \int^{(T_1+t)/2}_t  ds \   b(y,t) \; G(y-y_1 - \eta, \ve(s-t)) g_0(y_1+\eta,y',s,T_1). 
\end{eqnarray*}
Denoting the first integral on the RHS of (\ref{AM3}) by $I_1(y)$ we see from (\ref{AK3}) that $I_1(y)$ is differentiable with respect to $y$ and
\be \label{AN3}
\Big| \frac{dI_1}{dy}(y)\Big| \le \left[ A(T_1 - t) + \sqrt{\nu}(T_1-t)^{1/2} \Big/ \sqrt{\De} \right] \left[ \frac A{\sqrt{\ve(T_1-t)}} + \frac{\sqrt{\nu}} {(T_1-t)\sqrt{\ve\De}}\right] C\;G(y-y',2\ve(T_1-t)),
\ee
for some universal constant $C$.  Let $I_2(y)$ denote the second integral on the RHS of (\ref{AM3}).  Using the fact that
\[	\int^\del_0 \; ds \; \frac \xi {(\ve s)^{3/2}} \exp \left[ - \frac{\xi^2}{2\ve s} \right] = \int^\infty_{\xi^2/\ve\del} dz\;e^{-z/2} \big/ \ve z^{1/2},  \]
we see that $I_2(y)$ is differentiable w.r. to $y$ and
\be \label{AO3}
\Big| \frac{dI_2(y)}{dy} \Big| \le \left( \frac \nu{\ve\De} \right)^{1/2} \left[ A + \left\{ \frac \nu{\De(T_1-t)} \right\}^{1/2} \right] 
C\;G(y-y',2\ve(T_1-t)),
\ee 
for some universal constant $C$.  We get a similar estimate to (\ref{AO3}) for the third integral on the RHS of (\ref{AM3}).  It is clear that the higher terms in the series (\ref{AL3}) for ${\mathcal L}_{t,y} \;K(y,z,t,s)$ make smaller contributions to $\pa g_1/\pa y$ than the RHS of (\ref{AN3}), (\ref{AO3}).  We conclude that
\be \label{AP3}
\Big| \frac{\pa g_1}{\pa y} (y, y', t, T_1)\Big| \le \left( \frac{T_1 - t}{\ve}\right)^{1/2} \left[ A + \left\{ \frac \nu{\De(T_1-t)} \right\}^{1/2} \right]^2 C\,G(y-y',2\ve(T_1-t)),
\ee
for some universal constant $C$.  Using the representation (\ref{AA3}) for $g_{n+1}$ we can now see by induction that 
\be \label{AQ3}
\Big| \frac{\pa g_n}{\pa y} (y, y', t, T_1)\Big| \le  \frac{(T_1 - t)^{n\;-\;1/2}} {\sqrt{\ve}} \left[ A + \left\{ \frac \nu{\De(T_1-t)} \right\}^{1/2} \right]^{n+1} C^n \ G(y-y',2\ve(T_1-t)),  \ \ n \ge 0,
\ee
for some universal constant $C$.  We may use (\ref{AB3}) and (\ref{AQ3}) to estimate the second derivative of the function $v_n(y, y', t, T_1)$ in (\ref{AA3}) with respect  to $y$.  Thus we have 
\be \label{AR3}
\Big| \frac{\pa^2 v_n}{\pa y^2} (y, y', t, T_1)\Big| \le  \frac{(T_1 - t)^n} {{\ve}} \left[ A + \left\{ \frac \nu{\De(T_1-t)} \right\}^{1/2} \right]^{n+1} C^n \ G(y-y',2\ve(T_1-t)),  \ \ n \ge 0,
\ee
for some universal constant $C$.  We conclude then from (\ref{AR3}) that $G(y,y',t,T_1)$ is twice differentiable with respect to $y$ for $t \in [T_1 - \De, T_1]$ and
\be \label{AS3}
\left| \frac{\pa^2G(y,y',t,T_1)}{\pa y^2} \right|  \ \le \ \frac C {\ve(T_1-t)} G\big(y-y',2\ve(T_1-t)\big)
\ee
for some universal constant $C$.

Next we wish to estimate $\pa^2 G(y, y', t,T_1)/\pa y \pa y'$.  We can easily obtain this from the representation (\ref{AA3}) for $v_n$ and (\ref{AG3}).  Thus from (\ref{AG3}) we can estimate $\pa^2 v_n(y, y', t,T_1)/\pa y \pa y'$ for $n \ge 1$.  We need to integrate by parts to estimate $\pa^2 v_0(y, y', t,T_1)/\pa y \pa y'$ just as was the case for the estimate (\ref{AE3}).  We conclude that
\be \label{AT3}
\left| \frac{\pa^2G(y,y',t,T_1)}{\pa y \pa y'} \right| \le \frac C{\ve(T_1-t)} G\big(y-y',2\ve(T_1-t)\big)
\ee
for some universal constant $C$, provided $t \in [T_1 - \De, T_1]$. Finally we need to estimate the derivative $\pa^3 G(y, y', t,T_1)/\pa^2 y \pa y'$.  To do this we must first obtain estimates on $\pa^2 g_n(y, y', t,T_1)/\pa y \pa y'$.  Evidently we have that
\be \label{AU3}
\left| \frac{\pa^2g_0(y,y',t,T_1)}{\pa y \pa y'} \right| \le \frac 1{\ve(T_1-t)}  \left[ A + \left\{ \frac \nu{\De(T_1-t)} \right\}^{1/2} \right] \; C \ G \big(y-y',2\ve(T_1-t)\big),
\ee
for some universal constant $C$.  To estimate $\pa^2 g_1(y, y', t,T_1)/\pa y \pa y'$ we write the integral representation (\ref{AA3}) for $g_1$ as an integral over $t < s < (T_1 +t)/2$ plus an integral over $(T_1 +t)/2 < s < T_1$.  The second integral cannot be bounded by using (\ref{AD3}) so we need to resort to integration by parts as we did for the estimate (\ref{AE3}). To bound the contribution to $g_1$ from the integral over $t < s < (T_1+t)/2$ we use the representation (\ref{AM3}).  We conclude that
\be \label{AV3}
\Big| \frac{\pa^2 g_1}{\pa y\pa y'} (y, y', t, T_1)\Big| \le \frac 1\ve  \left[ A + \left\{ \frac \nu{\De(T_1-t)} \right\}^{1/2} \right]^2\; C\,G(y-y',2\ve(T_1-t)),
\ee
for some universal constant $C$.  Now by induction we see from the representation (\ref{AA3})  for $g_n$ that
\be \label{AW3}
\Big| \frac{\pa^2 g_n}{\pa y \pa y'} (y, y', t, T_1)\Big| \le  \frac{(T_1 - t)^{n-1}}{\ve} \left[ A + \left\{ \frac \nu{\De(T_1-t)} \right\}^{1/2} \right]^{n+1} C^n\,G(y-y',2\ve(T_1-t)), \ \  n \ge 0,
\ee
for some universal constant $C$.  Similarly to how we obtained (\ref{AR3}) from (\ref{AQ3}) we conclude from (\ref{AW3}) that
\be \label{AX3}
\Big| \frac{\pa^3 v_n}{\pa y^2 \pa y'} (y, y', t, T_1)\Big| \le  \frac{(T_1 - t)^{n\;-\;1/2}} {\ve^{3/2}} \left[ A + \left\{ \frac \nu{\De(T_1-t)} \right\}^{1/2} \right]^{n+1} C^n\,G(y-y',2\ve(T_1-t)),  \ \ n \ge 0,
\ee
for some universal constant $C$.  We conclude then from (\ref{AX3}) that provided $t \in [T_1 - \De, T_1]$, there is a universal constant $C$ such that
\be \label{AY3}
\Big| \frac{\pa^3 G(y,y',t,T_1)}{\pa y^2 \pa y'} (y, y', t, T_1)\Big| \le  \frac C{[\ve(T_1-t)]^{3/2}} \ G\big(y-y', 2\ve(T_1-t) \big).
\ee

We use the estimates (\ref{AJ3}), (\ref{AS3}), (\ref{AT3}) and (\ref{AY3}) to obtain bounds on the derivatives in (\ref{U3}).  In (\ref{X3}) we set $w(y,t) = 1 - u_\ve(x,y,t)$, where the boundary functions $w_0, w_-, w_+$ are all bounded by 1.  Then we estimate the derivatives of $u_\ve(x,y,t)$ with respect to $y$ by setting $y=y_1$ and estimating $\pa w(y,t)/\pa y, \; \pa^2 w(y,t)/\pa^2 y$ at $y = y_1$ using the Green's functions estimates.  It is clear then that by choosing $\De$ to be given by its maximum value in (\ref{Y3}) that we get an estimate
\be \label{AZ3}
|\pa u_\ve(x,y,s)/\pa y| + |\pa^2 u_\ve(x,y,s)/\pa y^2| \le C(\del, y_0, y_\infty, t)
\ee
for $(y,s)$ in any interval $y_0 \le y \le y_\infty$, $t \le s \le T-\del$.  Our final task is to show that the constant $C(\del,y_0,y_\infty,t)$ can be chosen independent of $y_\infty$ as $y_\infty \ra \infty$.  To see this we use the fact that the boundary functions $w_0, w_-, w_+$ converge to $0$ as $y_1 \ra \infty$.

Let $Y_\ve(s), \; t \le s \le T$, be the solution of the stochastic equation (\ref{E1}) with $Y_\ve(t) = y$, where $y > x$.  We need to estimate $P\big( Y_\ve(T) < x \  | \  Y_\ve(t) =y\big) $ as $y \ra \infty$.    To do this we let $Z_\ve(s)$ be the solution to the equation
\be \label{BA3}
dZ_\ve(s) = \left[ -A \; Z_\ve(s) + b(x,s)\right]ds + \sqrt{\ve} dW(s),  \quad s > t, \; Z_\ve(t) = y - x,
\ee
where $A$ is the upper bound in (\ref{A1})  for the derivative of $b(\cd,\cd)$.  
Then $Y_\ve(s) \ge Z_\ve(s) + x $, $t \le s \le \tau$,  where $\tau > t$ is the first hitting time at $0$ for the diffusion $Z_\ve(s)$ with $Z_\ve(t)=y-x$. 
The solution to (\ref{BA3}) is given by
\be \label{BB3}
Z_\ve(s) = (y-x)e^{-A(s-t)} + \int^s_t \; e^{-A(s-s')} \; b(x,s')ds' + \xi_\ve(s), \quad s > t,
\ee
where $\xi_\ve(s)$ satisfies the stochastic integral equation
\be \label{BC3}
\xi_\ve(s) = -A \int^s_t \xi_\ve(s')ds' + \sqrt{\ve}\; W(s), \quad s > t.
\ee
Applying Gronwall's inequality to (\ref{BC3}) we have that
\be \label{BD3}
\sup_{t \le s \le T} |\xi_\ve(s)| \le e^{A(T-t)} \; \sqrt{\ve} \ \sup_{t \le s \le T} |W(s)|.
\ee
We can estimate the probability that $\inf_{t \le s \le T} Z_\ve(s) < 0$ by using the inequality
\be \label{BE3}
P \left( \sup_{t \le s \le T} |W(s)| > a \right) \le \left[ \frac{ 8(T-t)}{\pi a^2}\right]^{1/2} \exp\left[ - \; \frac{a^2}{2(T-t)} \right].
\ee
Let us assume that the second term on the RHS of (\ref{BB3}) is smaller in absolute value than 1/2 the first term for $t \le s \le T$.  This can evidently be accomplished by choosing $y-x$ sufficiently large.  Then from  (\ref{BD3}), (\ref{BE3}) we conclude that 
\be \label{BF3}
P \left( \inf_{t \le s \le T} Z_\ve(s) < 0 \right) \le \left[ \frac{8(T-t)}{\pi }\right]^{1/2} \; \frac{2\sqrt{\ve} e^{2A(T-t)}} {(y-x)}  \exp\left[ - \; \frac {(y-x)^2}{8\ve(T-t)} e^{-4A(T-t)} \right].
\ee
Using the inequality
\[	P\left( Y_\ve(T) < x \  |  \ Y_\ve(t) = y \right) \le P \left( \inf_{t \le s \le T} Z_\ve(s) < 0 \right), \]
we obtain from (\ref{BF3}) bounds on the boundary functions $w_0, w_-, w_+$ in (\ref{X3}).  Evidently these are decaying exponentially in $y_1$ as $y_1 \ra \infty$, whereas it follows from (\ref{Y3}) and the Lipschitz condition (\ref{A1}) on $b(\cd,\cd)$ that we may take $\De \sim 1/y^2_1$ as $y_1 \ra \infty$.  We conclude that (\ref{AZ3}) holds uniformly as $y_\infty \ra \infty$. 
\end{proof} 
\begin{lem} Suppose $b(\cd, \cd)$ satisfies (\ref{A1}).  Then for $x,y \in \R, \; t < T$, and $\ve < 1$ there is the inequality
\be \label{BG3}
q_\ve(x,y,t) \ge q(x,y,t) - C(x,y,t,T) \sqrt{\ve},
\ee
where $q(x,y,t)$ is given by (\ref{J1}) and $C(x,y,t,T)$ is a constant independent of $\ve$.
\end{lem}
\begin{proof} Suppose $y_0< x$ and $y > y_0$.  Then by Lemma 3.4 we have the representation
\begin{multline} \label{BH3}
q_\ve(x,y,t) = E\bigg\{ \frac 1 2 \int^{(T-\del)\wedge\tau}_t \left[ \la_\ve \left( y_\ve(s), s\right) - b( y_\ve(s), s)\right]^2\; ds \\
+ q_\ve( x, y_\ve (T-\del)\wedge\tau), \ (T - \del) \wedge \tau)  \ \big| \  y_\ve(t) = y \bigg\}. 
\end{multline}
Here $\la_\ve(y,s)$ is given by the formula
\be \label{BI3}
\la_\ve(y,s) = b(y,s) - \pa q_\ve(x,y,s)/\pa y, \quad  y \in \R, \ s < T.
\ee
By Lemma 3.4 the function $\la_\ve(y,s)$ is uniformly Lipschitz in $y$ for $y \ge y_0$ and $t \le s \le T-\del$.  Hence (\ref{N1}) has a unique solution $y_\ve(s)$, $t \le s \le (T-\del) \wedge \tau$, where $\tau$  is the first hitting time at $y_0$.

We consider a random path $y_\ve(s), \; t \le s \le T-\del $, for which $\tau > T-\del$, and associate with it a classical path $y_{\ve,c}(s), \; t \le s \le T$.  To do this let $k$ be defined by 
\be \label{BJ3}
k = \max \left[ x - \del - y - \int^{T-\del}_t \; \la_\ve( y_\ve(s), s) ds,  \ 0 \right].
\ee
Then $y_{\ve,c}(s), \;t \le s \le T $, is the solution to the initial value problem
\begin{eqnarray} \label{BK3}
\frac{dy_{\ve,c}(s)}{ds} &=& \la_\ve( y_\ve(s), s) + k/(T-t-\del), \quad  t \le s \le T-\del, \\
\frac{dy_{\ve,c}(s)}{ds} &=& 2+b(y_{\ve,c}(s),s), \quad T-\del \le s \le T, \ \ y_{\ve,c}(t) = y. \nn
\end{eqnarray}
Since from (\ref{BJ3}) one has that $y_{\ve,c}(T-\del) \ge x-\del$, it follows that $y_{\ve,c}(T) \ge x$ provided $\del$ is sufficiently small.  Hence from (\ref{J1}) we conclude that
\be \label{BL3}
\frac 1 2 \int^T_t \left[ \frac{dy_{\ve,c}(s)}{ds} - b( y_{\ve,c}(s), s) \right]^2 \; ds \ \ \ge \ \ q(x,y,t).
\ee
From (\ref{N1}), (\ref{BK3})  we see that
\be \label{BM3}
y_{\ve,c}(s) - y_\ve(s) = \frac{k(s-t)}{(T-t-\del)} + \sqrt{\ve} \ [W(s) - W(t)], \quad  t \le s \le T-\del.
\ee
We may also rewrite the parameter $k$ in (\ref{BJ3}) as
\be \label{BN3}
k = \max \big[ x - \del - y_\ve(T-\del) + \sqrt{\ve} \ [W(T-\del) - W(t)],  \ 0 \big].
\ee
Observe now that
\begin{multline} \label{BO3}
\frac 1 2 \int^{(T-\del)}_t \left[ \la_\ve( y_\ve(s), s) - b( y_\ve(s), s)\right]^2\; ds \ge 
\frac 1 2 \int^{(T-\del)}_t \left[ \frac{dy_{\ve,c}(s)} {ds}  - b( y_{\ve,c}(s), s)\right]^2\; ds  \\
- \int^{(T-\del)}_t \left| \frac{dy_{\ve,c}(s)}{ds}  - b( y_{\ve,c}(s), s)\right| \left| b( y_{\ve,c}(s), s) - b( y_{\ve}(s), s) - k/ (T-t-\del)\right| ds. 
\end{multline}
Evidently from (\ref{BL3}) the first term on the RHS of (\ref{BO3}) is bounded below by $q(x,y,t) - C\del$ for some constant $C$.  Using (\ref{BN3}) and Lemma 3.1 we may bound the second term on the RHS of (\ref{BO3}).  First observe that this second term is bounded in absolute value by
\begin{multline} \label{BP3}
\frac \eta 2 \int^{(T-\del)}_t \left[ \la_\ve( y_\ve(s), s) - b( y_\ve(s), s)\right]^2\; ds \  + \\
\left[ 1 + \frac 1{2\eta}\right]  \int^{(T-\del)}_t \left[ b( y_{\ve,c}(s), s) - b( y_\ve(s), s) 
- k/(T-t-\del)\right]^2\; ds ,  
\end{multline}
for any $\eta > 0$.  From (\ref{BM3}) and the Lipschitz condition (\ref{A1}) on $b(\cd, \cd)$ the second term in (\ref{BP3}) is bounded above as
\begin{multline} \label{BQ3}
\int^{(T-\del)}_t \left[ b( y_{\ve,c}(s), s) - b( y_\ve(s), s) - k/(T-t-\del)\right]^2\; ds  \\
\le C_1 \; \ve \; \int^{T-\del}_t \left[ W(s) - W(t) \right]^2\; ds + C_2\; k^2,	
\end{multline}
where the constants $C_1, C_2$ depend only on $T-t$, assuming $\del < (T-t)/2$.  Hence from (\ref{BN3}), (\ref{BP3}), (\ref{BQ3}) we conclude that the second term on the RHS of (\ref{BO3}) is bounded by 
\begin{multline} \label{BR3}
\frac{\eta}{2} \int^{(T-\del)}_t \left[ \la_\ve( y_\ve(s), s) - b( y_\ve(s), s)\right]^2\; ds 
 + \  \frac{C_1\ve}\eta \int^{(T-\del)}_t  \left[ W(s) - W(t) \right]^2\; ds  \\
 + \frac{C_2\ve}\eta   \left[ W(T-\del) - W(t) \right]^2  
+ \  \frac{C_3}\eta \left\{ \max \left[ x - \del - y_\ve(T-\del), \  0\right]\right\}^2, 
\end{multline}
for any $\eta, \; 0 < \eta < 1$ and constants $C_1, C_2, C_3$ depending only on $T-t$.  It follows then on taking $\eta \sim \del$ in (\ref{BR3}) and using Lemma 3.1 that
\begin{multline} \label{BS3}
q_\ve(x, y_\ve(T-\del), T-\del) + \frac 1 2 \int^{(T-\del)}_t \left[ \la_\ve( y_\ve(s), s) - b( y_\ve(s), s)\right]^2\; ds \\
\ge q(x,y,t) - C_1\del - \frac{C_2\ve}{\del} \int^{(T-\del)}_t \left[ W(s) - W(t)\right]^2 \; ds +O(\ve)  \\
 - \frac{C_3\ve}{\del} \left[ W(T-\del) - W(t)\right]^2 - C_4 \del \int^{(T-\del)}_t \left[ \la_\ve( y_\ve(s), s) - b( y_\ve(s), s)\right]^2\; ds, 
 \end{multline}
for constants $C_1, C_2, C_3, C_4$ depending only on $T-t$.

To conclude the proof we take the expectation of (\ref{BS3}) on a set of paths $y_\ve(s)$, $t \le s \le T-\del$, for which $\tau > T-\del$.  To find a suitable set of paths note that $\pa q_\ve(x,y,s)/\pa y \le 0$, $y \in \R, \; s < T$, whence (\ref{BI3}) implies that $\la_\ve(y,s) \ge b(y,s)$, $y \in \R, s < T$.  Thus $y_\ve(s) \ge Y_\ve(s), \ t \le s \le T-\del$, where $Y_\ve(s)$ is the solution to (\ref{E1}) with $Y_\ve(t) = y$.  We have already estimated the fluctuation of $Y_\ve(s), \ t \le s \le T$, from $y$ by (\ref{K3}).  We therefore conclude that for given $y$ we may choose $y_0 < y$ such that
\be \label{BT3}
\sup_{t \le s \le T} |W(s) - W(t)| < 1/\sqrt{\ve} \ \ {\rm implies} \ \ \tau > T-\del.
\ee
The inequality (\ref{BG3}) follows now on taking $\del = \sqrt{\ve}$ in (\ref{BS3}) and taking the expectation on the paths for which (\ref{BT3}) holds. 
\end{proof}
\begin{proof}[Proof of Theorem 1.1] Evidently (\ref{L1}) follows from Lemma 3.3 and Lemma 3.5.
\end{proof}

\section{The Optimally Controlled Process}
In  Lemma 3.5 we already used the optimally controlled process $y_\ve(s)$ of (\ref{N1}) with controller (\ref{P1}) to obtain a lower bound on $q_\ve(x,y,t)$. The main goal of this section is to prove that $\liminf _{s\ra T} y_\ve(s) >x$ with probability $1$.  To do this we need to prove some short time asymptotic results for the cost function $q_\ve(x,y,t)$.
\begin{lem}  Suppose that $0 < T-t<\del \le \ve <1$.  Then the function $q_\ve(x,y,t)$ satisfies the inequalities
\be \label{Z4}
0 < q_\ve(x,y,t) \le C\ve + (x-y)^2/(T-t) + C(x,\del) \left[ (y-x)^2 + |y-x| + \sqrt{\ve(T-t)}\right], \quad  y < x,
\ee
\[   0 < q_\ve(x,y,t) \le C\ve \exp  \left[ -(x-y)^2 /2\ve(T-t)\right] + C(x,\del)\left[ (y-x)^2 + |y-x| + \sqrt{\ve(T-t)}\right], \quad  y > x,  \]
where $C$ is a universal constant and $C(x,\del)$ depends only on $x$ and $\del$.  The function $\pa q_\ve(x,y,t)/\pa y$ satisfies the inequality
\be \label{M4}
-\frac{\pa q_\ve}{ \pa y}(x,y,t) \ge \frac{x-y}{T-t} \exp \left[ -C(x,\del) \left\{ (T-t) | \log (T-t)| + [(T-t)/\ve]^{1/2} + (y-x)^2/\ve +|y-x|/\ve \right\} \right],
\ee
for a constant $C(x,\del)$ depending only on $x$ and $\del$.
\end{lem}
\begin{proof}   We apply the Schwarz inequality in the PDE (\ref{H1}) for $q_\ve(x,y,t)$.  Thus for any $\al > 0$, 
\be \label{A4}
\frac {\pa q_\ve}{\pa t} +b(x,t) \frac {\pa q_\ve}{\pa y} -\frac{1}{2}(1-\al) \left( \frac {\pa q_\ve}{\pa y} \right)^2 + \frac \ve 2 \ \frac {\pa^2 q_\ve}{\pa y^2} + \frac{1}{2\al}[ b(y,t) - b(x,t)]^2  \ge 0.
\ee
Setting $v_\al(y,t) = \exp[-(1-\al) q_\ve(x,y,t)/\ve ]$, we see from (\ref{A4}) that
\be \label{B4}
\frac{\pa v_\al}{\pa t} + b(x,t) \frac{\pa v_\al}{\pa y} + \frac \ve 2 \ \frac{\pa^2 v_\al}{\pa y^2}  \le \frac{(1-\al)}{2\al \ve}
\left[ b(y,t) - b(x,t)\right]^2 \ v_\al,
\ee
provided $\al<1$. It follows now from (\ref{B4}) that $v_\al$ is bounded below by
\begin{multline} \label{C4}
v_\al(y,t) \ge E\Bigg[\exp\left\{-\int_t^T (1-\al)[b(y+g(s)+\sqrt{\ve} \ W(s-t),s)-b(x,s)]^2 ds/2\al\ve\right\}  \\
v_\al(y+g(T)+\sqrt{\ve} \ W(T-t),T)\Bigg],
\end{multline}
where $W(\cd)$ is Brownian motion and $g(\cd)$ is given by
\be \label{D4}
g(s) = \int^s_t \ b(x,s')ds', \quad t \le s \le T.
\ee
Observing that $v_\al$ has terminal data $v_\al(y,T) = 0$ for  $y < x$, and $v_\al(y,T) = 1$ for $y > x$, we conclude from (\ref{C4}) that
\be \label{E4}
v_\al(y,t) \ge \int^\infty_{x-g(T)} \ \frac 1 {\sqrt{2\pi \ve(T-t)}} \exp \left[ - \frac{(y-z)^2}{2\ve(T-t)} \right] F(y,z)  \ dz,
\ee
where $F(y,z)$ is given by the formula
\begin{multline} \label{F4}
F(y,z) = E \Bigg[ \exp \Big\{ - \int^T_t \; (1-\al) \Big[ b([(T-s)y + (s-t)z]/(T-t) \\
	+ \sqrt{\ve} [W(s-t) - (s-t)W(T-t)/(T-t)] + g(s), s) - b(x,s) \Big]^2 \; ds \Big/2\al \ve \Big\} \Bigg].  
\end{multline}
In (\ref{E4}) we have used the Brownian bridge representation for Brownian motion conditioned at times $t$ and $T$.  Using Jensen's inequality in (\ref{F4}) and the Lipschitz bound (\ref{A1}) on $b(\cd, \cd)$, we conclude that 
\begin{multline} \label{G4}
- \log F(y,z) \le \frac{A^2(1-\al)}{2\al \ve} \int^T_t \; ds \; E\Big[ \Big\{ (T-s)(y - x) \\
+ (s-t)(z-x)] / (T-t) + g(s) + \sqrt{\ve} \ [ W(s-t) - (s-t)W(T-t)/(T-t)]\Big\}^2 \Big]  \\
= \frac{A^2(1-\al)}{2\al \ve} \int^T_t \; ds \; \Big\{[ (T-s)(y - x) + (s-t)(z-x)]/(T-t) + g(s)\Big\}^2  \\
+ \frac{A^2(1-\al)}{2\al } \int^T_t \; ds \; (s-t)(T-s)/(T-t) \ . 
\end{multline}
It follows now from (\ref{D4}) and (\ref{C4}) that for any $\del > 0$ there is a constant $C(x,\del)$ depending only on $x,\del$ such that
\begin{multline}  \label{H4}
- \log F(y,z) \le  \frac{A^2(1-\al)}{2\al \ve}  \Big[ (z-x)^2(T-t) + (y-x)^2(T-t) \\
+ C(x,\del)(T-t)^3 + \ve(T-t)^2/6 \Big], \quad  T-t < \del.
\end{multline}
We may combine (\ref{E4}), (\ref{H4}) to obtain an upper bound on $q_\ve(x,y,t)$.  Thus on using the inequality $(z-x)^2 \le 2(z-y)^2 + 2(y-x)^2$ in (\ref{H4}), we conclude from (\ref{E4}) that 
\begin{multline} \label{I4}
v_\al(y,t) \ge  
\exp \left[ - \frac{A^2(1-\al)}{2\al \ve}  \left\{ 3(y-x)^2(T-t) + C(x,\del)(T-t)^3 + \ve(T-t)^2/6 \right\} \right] \\
\int^\infty_{x-y-g(T)} \; \frac 1{\sqrt{2\pi\ve (T-t)}} \exp \left[ -z'^2 \left\{ \frac 1{2\ve (T-t)} + \frac{A^2(1-\al)}{\al \ve} (T-t) \right\} \right] dz'  \  ,  \quad T-t<\del . 
\end{multline}

Let us recall the inequality
\be \label{J4}
\frac 1 a \left( 1 - \frac 1{a^2} \right)  e^{-a^2/2} < \int^\infty_a e^{-z^2/2} \; dz < \frac 1 a \ e^{-a^2/2}, \quad a > 0.
\ee
We shall use it to show that there is a universal constant $C$ such that
\be \label{K4}
\int^\infty_{a+\eta} e^{-z^2/2} \; dz \ge \exp \Big[ -\eta^2/2 - C\eta \max\{a,1\}  \Big] \int^\infty_a e^{-z^2/2} \; dz, \quad \eta>0, \ a\in\R.
\ee
To see this observe that by Jensen's inequality
$$
\int^\infty_{a+\eta} e^{-z^2/2} \; dz \ge \exp \Big[ -\eta^2/2 - \eta \left< Z \right> \Big] \int^\infty_a e^{-z^2/2} \; dz,
$$
 where $Z$ is the standard normal variable conditioned on  $Z > a$.   Evidently if $a \le 2$ then $|\left< Z \right>| \le C_1$ for some universal constant $C_1$.  If $a\ge 2$ we see from (\ref{J4}) that
\[	\left< Z \right> \ \le \  a\left( 1 - \frac 1{a^2}\right)^{-1} \ \le 4a/3, \]
whence (\ref{K4}) holds for all $a \in \R$.

We shall apply the inequality (\ref{K4}) in (\ref{I4}) to obtain an upper bound on $q_\ve(x,y,t)$ in terms of the cumulative distribution function $\Phi$ for the standard normal variable.  Now the integral with respect to $z'$ on the RHS of (\ref{I4}) is given by
\be \label{L4}
\frac 1{[1+2A^2(1-\al)(T-t)]^{1/2}}  \ \Phi \left( \frac{y-x+g(T)}{\sqrt{\ve(T-t)}} \ \left[ 1+2A^2(1-\al)(T-t)\right]^{1/2} \right)
\ee
if we set $\al = T-t$.  We write the argument of $\Phi$ in (\ref{L4}) as $-[a+\eta]$ with $a = (x-y)/\sqrt{\ve(T-t)}$ and apply (\ref{K4}).  Thus we obtain the inequality 
\begin{multline} \label{P4}
\Phi \left( \frac{y-x+g(T)}{\sqrt{\ve(T-t)}} \ \left[ 1+2A^2(1-\al)(T-t)\right]^{1/2} \right) \ge \\
\Phi  \left( \frac{y-x}{\sqrt{\ve(T-t)}} \right) \exp \left[ - \frac{C(x,\del)}{\ve} \left\{ (y-x)^2 + |y-x| + \sqrt{\ve(T-t)} \right\}\right], \quad T-t < \del \le \ve, 
\end{multline}
for some constant $C(x,\del)$ depending only on $x,\del$.  If we combine (\ref{P4}) with (\ref{I4}), taking $\al = T-t$, we obtain an upper bound on $q_\ve$,
\begin{multline} \label{N4}
q_\ve(x,y,t) \le -\ve \log \Phi\left( [y-x] / \sqrt{\ve(T-t)} \ \right) + \\
C(x,\del) \left[ (y-x)^2 + |y-x| + \sqrt{\ve(T-t)} \ \right], \quad T-t < \del \le \ve,
\end{multline}
 for a constant $C(x, \del)$ depending only on $x$ and $\del$.  The inequality (\ref{Z4}) follows from (\ref{N4}) on using (\ref{J4}).  Note that (\ref{J4}) for $y < x$ follows from (\ref{N4}) on using the fact that log $a \le a^2/2$ for $a>1$.

Next we turn to estimating $\pa q_\ve(x,y,t)/\pa y$.  To do this we consider the Green's function 
$G(y,y',t,T)$  of (\ref{D1}). It follows from (\ref{D1})  that $-\pa u_\ve(x,y,t)/\pa x$ $=G(y,x,t,T)$.  If we differentiate (\ref{B1}) with respect to $y$ and use the maximum principle, we see also that
\be \label{O4}
\pa u_\ve (x,y,t)/\pa y \ge e^{-A(T-t)} G(y,x,t,T),
\ee
where $A$ is the Lipschitz constant in (\ref{A1}).  Since 
$-\pa q_\ve(x,y,t)/\pa y = \ve \ [\pa u_\ve(x,y,t)/\pa y]/u_\ve(x,y,t)$, we may obtain the lower bound (\ref{M4}) by finding a lower bound for $G(y,x,t,T)$ and a lower bound for $q_\ve(x,y,t)$ which is complimentary to (\ref{N4}).

We turn to the problem of obtaining a lower bound for $q_\ve$.  Instead of (\ref{A4}) we use the differential inequality
\be \label{Q4}
\frac {\pa q_\ve}{\pa t} + b(x,t) \; \frac {\pa q_\ve}{\pa y}  - \frac 1 2 (1+\al) \left( \frac {\pa q_\ve}{\pa y}\right)^2 +
\frac \ve 2 \; \frac {\pa^2 q_\ve}{\pa y^2} - \big[ b(y,t) - b(x,t) \big]^2 / 2\al \le 0,
\ee
for any $\al > 0$.  Setting $v_\al(y,t) = \exp [-(1+\al)q_\ve(x,y,t)/\ve]$ we see from (\ref{Q4}) that
\be \label{R4}
v_\al(y,t) \le \int^\infty_{x-g(T)} \frac 1{\sqrt{2\pi \ve(T-t)}} \exp \left[ - \frac{(y-z)^2}{2\ve(T-t)} \right] F(y,z)dz,
\ee
where $F(y,z)$ is given by the formula
\begin{multline}  \label{S4}
F(y,z) = E\Big[ \exp \Big\{ \int^T_t \frac{A^2(1+\al)}{2\al\ve}\; ds \Big( [T-s)(y-x) + (s-t)(z-x)] \big/ (T-t)\\
+ g(s) + \sqrt{\ve}  \ [W(s-t)-(s-t)W(T-t)/(T-t)]\Big)^2\Big\} \Big]. 
\end{multline}
The expectation in (\ref{S4}) cannot be evaluated exactly as was the case with (\ref{G4}), but it may be estimated using the fact that one knows the probability density function of $\displaystyle{\sup_{t\le s\le T}} W(s-t)$.  Taking $\al = T-t$ in (\ref{S4}), we see from this that
\be \label{T4}
\log F(y,z) \le \frac{CA^2}\ve \left[ (z-x)^2 + (y-x)^2 + C(x,\del)(T-t)^2 + \ve(T-t)\right], \quad  T-t < \del,
\ee
for a universal constant $C$ and constant $C(x,\del)$ depending on only $x,\del$.  Note here that we require $\del < 1/A^2$ for the expectation (\ref{S4}) to be finite.  To obtain the lower bound on $q_\ve$ we combine (\ref{T4}) and (\ref{R4}) with the inequality (\ref{K4}).  Since we are obtaining an upper bound on the function $v_\al(y,t)$, we apply (\ref{K4}) with $a + \eta = (x-y)/\sqrt{\ve(T-t)}$.  Hence we get an inequality complimentary to (\ref{N4}),
\begin{multline} \label{U4}
q_\ve(x,y,t) \ge - \ve \; \log \Phi \left( [y-x]/\sqrt{\ve(T-t)}  \ \right) - \\ 
C(x,\del) \left[(y-x)^2 + |y-x| +  \sqrt{\ve(T-t)} \  \right], \quad T-t< \del \le \ve,
\end{multline}
for a constant $C(x,\del)$ depending only on $x$ and $\del$.

The lower bound for $G(y,x,t,T)$ may be obtained in a similar way to the upper bound on $q_\ve(x,y,t)$.  Let $0 < \De < T-t$ and $0 < \al < 1$.  Then just as in (\ref{C4}) we have that
\begin{multline}  \label{V4}
G(y,x,t,T)^{1-\al} \ge \int^\infty_{- \infty} \frac{1}{\sqrt{2\pi\ve (T-t-\De)}}  \\
\exp \left[ - \frac{(y-z)^2}{2\ve(T-t-\De)} \right] F_\De(y,z)  \ G(z+ g(T-\De), x, T-\De, T )^{1-\al}\; dz,  
\end{multline}
where $F_\De$ is as in (\ref{F4}) but with $T$ replaced by $T-\De$.  Observe that we cannot take $\De \ra 0$ on the RHS of (\ref{V4}) since the integrand would contain in the limit $\del(z+g(T-\De)-x)^{1-\al}$, which is identically zero.  We shall choose $\De$ so that $0 < \De << T-t$ and $\al = T-t$, in a way that the function $z \ra G(z+ g(T-\De), x, T-\De, T)^{1-\al}$ is approximately a Dirac delta function concentrated at $x$.

It is evident that the RHS of (\ref{V4}) is decreased upon replacing $G$ by the corresponding Dirichlet Green's function $G_D$ for an interval centered at $x$.  As in Lemma 3.4 we choose this interval sufficiently small and $\De$ sufficiently small so that $G_D$ may be expanded in a perturbation series.  The condition for this has already been given in (\ref{Y3}).  Thus the Green's function $G_D(z, x, T-\De, T)$ on the interval $x-\eta \le z \le x+\eta$ has a convergent perturbation expansion provided $\eta, \De$ satisfy the inequalities
\be \label{W4}
\ve\De \le \eta^2, \ \ \De \le \nu \ve \Big/ [A\eta + C(x,\del)]^2, \quad  \De \le \del,
\ee
where $A$ is the Lipschitz constant from (\ref{A1}) and $C(x,\del)$ is a constant depending only on $x,\del$.  In that case there are universal constants $C_1, C_2$ such that
\begin{multline} \label {X4}
\int^{x+\eta}_{x-\eta} G_D(z,x,T-\De,T)dz \ge 1 - C_2\exp[-\eta^2/4\ve\De] \\
 - C_2[A\eta + C(x,\del)](\De/\ve)^{1/2}, \quad 0 < \De < \del.
\end{multline}
Observe that if we take $\De = (T-t)^3, \eta = (T-t)\sqrt{\ve}$ then the RHS of (\ref{X4}) is bounded below by $1-C(x,\del)(T-t)$ for $0 < T-t < \del \le \ve$, where $C(x,\del)$ depends only on $x$ and $\del$.  Taking $\al = T-t$ we may see further that with the same values for $\De, \eta$ there is the inequality
 \be \label{Y4}
\int^{x+\eta}_{x-\eta} G_D(z,x,T-\De,T)^{1-\al}dz \ge 1 - C(x,\del)(T-t)|\log(T-t)|, \quad 0 < T-t < \del \le \ve,
\ee
for a constant $C(x,\del)$ depending only on $x,\del$.   It follows then from (\ref{H4}), (\ref{V4}), (\ref{Y4}), that
\begin{multline} \label{AA4}
G(y,x,t,T) \ge \frac 1{\sqrt{2\pi \ve(T-t)}} \exp \bigg[ - \frac{(y-x)^2}{2\ve(T-t)}  -C(x,\del)\big\{ (T-t)|\log(T-t)| \\
+ (y-x)^2/\ve + |y-x|/\ve\big\} \bigg],  \quad 0 < T-t < \del \le \ve,  
\end{multline}
for a constant $C(x,\del)$ depending only on $x,\del$.

To obtain the lower bound (\ref{M4}) we combine (\ref{U4}) and (\ref{AA4}) using (\ref{O4}).  The inequality (\ref{M4}) now follows from (\ref{J4}). 
\end{proof}
\begin{rem}  There is a vast literature on short time asymptotics of solutions to diffusive equations.  See in particular the classical papers of Kannai \cite{kan},  Minakshisundaram \cite{min},  Molchanov \cite{mol}, and Varadhan \cite{var}.
\end{rem}
Lemma 4.1 shows that for $y < x$ and $s < T$ with $T-s$ small, the optimal controller $\la^*(x,y,s)$,  given by (\ref{P1}) for the stochastic control problem (\ref{O1}),  is approximately $\la^*(x,y,s) = (x-y)/(T-s)$.  This will enable us to  show that the solution $y_\ve(s)$ of the corresponding stochastic differential equation (\ref{N1}) satisfies $\displaystyle{\liminf_{s\ra T}} \ y_\ve(s) > x$ with probability 1.  First we show this for the linear approximation which we have just established. 
\begin{lem} Suppose $\mu > 0, \ \ve > 0$ and $Z_\ve(s), \; t \le s < T$, is a solution to the SDE
\be \label{AB4}
dZ_\ve(s) = \frac{-\mu Z(s)}{T-s} \ ds + \sqrt{\ve}\; dW(s),
\ee
with initial condition $Z_\ve(t) = z \in \R$.  Then $\displaystyle{\lim_{s\ra T}}\; Z_\ve(s) = 0$ with probability 1, and if $\mu > 1/2$ then $\displaystyle{\liminf_{s\ra T}}\; Z_\ve(s)/\sqrt{T-s} = -\infty$ with probability 1.
\end{lem}
\begin{proof}  The SDE (\ref{AB4}) is explicitly solvable, whence we find
\be \label{AC4}
Z_\ve(s)  = \left( \frac{T-s}{T-t} \right)^\mu \; z + \sqrt{\ve} \ \int^s_t \ \left( \frac{T-s}{T-s'} \right)^\mu \; dW(s'), \quad t \le s < T.
\ee
Thus $Z_\ve(s)$ is a Gaussian variable with mean of order $(T-s)^\mu$ as $s \ra T$.  We shall assume wlog that $\mu > 1/2$, in which case the variance of $Z_\ve(s)$ is order $T - s$ as $s \ra T$.  Hence the standard deviation of $Z_\ve(s)$ dominates the mean for $s \ra T$.  For $n = 0,1,2....$, let $s_n = T-(T-t)/2^n$, so $t=s_0 < s_1 < s_2 < \cd\cd\cd < T$.  For $t < s < T$ we consider the Martingale $M(s)$ defined by
\[	M(s) = \int^s_t \ (T-s')^{-\mu} \ dW(s'),  \]
which by Doob's inequality satisfies
\[     P \bigg( \sup_{t\le s\le s_n} \ |M(s)| > a\bigg) \le 2^{(2\mu - 1)n} \big/ a^2(2\mu - 1)(T-t)^{2\mu-1}, \quad a>0.  \]
It follows that 
\[ \sum^\infty_{n=1} \  P \bigg( \sup_{t\le s\le s_n} \ |M(s)| > 2^{(\mu - 1/4)n} \bigg) < \infty.  \]
Hence by the Borel-Cantelli lemma $\displaystyle{\limsup_{s\ra T}}\; (T-s)^\mu \; |M(s)| = 0$ with probability 1.  We conclude from (\ref{AC4}) that $\displaystyle{\lim_{s\ra T}}\;Z_\ve(s) = 0$ with probability 1.

We turn to showing that $\displaystyle{\liminf_{s\ra T}}\;Z_\ve(s)\big/ \sqrt{T-s} = -\infty$ with probability 1. For $n=1,2,...$ we define variables $Y_n$ by $Y_n = (T-s_n)^{\mu - 1/2}$ $\big[ M(s_n) - M(s_{n-1})\big]$.  We may write the $Z_\ve(s_n)$ in terms of the $Y_n$ as 
\be \label{AD4}
Z_\ve(s_n)  = \left( \frac{T-s_n}{T-t} \right)^\mu \; z + \sqrt{\ve(T-s_n)} \ \sum^n_{m=1} \; Y_m/2^{(n-m)(\mu - 1/2)}, \quad  n=1,2,\cdots .  \ee
Evidently the $Y_n, \; n \ge 1$, are independent and Gaussian with zero mean and variance var$(Y_n) = \big[ 1-2^{1-2\mu}\big]/(2\mu-1)$.  By the Borel-Cantelli lemma for any $K>0$, one has $Y_n < -K$ for infinitely many $n$, with probability 1.  Thus if in (\ref{AD4}) we were to replace the sum over $1 \le m \le n$ by its dominant term $m=n$, we would have shown that $\displaystyle{\liminf_{n\ra \infty}}\; Z_\ve(s)\big/ \sqrt{T-s_n} = -\infty$ with probability 1.

To take account of the sum in (\ref{AD4}) we need to make a more elaborate argument.  Denoting the sum in (\ref{AD4}) by $\xi_n$ it is easy to see that
\be \label{AE4}
\xi_n = Y_n + \xi_{n-1} \big/ 2^{(\mu - 1/2)}, \quad  n \ge 1,
\ee
where $\xi_0 = 0$.  For $\xi \in \R$, $n\ge 1$, we put
\[	u(\xi,n) = P\big[ \xi_m >a, \ 1 \le m \le n  \  \big| \  \xi_0 = \xi \big],  \]
where the $\xi_n$ are defined by the recurrence (\ref{AE4}).  Setting $\del = 1/2^{(\mu - 1/2)} < 1$, it is easy to see that the $u(\xi,n)$ satisfy the recurrence equation
\be \label{AF4}
u(\xi,n) = \frac 1{\sqrt{2\pi \sig^2}} \ \int^\infty_a d\xi ' \; u(\xi', n-1) \exp \left[ - \big(\xi' - \del\xi\big)^2 \big/ 2\sigma^2 \right] ,  \  \quad n\ge 1,
\ee
where we define $u(\xi,0) = 1, \ \xi \in \R$, and $\sigma^2 = \big[ 1-2^{1-2\mu}\big] / (2\mu - 1)$.  If for $z > 0, \ \hat u(\xi,z)$ is the Laplace transform of $u(\xi,n)$, 
\[	\hat u(\xi,z) = \sum^\infty_{n=0} u(\xi,n)e^{-nz}, \quad  \xi \in \R, \ z > 0,	\]
then we see from (\ref{AF4}) that
\be \label{AG4}
\hat u(\xi,z) = 1 + \frac{e^{-z}}{\sqrt{2\pi\sigma^2}} \int^\infty_a d\xi' \  \hat u(\xi',z) \exp \big[ -(\xi' - \del \xi)^2/2\sigma^2\big], \quad  \xi \in \R, \ z > 0.
\ee
It follows from (\ref{AG4}) that for $\eta > 0$,
\be \label{AH4}
\sup_{\xi > a} \; \big[\hat u(\xi,z)e^{-\eta\xi} \big] \le e^{-\eta a} + e^{-z} \sup_{\xi > a} \; \big[\hat u(\xi,z)e^{-\eta\xi} \big] \sup_{\xi > a} \; h_\eta(\xi),
\ee
where $h_\eta(\xi)$ is given by the expression
\[	h_\eta(\xi) = \frac 1{\sqrt{2\pi \sig^2}} \ \int^\infty_a d\xi '  \exp \left[ \eta (\xi' -\xi) - (\xi' - \del\xi)^2 \big/ 2\sigma^2 \right].   \]
Evidently $\displaystyle{\sup_{\xi > a}}\; h_\eta(\xi) = 1$ if $\eta = 0$.  We shall show that there is an $\eta > 0$ such that $\displaystyle{\sup_{\xi > a}}\; h_\eta(\xi) < 1$.

To see this we shall assume wlog that $a < 0$ and $0 < \eta < 1$.  We choose $\al$ to satisfy $\del < \al < 1$, and  for $\xi > 0$ consider the integral
\begin{multline*}
\int^\infty_{\al\xi} d\xi' \exp\Big[ \eta(\xi' - \xi) - (\xi' - \del \xi)^2 / 2\sig^2 \Big] = \\
\sigma \exp \Big[ -\eta(1-\del)\xi\Big] \int^\infty_K \exp\Big[ \eta \sig \zeta - \zeta^2/2 \Big]d\zeta, 
\end{multline*}
where $K = [\al - \del]\xi/\sigma$.  We have now that
\begin{multline*}	\int^\infty_K \exp\Big[ \eta \sig \zeta - \zeta^2/2 \Big]d\zeta = e^{\eta^2\sigma^2/2} \int^\infty_{K-\eta\sigma} \ e^{-\zeta^2/2} \ d\zeta  \\
\le \exp \big[ \eta^2\sigma^2 + C(K-\eta\sigma)\eta\sigma\big]  \int^\infty_K \ e^{-\zeta^2/2} \ d\zeta , 
\end{multline*}
where we have used (\ref{K4}) and assumed $K-\eta\sigma > 1$.  Taking $C > 1$ and choosing $\al$ so that $(1-\del) > C(\al-\del)$, we conclude from the last 2 inequalities that there exists $\xi_0 > 0$ depending only on $\sigma,\al$, such that
\begin{multline*}	\int^\infty_{\al\xi} d\xi' \exp\Big[ \eta(\xi' - \xi) - (\xi' - \del \xi)^2 / 2\sig^2 \Big]  \\ 
\le    \exp \big[ -\eta\xi\{ (1-\del)-C(\al - \del)\}\big] \int^\infty_{\al\xi} d\xi' \exp\Big[ -(\xi' - \del \xi)^2/2\sig^2 \Big]
\end{multline*}
provided $\xi > \xi_0$.  It easily follows that
\be \label{AI4}
h_\eta(\xi) \le \exp [-\rho\eta\xi], \quad  \xi > \xi_0, \ 0 < \eta < 1,
\ee
where $\rho = \min \big[ (1-\del) - C(\al - \del), \ 1 - \al\big]$.  One can also see that we may choose $\eta > 0$ sufficiently small such that $\displaystyle{\sup_{a<\xi<\xi_0}}\; h_\eta(\xi) < 1$.  Combining this with (\ref{AI4}), we conclude that $\displaystyle{\sup_{\xi > a}}\; h_\eta(\xi) < 1$ for sufficiently small $\eta > 0$.  Now on letting $z \ra 0$ in (\ref{AH4}), we see that
\[	\sum^\infty_{n=1} \ P\Big( \xi_m > a, \ 1 \le m\le n \  \Big| \  \xi_0 = \xi \Big) < \infty.	\]
Hence by the Borel-Cantelli lemma $\displaystyle{\liminf_{n\ra \infty}}\; \xi_n \le a$ with probability 1.  Now (\ref{AD4}) implies that $\displaystyle{\liminf_{n\ra \infty}}\; Z_\ve(s_n) \big/ \sqrt{T-s_n} = -\infty$ with probability 1.
\end{proof}
\begin{theorem}  Let $\la_\ve(\cd, \cd)$ be the optimal controller defined by (\ref{BI3}).  Then the SDE (\ref{N1}) has a unique strong solution $y_\ve(s), \; t \le s < T$, with initial condition $y_\ve(t) = y$,  Furthermore $\displaystyle{\liminf_{s\ra T}}\; y_\ve(s) > x$ with probability 1.
\end{theorem}
\begin{proof}  To show existence and uniqueness of a solution to (\ref{N1}) we argue as in Lemma 3.5.  Thus for $y_0 < y$ let $\tau(y_0) = \inf\{s \ge t : s < T, \ y_\ve(s) = y_0\}$.  Since $\la_\ve(y',s) \ge b(y',s), \  y' \in \R, \ s < T$, it follows that $\displaystyle{\lim_{y_0 \ra -\infty}}\; P(\tau(y_0) < T) = 0$.  Hence by the Lipschitz property of $\la_\ve(y',s)$ for $y' \ge y_0, \ t \le s \le T-\del$, for any $\del > 0$, we obtain a unique strong solution to (\ref{N1}) up to time $T-\del$.  Letting $\del \ra 0$ we get existence and uniqueness in the interval $t \le s < T$.

To show that $\displaystyle{\liminf_{s\ra T}}\; y_\ve(s) > x$ we consider for $y_0 < y$ solutions $y_\ve(s), \ t \le s < T$, of (\ref{N1}) with $y_\ve(t) = y$ such that $\tau(y_0) = T$.  From (\ref{M4}) and the fact that $b(\cd, s)$ is uniformly 
Lipschitz for $t \le s \le T$, we see that there exists $s_0$ with $t\le s_0<T$, and $\mu_0>0$, such that such that
\be \label{AJ4}
dy_\ve(s) \ge \left( b(x,s) + \frac{\mu_0[x-y_\ve(s)]}{T-s} \right)ds + \sqrt{\ve} \  dW(s), \quad  s_0 \le s < T,
\ee
on paths $y_\ve(\cd)$ for which $\tau(y_0) = T$.  It follows then from (\ref{AJ4}) and Lemma 4.2 that on paths $y_\ve(\cd)$ for which  $\tau(y_0) = T$ one has in fact $\displaystyle{\liminf_{s\ra T}}\; y_\ve(s) \ge x$ with probability 1.  Letting $y_0 \ra -\infty$, we conclude that $\displaystyle{\liminf_{s\ra T}}\; y_\ve(s) \ge x$ with probability 1 on all paths $y_\ve(\cd)$ for which $y_\ve(t) = y$.

Next for $\eta > 0$ and $s_0 < T$ let $U_{\eta,s_0} = \big\{ y_\ve(\cd) : y_\ve(t) = y, \ y_\ve(s) \ge x-\eta,  \ s_0 \le s < T\big\}$.  If $\eta$ and $T - s_0$ are sufficiently small it follows from (\ref{M4}) that we may take $\mu_0 > 1/2$ for a path $y_\ve(\cd) \in U_{\eta,s_0}$.  Hence by Lemma 4.2 we have that  $\displaystyle{\limsup_{s\ra T}}\;[ y_\ve(s)-x] \big/ \sqrt{T-s} = +\infty$ with probability 1 for all paths $y_\ve(\cd) \in U_{\eta,s_0}$.  Since $\lim_{s_0\ra T} P(U_{\eta, s_0})=1$, we conclude that $\displaystyle{\limsup_{s\ra T}}\;[ y_\ve(s)-x] \big/ \sqrt{T-s} = +\infty$ with probability 1 on all solutions to (\ref{N1}) with $y_\ve(t) = y$.

For $K > 0$ we define a stopping time $\tau_K$ by $\tau_K = \inf\big\{ s \ge t : s < T, \; y_\ve(s) - x = K\sqrt{T-s} \big\}$.  We have just shown that $P(\tau_K < T) = 1$.  Consider now a solution $y_\ve(s)$ to (\ref{N1}) for $s_1 \le s < T$ with initial condition $y_\ve(s_1) = y_1$.  Now $y_\ve(s) \ge Y_\ve(s), \; s_1 \le s < T$, where $Y_\ve(s)$ is the solution to (\ref{E1}) with $Y_\ve(s_1) = y_1$.  From (\ref{K3}) we conclude that
\be \label{AK4}
\inf_{s_1\le s < T} y_\ve(s) \ge y_1 -C\; \sup_{s_1\le s < T} \Big| \int^s_{s_1} b(y_1,s')ds' +\sqrt{\ve} \; \big[ W(s) - W(s_1)\big] \Big|,
\ee
for some constant $C$.  We take now $s_1 \ge t$ and $y_1 = x + K\sqrt{T-s_1}$ in (\ref{AK4}).  It is clear that there is a constant $K_0 > 0$ such that for $K > K_0$,
\be \label{AL4}
P\left( \inf_{s_1\le s <T} y_\ve(s) \le x \right) \le P \left( \sqrt{\ve} \; \sup_{s_1\le s < T} | W(s) - W(s_1)| >  K\sqrt{T-s_1} \big/ 2 \right) \le 4\ve/K^2.
\ee
Taking $s_1 = \tau_K$ in (\ref{AL4}) we conclude that for $K > K_0$ one has $P\big( \displaystyle{\liminf_{s\ra T}}\; y_\ve(s) \le x \big) \le 4\ve/K^2$.  Letting $K \ra \infty$ yields the result. 
\end{proof}
\begin{corollary} Let $\la_\ve(\cd,\cd)$ be the optimal controller defined by (\ref{BI3}), and $y_\ve(s)$ be the corresponding solution to (\ref{N1}) with initial condition $y_\ve(t) = y$.  Then one has
\end{corollary}
\be \label{AM4}
\lim_{\del \ra 0} \ q_\ve\big(x, y_\ve(T-\del), T-\del\big) = 0 \ \ {\rm with \ probability} \ \ 1.
\ee
\begin{proof}  We use the second inequality of (\ref{Z4}) to obtain an estimate on $q_\ve(x, y, T-\del)$ when $y>x$.  Since $q_\ve\big(x, \cd, T-\del)$ is a positive decreasing function we have that
\be \label{AN4}
q_\ve(x, y,T-\del ) \le C\ve \exp \left[ -(x-y)^2/2\ve\del \right] + C_1(\ve\del)^{1/4}, \ \   x < y < x + (\ve\del)^{1/4}, 
\ee
$$
q_\ve( x, y, T-\del ) \le C_1(\ve\del)^{1/4}, \quad  y > x + (\ve\del)^{1/4}, 
$$
for some constants $C, C_1$.  Now (\ref{AM4}) follows from (\ref{AN4}) and Theorem 4.1. 
\end{proof}

\section{Proof of Theorem 1.2}
The problem of estimating $\pa q_\ve(x, y, t)/\pa y$ is closely related to the problem of estimating certain conditional probabilities.  For $0 < \del < T/2$ we shall consider the conditional probability $P\big(Y_\ve(T-\del) \in U \  |  \ Y_\ve(0) = y, Y_\ve(T) = 0 \big)$, where $Y_\ve(s),  \ 0 \le s \le T$, satisfies the SDE (\ref{E1}) and $U$ is an arbitrary open set.  In the linear approximation $b(y,s) = A(s)y$ the variable  $Y_\ve(T)$ conditioned on $Y_\ve(0)=y$ is Gaussian with mean $\La(T)y$ and variance $\ve\sig^2(T)$, where
$\La(T), \ \sig^2(T)$ are given by the formulas,
\be \label{AQ4}
\La(T) = \exp \left[ \int^T_0 A(s)ds\right], \ \sig^2(T) = \int^T_0 \exp \left[ 2\; \int^T_s A(s')ds'\right]ds.
\ee
The variable 
$Y_\ve(T-\del)$ conditioned on $Y_\ve(0) = y, \; Y_\ve(T) = 0$, is also Gaussian with mean and variance given by the formulas
\begin{eqnarray}  \nn
E\big[ Y_\ve(T-\del) \  |  \ Y_\ve(0) = y, \ Y_\ve(T) = 0 \big] &=& \frac{\La(T-\del)}{\sig^2(T)} \; y \; \int^T_{T-\del} \exp \left[ 2 \; \int^T_s A(s')ds'\right]ds,   \\  \label{AS4} \\
 Var \; \big[ Y_\ve(T-\del) \  |  \ Y_\ve(0) = y, \ Y_\ve(T) = 0 \big] &=& \frac{\ve \sig^2(T-\del)}{\sig^2(T)} \; \int^T_{T-\del} \exp \left[ 2 \; \int^T_s A(s')ds'\right]ds. \nn
\end{eqnarray}
The mean in (\ref{AS4}) is equal to $y_{\rm min}(T-\del)$ where $y_{\rm min}(s), \ 0 \le s \le T$, is the unique minimizer for the functional ${\mathcal F}[y(\cd)]$ of (\ref{R2}) conditioned on $y(0) = y, \ y(T) = 0$.  One easily sees from (\ref{AS4}) that there are positive universal constants $C_1, C_2$ such that
\begin{eqnarray} \label{AT4}
\frac{C_2\del y}{T} &\le& E\left[ Y_\ve(T-\del) \  |  \ Y_\ve(0) = y, \ Y_\ve(T)=0 \right] \le \frac{C_1\del y}{T} ,  
 \\
C_1 \ve \del &\le& \; {\rm Var}\; \left[ Y_\ve(T-\del) \  |  \ Y_\ve(0) = y, \ Y_\ve(T)=0 \right] \le C_2 \ve \del, \nn
\end{eqnarray}
for $y < 0$ provided $0 < \del < T/2, \ AT < 1$. It follows from (\ref{AT4}) that there are positive universal constants $C_3, \ga_3, C_4, \ga_4$ such that
\begin{eqnarray} \nn
P\Big( Y_\ve(T-\del) &<& \frac{C_3\del y}{T} \  \big|  \ Y_\ve(0) = y,  Y_\ve(T)=0 \Big) \le 
\exp \left[-\frac{\ga_3 \del y^2}{\ve T^2}\right], \  y < -T\sqrt{\ve/\del},  \\  \label{AU4}  \\
P\Big( Y_\ve(T-\del) &>& \frac{C_4\del y}{T} \  \big|  \  Y_\ve(0) = y,  Y_\ve(T)=0 \Big) \le 
\exp \left[-\frac{\ga_4 \del y^2}{\ve T^2}\right], \  y < -T\sqrt{\ve/\del}, \nn
\end{eqnarray}
provided $0 < \del < T/2, \ AT < 1$. 

Evidently (\ref{AU4}) proves Theorem 1.2 in the case of  $b(y,\cdot)$
linear in $y\in\R$.
We need to show therefore that (\ref{AU4}) continues to hold for nonlinear $b(\cd,\cd)$ satisfying (\ref{A1}) and $b(0,\cdot)\equiv 0$.  Towards that goal we first observe that in the linear case there are positive universal constants $C_3, \ga_3, C_4, \ga_4$ such that if ${\mathcal F}_{\rm min} = {\mathcal F}[y_{\rm min}(\cd)]$ then
\begin{eqnarray} \label{AV4}
{\mathcal F}[y(\cd)] - {\mathcal F}_{\rm min} &\ge& \ga_3\;\del\;y^2/T^2 \ {\rm if} \ y(T-\del) < C_3\del y /T, \\
{\mathcal F}[y(\cd)] - {\mathcal F}_{\rm min} &\ge& \ga_4\;\del\;y^2/T^2 \ {\rm if} \ y(T-\del) > C_4\del y /T, \nn
\end{eqnarray}
provided $0 < \del < T/2, \ AT < 1, \ y < 0$.  For nonlinear $b(\cd, \cd)$ there is not necessarily a unique minimizer of the functional ${\mathcal F}[y(\cd)]$ subject to $y(0) = y < 0$, $y(T) = 0$.  Nevertheless, if ${\mathcal F}_{\rm min}$ denotes now the minimum of ${\mathcal F}[y(\cd)]$ then (\ref{AV4}) continues to hold.
\begin{lem} Let $b(\cd,\cd)$ satisfy (\ref{A1}) and $b(0,\cdot)\equiv 0$. Assume further that  $y < 0, \; \del < T/2, \ AT < 1$ and ${\mathcal F}_{\rm min}$ is the minimum of the functional ${\mathcal F}[y(\cd)]$ of (\ref{R2}) subject to $y(0) = y,  \ y(T) = 0$.  Then (\ref{AV4}) holds for some positive universal constants $C_3, \ga_3, C_4, \ga_4$, on any path $y(s), 0 \le s \le T$, satisfying $y(0) = y, y(T) =0$.
\end{lem}
\begin{proof} We first show that there are positive universal constants $C_1,C_2$ such that
\be \label{AW4}
C_1 y^2/T \le {\mathcal F}_{\rm min} \le C_2y^2/T.
\ee
The upper bound in (\ref{AW4}) can be obtained by estimating ${\mathcal F}[y(\cd)]$ for the linear path $y(s) = (T-s)y/T$, $0 \le s \le T$.  To get the lower bound we consider a path $y(s), 0 \le s \le T$, satisfying $y(0) = y, y(T) =0$, and write \be \label{AX4} 
\frac{dy}{ds} = b( y(s), s) + f(s) = A(s) y(s) + f(s),
\ee
where $|A(s)| \le A, \ 0 \le s \le T$.  Evidently we see from (\ref{AX4}) that
\[	y = y(0) = - \int^T_0 \exp \left[ - \int^s_0 A(s')ds' \right] f(s) ds.   \]
Since $AT < 1$ we conclude that
\[	|y| \le e \ \int^T_0 |f(s)|ds \ \le \ e\sqrt{T} \ \left[ \int^T_0 |f(s)|^2\;ds\right]^{1/2},  \]
whence we obtain the lower bound in (\ref{AW4}) with $C_1 = 1/2e^2$.

To prove the first inequality in (\ref{AV4}) we consider for $\la > 1$ a path $y_\la(s), \; 0 \le s \le T$, satisfying $y_\la(0) = y$, $y_\la(T) = 0$ and $y_\la(T-\del) = \la \del y/T$.  We derive a second path $y^*_\la$ from $y_\la$ by setting $y^*_\la(s) = 0$, $T-\del < s < T, \; y^*_\la(s) = y_\la(s) - s\la \del y/T(T-\del)$, $0 < s < T-\del$.  Thus $y^*_\la(\cd)$ is continuous and $y^*_\la(0) = y, \; y^*_\la(T) = 0$, whence we must have ${\mathcal F}[y^*_\la(\cd)] \ge {\mathcal F}_{\rm min}$.  We also have that
\begin{multline} \label{AY4}
{\mathcal F}[y_\la(\cd)] = \frac 1 2 \ \int^{T-\del}_0 \ \left[ \frac{d y^*_\la(s)}{ds} + \frac{\la\del y}{T(T-\del)} 
- b( y^*_\la(s) + s\la\del y \; / T(T-\del), s) \right]^2\; ds  \\
+ \frac 1 2 \int^T_{T-\del} \left[ \frac{dy_\la}{ds} - b(y_\la(s), s ) \right]^2\; ds. 
\end{multline}
Arguing as we did to get the lower bound in (\ref{AW4}) we see that
\be \label{AZ4}
\frac 1 2 \ \int^T_{T-\del} \ \left[ \frac{d y_\la(s)}{ds} - b(y_\la(s),s)\right]^2\; ds \ge  \frac{\la^2\del y^2}{2e^2T^2} \  .
\ee
The first term on the RHS of (\ref{AY4}) is bounded below by
\be \label{BA4}
{\mathcal F}[y^*_\la(\cd)] - \frac{2\la\del|y|}{T(T-\del)} \ \int^{T-\del}_0 \left| \frac{dy^*_\la(s)}{ds} - b(y^*_\la(s), s ) \right|\; ds,
\ee
where we have used the fact that $AT < 1$.  It follows then from (\ref{AZ4}), (\ref{BA4}) that
\be \label{BB4}
{\mathcal F}[y_\la(\cd)]  \ge {\mathcal F}[y^*_\la(\cd)] - \frac{2\sqrt{2} \la\del|y|}{T\sqrt{T-\del}} {\mathcal F}[y^*_\la(\cd)]^{1/2}+\frac{\la^2\del y^2}{2 e^2T^2}.
\ee 
Observe now from (\ref{AW4}), (\ref{AY4}), (\ref{AZ4}) that there is a universal constant $C_3$ such that if $\la\del/T > C_3$ then ${\mathcal F}[y_\la(\cd)]  - {\mathcal F}_{\rm min} \ge \la^2\del y^2/2e^2 T^2$.  Suppose now that $\la \del/T < C_3$.  If  ${\mathcal F}[y^*_\la(\cd)] \ge [2C_2 + 64 C^2_3]y^2/T$ it follows from (\ref{BB4}) that ${\mathcal F}[y_\la(\cd)]  - {\mathcal F}_{\rm min} \ge \la^2\del y^2/2e^2 T^2$.  On the other hand if ${\mathcal F}[y^*_\la(\cd)] \le [2C_2 + 64 C^2_3]y^2/T$ we see again from (\ref{BB4}) that ${\mathcal F}[y_\la(\cd)]  - {\mathcal F}_{\rm min} \ge \la^2\del y^2/4e^2 T^2$ if $\la > \la_0 \ge 1$ for some universal $\la_0$.  We have proven the first inequality of (\ref{AV4}).

We turn to the proof of the second inequality in (\ref{AV4}).  Let $y_1(\cd)$ be a trajectory satisfying $y_1(0) = y, \ y_1(T) = 0$ and set $\tau = \inf\{ s \ge 0 : y_1(s) = 0\}$.  Suppose now that $\tau \le T-\del$.  From (\ref{AW4})  one has that ${\mathcal F}[y_1(\cd)] \ge C_1y^2/\tau$, and so the second inequality of (\ref{AV4}) follows if $\tau < C_1T/2C_2$.  We assume therefore that $C_1T/2C_2 < \tau \le T-\del$.  Let $y_{\rm min}(\cd)$ be a minimizing path for the functional ${\mathcal F}[y(\cd)]$ subject to the conditions $y(0) = y, \; y(s) = 0, \; \tau \le s \le T$.  Then ${\mathcal F}[y_1(\cd)] \ge {\mathcal F}[y_{\rm min}(\cd)]$.  From (\ref{X2}) we see that there are positive universal constants $C_3, C_4$ such that
\be \label{BC4}
\frac{C_3|y|}{T} \le \frac{dy_{\rm min}(s)}{ds} - b(y_{\rm min}(s), s) \le \frac{C_4|y|}{T}, \quad  0 \le s \le \tau.
\ee
Since $AT < 1$ we conclude from (\ref{BC4}) that
\be \label{BD4}
{eC_4(\tau - s)y}/T \le y_{\rm min}(s) \le {C_3(\tau - s)y} /{eT}, \quad  0 \le s \le \tau.
\ee
It is clear that there is a positive universal constant $\ve_0$ such that for $0 < \ve < \ve_0$ we may define a path $y_\ve(\cd)$ as follows:  $y_\ve(s) = y_{\rm min}(s), \ 0 \le s \le \tau - \ve \del$; $y_\ve(s) = (T-s) y_{\rm min}(\tau - \ve\del)/(T-\tau + \ve\del)$, $\tau - \ve\del \le s \le T$.  Since $y_\ve(\cd)$ is continuous, $y_\ve(0) = y, \; y_\ve(T)=0$, we have that ${\mathcal F}[y_\ve(\cd)] \ge {\mathcal F}_{\rm min}$.  From (\ref{BC4}), (\ref{BD4}) we also have that
\be \label{BE4}
{\mathcal F}[y_{\rm min}(\cd)] - {\mathcal F}[y_\ve(\cd)] \ge {\ve\del C^2_3y^2}/{2T^2} - {e^2C^2_4\ve^2\del y^2}/{2(1+\ve)T^2}\; ,
\ee
where we have used the fact that $\tau \le T - \del$.  Evidently the second inequality of (\ref{AV4}) follows from (\ref{BE4}) by choosing $\ve = {\rm min}[1, C_3/2eC_4]^2$.

To complete the proof of the second inequality of (\ref{AV4}) we need to consider the case $T -\del \le \tau \le T$.  It is evident that if $C\del y/T < y_1(T-\del) \le 0$ for sufficiently small universal $C>0$ we may repeat the argument of the previous paragraph.  Hence the result follows in all cases. 
\end{proof}
We begin the proof of (\ref{AU4}) by sharpening the estimate (\ref{AA4}) on the Green's function $G(y,x,t,T)$ defined by (\ref{D1}).
\begin{lem}  Suppose $b(\cd,\cd)$ satisfies (\ref{A1}) and in addition $b(0,\cdot)\equiv 0$. Then there are universal constants $C,\del > 0$ such that the Green's function $G$ defined by (\ref{D1}) satisfies the inequalities
\be \label{BH4}
G(y,0,0,T) \le \frac 1{\sqrt{2\pi\ve T}} \exp \left[ \frac{-y^2}{2\ve T(1 +  CAT)} + CAT\right],
\ee
\be \label{BI4}
G(y,0,0,T) \ge \frac 1{\sqrt{2\pi\ve T}} \exp \left[ \frac{-y^2(1 + CAT)} {2\ve T} - CAT\right],
\ee
provided $AT \le \del$.
\end{lem}
\begin{proof}   We shall first prove (\ref{BI4}).  Suppose that we have shown that
\be \label{BJ4}
G(y,0,t,T) \ge \frac 1{\sqrt{2\pi\ve (T-t)}} \exp \left[ \frac{-y^2}{2\ve (T-t)} \left\{ 1 + CA(T-t)\right\} -  CA(T-t) \right],
\ee
for $T-t = T/2^N$, where $N$ is some integer $N \ge 1$.  We shall show that for sufficiently large universal constant $C > 0$ then (\ref{BJ4}) also holds for $T-t = T/2^{N-1}$.  The inequality (\ref{BI4}) will then follow by induction if we can prove (\ref{BJ4}) holds as $T-t \ra 0$.

Defining $t_N$ by $T-t_N = T/2^N, N=0,1,2,...$ we see in a similar way to how we derived (\ref{V4}) that
\be \label{BK4}
G(y,0, t_{N-1},T)^{1-\al} \ge \int^\infty_{-\infty} \ \frac 1{\sqrt{2\pi\ve T/2^N}} \exp \left[ - \frac{(y-z)^2}{2\ve T/2^N}\right] F_N(y,z) G(z,0,t_N, T)^{1-\al}\; dz,
\ee
where $F_N(y,z)$ is given by the formula,
\be \label{BN4}
F_N(y,z) = \exp \left\{ - \frac{A^2(1-\al)}{6\al\ve} \; \frac T{2^N}(y^2 + zy + z^2) - \frac{A^2(1-\al)}{12\al} \frac {T^2}{2^{2N}} \right\} .
\ee
Assuming now that we may bound $G(z,0,t_N,T)$ according to (\ref{BJ4}), then the RHS of (\ref{BK4}) becomes a Gaussian integral which we can evaluate.  Taking $\al = AT/2^N$ in (\ref{BK4}) and $C_N$ to be the constant $C$ in (\ref{BJ4}) when $t=t_N$, we see that it is possible to take $C_{N-1} = 5C_N/8 + 2$ provided $N \ge 1$ and $\del \le 1$.  We conclude therefore that
\be \label{BL4}
C_0 = \frac{16}{3} \left[ 1 - \left( \frac 5 8 \right)^N \right] + \left( \frac 5 8 \right)^N \; C_N, \quad N \ge 1
\ee
The inequality (\ref{BI4}) follows from (\ref{BL4}) if we can show that $\displaystyle{\lim_{N\ra \infty}} 5^NC_N/8^N=0$.   We can do this by the same method we used to derive (\ref{AA4}).

We shall show that the inequality (\ref{BI4}) holds with a constant $C=C(AT)$ which can diverge as $T\ra 0$, but in a mild in fact logarithmic way.  As in (\ref{V4}) we write
\begin{multline} \label{BM4}
G(y,0,0,T)^{1-\alpha} \ge \int^\eta_{-\eta} \frac 1{\sqrt{2\pi\ve(T-\De)}} \ \exp \left[ - \frac{(y-z)^2}{2\ve(T-\De)} \right]
\\
 F_0(y,z) G_{D,\eta} \left( z,0,T-\De, T\right)^{1-\al}\; dz,  
 \end{multline}
where $G_{D,\eta}$ is the Dirichlet Green's function for the equation (\ref{B1})  on the interval $[-\eta,\eta]$.  The function $F_0$ is given by the formula (\ref{BN4}) when $N=0$, and we take $\al = AT$.  As in Lemma 3.4 we use perturbation theory to estimate $G_{D,\eta}$.  In order for the perturbation expansion to converge we need that
\be \label{BNN4}
\eta = K \sqrt{\ve\De}, \ \ (A\eta)^2\De = \nu \ve, 
\ee
where $K > > 1$ and $\nu < < 1$.  In that case there is the lower bound
\begin{multline} \label{BO4}
G_{D,\eta} (z, 0, T-\De, T) \ge \frac 1{\sqrt{2\pi\ve\De}} \ \bigg[ \exp \left\{ \frac{-z^2}{2\ve\De}\right\} -  
C_1\;e^{-K^2/4}  \\
- C_2(\rho)\nu^{1/2} \exp \left\{ \frac{-z^2}{2\ve(1+\rho)\De}\right\} \bigg], \quad  |z| < \eta, 
\end{multline}
where $C_1$ is a universal constant, $\rho > 0$ can be arbitrary and $C_2(\rho)$ is a constant depending only on $\rho$.  We shall substitute the RHS of (\ref{BO4}) into (\ref{BM4}), choosing $\De/T, \; K$ and $\nu$ to be powers of $AT$, in order to obtain a lower bound as in (\ref{BI4}).

Consider the situation when we approximate $G_{D,\eta}$ by the first term on the RHS of (\ref{BO4}).  From (\ref{BM4}) we have that
 \begin{multline}  \label{BP4}
G(y,0,0,T)^{1-\alpha} \ge \inf_{|z| < \eta}  \ \left\{ \exp \left[ - \frac{(y-z)^2}{2\ve(T-\De)} \right] F_0(y,z)\right\}
\\
 \frac 1{\sqrt{2\pi\ve(T-\De)}} \int^\eta_{-\eta} G_{D,\eta} \left( z,0,T-\De, T\right)^{1-\al}\; dz. 
\end{multline}
Observe now that
\begin{multline} \label{BQ4}
\left\{  \frac 1{\sqrt{2\pi\ve(T-\De)}} \int^\eta_{-\eta} \frac 1{{(2\pi\ve\De)}^{(1-\al)/2}} \  \exp\left[- \frac{z^2(1-\al)}{2\ve\De}\right] dz \right\}^{1/(1-\al)}
\\
	\ge \frac 1{\sqrt{2\pi\ve T}} \left[ 1 - e^{-K^2/4} \right] \exp \left[ -C k_0 AT | \log(AT)| \right], 
\end{multline}
for some universal constant $C$, provided we choose $\De/T = (AT)^{k_0}$ with $k_0 > 1$ and $AT \le 1/2$.  From (\ref{BQ4}) it is clear that it is sufficient to choose $K = (AT)^{-k_1}$ for any $k_1 > 0$, whence (\ref{BM4}) implies that $\nu^{1/2} = (AT)^{k_0 + 1 - k_1}$.  If we now use the inequality
\[	{2\eta|y|}/{\ve T} \le (AT)^{k_0/2 - k_1} \left[ {y^2}/{\ve T} + 1 \right],   \]
and choose $k_0 > 2k_1 + 2$, we conclude from (\ref{BP4}), (\ref{BQ4}) that (\ref{BI4}) holds with $C = C'|\log (AT)|$ for some universal constant $C'$.  We may easily extend this argument to apply to the actual lower bound (\ref{BO4}) on $G_{D,\eta}$ by using the inequality
\be \label {BQQ4}
\max[a-b,0]^{1-\al} \ge (a-b), \quad  a,b > 0,  \ a < 1.   
\ee
Returning now to (\ref{BL4}), it follows that we may take $C_N = O(N)$ whence $\lim_{N\ra 0} 5^N C_N/8^N = 0$.  We have therefore show that (\ref{BI4}) holds for some universal constant $C > 0$ provided $AT < \del$ where $\del$ is also universal.

To prove (\ref{BH4}) we use a similar method as in the proof of the lower bound.  Suppose we have shown that
\be \label{BR4}
G(y,0,t,T) \le \frac 1 {\sqrt{2\pi \ve(T-t)}} \exp \left[ \frac{-y^2}{2\ve(T-t)[1+CA(T-t)]} + CA(T-t) \right],
\ee
for $T-t = T/2^N$ where $N$ is some integer $N \ge 1$.  We shall show that for sufficiently large universal constant $C > 0$, the inequality (\ref{BR4}) also holds for $T-t = T/2^{N-1}$.  Analogously to (\ref{BK4}) there is the inequality
\begin{multline} \label{BS4}
G(y,0,t_{N-1},T)^{1+\alpha} \le \int^\infty_{-\infty} \frac 1{\sqrt{2\pi\ve T/2^N}} \ \exp \left[ - \frac{(y-z)^2}{2\ve T/2^N} \right]
\\
  F_N(y,z) G \left( z,0,t_N, T\right)^{1+\al}\; dz,  
\end{multline}
where $F_N(y,z)$ is given by (\ref{S4}) with $g \equiv 0, x = 0, \; t= t_{N-1}$ and $T$ is replaced by $t_N$.  Using the fact that one knows the pdf of $\displaystyle{\sup_{t\le s \le T}\; W(s-t)}$ we see that $F_N(y,z)$ is bounded above by
\be \label{BT4}
F_N(y,z) \le \exp \left\{ \frac{A^2(1+\al)}{3\al\ve} \ \frac T{2^N} \ (y^2 + zy + z^2) + \frac{K_0 A^2(1+\al)}{\al} \; \frac{T^2}{2^{2N}} \right\} 
\ee
for a universal constant $K_0 > 0$, where we are assuming $\al = AT/2^N < \del$ and $\del$ is a sufficiently small universal constant.  Letting $C_N$ be the constant $C$ in (\ref{BR4}) when $t=t_N$, we see from (\ref{BS4}), (\ref{BT4}) that it is possible to take $C_{N-1} = 2C_N/3 + K_0 + 4$, $N \ge 1$, provided $AT < \del$ and $\del$ is sufficiently small.  Arguing as before then, in order to complete the proof of (\ref{BH4}) we need to show that $\displaystyle{\lim_{N\ra \infty}}\; 2^NC_N/3^N = 0$.

To do this we show that (\ref{BH4}) holds with a constant $C = C(AT)$ which can diverge as $T \ra 0$ but only in a logarithmic way.  We use the inequality
\begin{multline} \label{BU4}
G(y,0,0,T)^{1+\alpha} \le \int^\infty_{-\infty} \frac 1{\sqrt{2\pi\ve (T-\De)}} \ \exp \left[ - \frac{(y-z)^2}{2\ve (T-\De)} \right]
\\
F_0(y,z) G \left( z,0, T-\De, T \right)^{1+\al}\; dz,  
\end{multline}
where $F_0$ is given by the RHS of (\ref{BT4}) when $N=0$.  

Choosing $\eta,\nu$ as in (\ref{BNN4}) we see by perturbation theory that there is an upper bound
\be \label{BW4}
G_{D,\eta} ( z,0, T-\De, T) \le  \frac 1{\sqrt{2\pi\ve\De}} \left[ \exp \left\{ - \frac{z^2}{2\ve \De}\right\} + C_2(\rho)\nu^{1/2} \exp \left\{ - \frac{z^2}{2\ve (1+\rho) \De}\right\}  \right], \quad |z| < \eta,
\ee
analogous to the lower bound (\ref{BO4}).  Suppose now that $0 < z < \eta/2$.  Then
\be \label{BX4}
G( z,0, T-\De, T) = G_{D,\eta} ( z,0, T-\De, T) + \int^T_{T-\De} dt \  \rho(t)  \ G_{D,\eta} ( \eta/2, 0, t, T) ,
\ee
where $\rho(t)$ is the density of the hitting time at $\eta/2$ for paths of the diffusion $Y_\ve(\cd)$ satisfying (\ref{E1}) with $Y_\ve(T-\De) = z$, which exit the interval $[0,\eta]$ through $\eta$ before time $T$.  Since $|z| < \eta/2$, it is evident that
\be \label{BY4}
\int^T_{T-\De} \rho(t) dt \le 1 - \int^\eta_{-\eta} G_{D,\eta}( z,z',T-\De, T) \ dz' \le C_1 e^{-K^2/16} + C_2\; \nu^{1/2},
\ee
for universal constants $C_1, C_2$.  One can also see from (\ref{BW4}) on replacing $T-\De$ by $t > T-\De$ that
\be \label{BZ4}
G_{D,\eta} ( \eta/2, 0, t, T) \le \frac{C_3}{\sqrt{2\pi\ve\De}} \; e^{-K^2/16} , \quad  T-\De < t < T,
\ee
for some universal constant $C_3$.  Substituting the RHS of (\ref{BY4}), (\ref{BZ4}) into the RHS of (\ref{BX4}) we conclude from (\ref{BW4}) that
\begin{multline} \label{CA4}
G( z,0, T-\De, T) \le  \frac 1{\sqrt{2\pi\ve\De}} \Bigg[ \exp \left\{ - \frac{z^2}{2\ve \De}\right\} + C_4\;e^{-K^2/16} \\
+ C_2(\rho)\nu^{1/2} \exp \left\{ - \frac{z^2}{2\ve (1+\rho) \De}\right\}  \Bigg], \quad  |z| < \eta/2.
\end{multline}
We may estimate $G(z,0,T-\De, T)$ similarly for $|z| > \eta/2$.  Thus we have
\be \label{CB4}
G(z,0,T-\De, T) = \int^T_{T-\De} dt \ \rho(t) G(\eta/2, 0, t, T), \quad z > \eta/2,
\ee
where again $\rho(\cd)$ is the hitting time density at $\eta/2$.  Evidently we have that
\be \label{CC4}
\int^T_{T-\De}  \rho(t) dt = P\Big( \inf_{T-\De < t < T}\;Y_\ve(t) < \eta/2 \ \big| \  Y_\ve(T-\De) = z \Big).
\ee
It is easy to bound the RHS of (\ref{CC4}) by using the inequality $b(y,s) \ge - Ay, \ y > 0$, in (\ref{E1}) and estimating the probability on the RHS of (\ref{CC4}) for the corresponding Gaussian process.  Assuming that $A\De < 1/10$ and $z > 2\eta$ we have that
\be \label{CD4}
\int^T_{T-\De}  \rho(t) dt \le P\Big( \inf_{0 < t < \De }\; \int^t_0 e^{As} dW(s) < -z/2\sqrt{\ve} \Big),
\ee
where $W(\cd)$ is Brownian motion.  We may estimate the RHS of (\ref{CD4}) by using the fact that
\[  \exp \left[ \la \  \int^t_0 e^{As} dW(s) - \la^2 \left[ e^{2At} - 1\right] \big/ 4A \right]    \]
is a Martingale for any $\la \in \R$.  We conclude that
\be \label{CE4}
\int^T_{T-\De} \rho(t)dt  \le \exp \Big[ -z^2/16\ve \De \Big], \quad  z > 2\eta.
\ee
From  (\ref{CA4}) and  (\ref{CE4}) applied to (\ref{CB4}) we can see now that there is a universal constant $C_5$ such that
\be \label{CF4}
G(z,0,T-\De,T) \le \frac{C_5}{\sqrt{2\pi\ve\De}} \exp \left[ \frac{-z^2}{2C_5\ve \De}\right], \quad  |z| > \eta/2.
\ee
The estimates (\ref{CA4}), (\ref{CF4}) may be substituted into the RHS of (\ref{BU4}) to obtain the inequality
\begin{multline} \label{CG4}
G(y,0,0,T)^{1+\al} \le \\
 \sup_{|z| < \eta/2} \left\{ \exp \left[- \frac{(y-z)^2}{2\ve (T-\De)}\right] F_0(y,z) \right\}
	\frac 1{\sqrt{2\pi\ve(T-\De)}} \int^{\eta/2}_{-\eta/2} G\big( z,0,T-\De,T\big)^{1+\al} \ dz  \  + \\
\exp \big[ -K^2/C_6] \int^\infty_{-\infty} \frac 1{\sqrt{2\pi\ve(T-\De)}}  \exp \left[- \frac{(y-z)^2}{2\ve (T-\De)}\right] F_0(y,z) 
\frac {C_6}{{(2\pi\ve\De)}^{(1+\al)/2}} \exp \left[ -\frac{z^2}{2C_6 \ve \De}\right] dz ,
\end{multline}
where $C_6$ is a universal constant.  The second term on the RHS of (\ref{CG4}) is a Gaussian integral and so can be explicitly evaluated.  To estimate the first term we use (\ref{CA4}) and the inequality
\[  (a+b)^{1+\al} \le a^{1+\al} + 2^\al(1+\al)a^\al b + 2^{1+\al} b^{1+\al},  \quad a,b > 0,	\]
in the integration over the interval $[-\eta/2, \eta/2]$.  One sees then from (\ref{CG4}) that (\ref{BH4}) holds for a constant $C = C' |\log (AT)|$ where $C'$ is universal.  Hence as for the lower bound we may conclude that (\ref{BH4}) holds for some universal $C$ provided $AT < \del$ with $\del > 0$ also universal. 
\end{proof}
We can use the methodology of Lemma 5.2 to obtain similar estimates on $G(y, \xi, 0, T)$ for all $\xi \in \R$.  To motivate the estimates we shall obtain, consider the linear case $b(y,s) = A(s)y$ for which
\[	G(y,\xi,0,T) = \frac 1 {\sqrt{2\pi\ve\sig^2(T)}} \exp \left[ - \frac{(\xi - \La(T)y)^2}{2\ve\sig^2(T)} \right] , \]
where $\La(T), \sig^2(T)$ are as in (\ref{AQ4}).  Observe now that
\[	y - \xi/\La(T) = y + \int^T_0 \; b(\xi, s)ds - \xi + O[(AT)^2]\xi .  \]
It follows that provided $AT \le 1$ there is a universal constant $C > 0$ such that
\[
[y - \xi/\La(T)]^2 \le \left[ y + \int^T_0 b(\xi, s) ds - \xi \right]^2 (1 + CAT) + C(AT)^3 \xi^2.
\]

\begin{lem} Suppose $b(\cd, \cd)$ satisfies (\ref{A1}) and in addition $b(0,\cd)\equiv 0$.  Then there are universal constants $\del, C>0$ such that the Green's function $G$ defined by (\ref{D1}) satisfies the inequalities,
\be \label{CH4}
G(y,\xi,0,T) \le \frac 1{ \sqrt{2\pi\ve T}} \exp \left[- \frac{\{ y+\int^T_0 b(\xi,s)ds-\xi\}^2}{2\ve T(1+ CAT)} + \frac{C(AT)^3\xi^2}{2\ve T} + CAT \right], 
\ee
\be \label{CI4}
G(y,\xi,0,T) \ge \frac 1{ \sqrt{2\pi\ve T}} \exp \left[ -\frac{\{ y+\int^T_0 b(\xi,s)ds-\xi\}^2}{2\ve T}(1+ CAT) - \frac{C(AT)^3\xi^2}{2\ve T} - CAT \right], 
\ee
provided $AT \le \del$.
\end{lem}
\begin{proof}  We proceed as in Lemma 5.2.  To establish (\ref{CI4}) we suppose we have already shown that
\begin{multline}  \label{CJ4}
G(y,\xi, t, T) \ge \frac 1 {\sqrt{2\pi\ve(T-t)}} \exp \bigg[ - \frac{\{ y + \int^T_t b(\xi,s)ds - \xi\}^2}{2\ve(T-t)}
\{ 1 + CA(T-t)\}  \\
- \frac{C[A(T-t)]^3\xi^2}{2\ve(T-t)} - CA(T-t) \bigg]  
\end{multline}
for $T-t = T/2^N$, where $N$ is some integer $N \ge 1$.  We shall show that for sufficiently large constant $C>0$ then (\ref{CJ4}) also holds for $T-t = T/2^{N-1}$.
Using (\ref{B4}) with $x = \xi$ we may obtain an inequality analogous to (\ref{BK4}).  Thus on setting $T-t_N = T/2^N, \ N = 0,1,2,...,$ it follows from (\ref{C4}) that
\begin{multline}  \label{CK4}
G( y, \xi, t_{N-1}, T)^{1-\al} \ge \int^\infty_{-\infty} \frac 1{\sqrt{2\pi\ve T/2^N}} \exp \left[- \frac{(y-z)^2}{2\ve T/2^N} \right] F_N(y,z)  \\
G \left (z + \int^{t_N}_{t_{N-1}} b(\xi,s)ds,  \ \xi, \  t_N,  \ T \right)^{1-\al} \ dz,   
\end{multline}
where similarly to (\ref{BN4}) one may take
\begin{multline} \label{CL4}
- \log F_N(y,z) = \frac{A^2(1-\al)}{3\al\ve} \ \frac T{2^N} \left[ (y-\xi)^2 + (y - \xi)(z-\xi)+(z-\xi)^2 \right] \\
+ \frac{A^2(1-\al)}{\al\ve}  \int^{t_N}_{t_{N-1}} ds \left\{  \int^s_{t_{N-1}} b(\xi,s')ds' \right\}^2 + \frac{A^2(1-\al)}{12\al} \; \frac{T^2}{2^{2N}}. 
\end{multline}
We change the variable $z$ of integration in (\ref{CK4}) to $z'$ where
\be \label{CNN4}
z' = z + \int^T_{t_{N-1}} b(\xi, s)ds - \xi.
\ee
From (\ref{CL4}) we see that
\begin{multline} \label{CM4}
-\log F_N(y,z)  \le \frac{A^2(1-\al)}{\al\ve} \; \frac T{2^N} \left[ \left\{ y + \int^T_{t_{N-1}} b(\xi,s)ds-\xi \right\}^2 + z'^2 \right] \\
+ \frac{3A^2(1-\al)}{\al\ve} \; \frac T{2^N} \left[ \frac{AT\xi}{2^N}\right]^2 + \frac{A^2(1-\al)}{12\al} \ \frac{T^2}{2^{2N}} \ . 
\end{multline}
Using the variable $z'$ of (\ref{CNN4}) and (\ref{CM4}) we may argue as in Lemma 5.2  that (\ref{CJ4}) holds for $t=t_{N-1}$ with constant $C_{N-1} = 5C_N/8 + K$ for some universal constant $K$, where $C_N$ is the constant in (\ref{CJ4}) when $t = t_N$.  Thus we have established (\ref{CI4}) provided we can show that $\lim_{N\ra\infty} \; 5^NC_N/8^N = 0$.

As in Lemma 5.2 we shall complete the proof of (\ref{CI4}) by showing that it holds with a constant $C = C(AT)$ which diverges logarithmically as $T \ra 0$.  To see this we observe as in (\ref{CK4}) that 
\begin{multline} \label{CN4}
G( y, \xi, 0, T)^{1-\al} \ge \int^\infty_{-\infty} \frac 1{\sqrt{2\pi\ve (T-\De)}} \exp \left[ -\frac{\big\{y+\int^T_0 b(\xi,s)ds-\xi-z\big\}^2}{2\ve (T-\De)}\right] F(y,z) \\
 G \left (z +\xi- \int^{T}_{T-\De} b(\xi,s)ds,  \ \xi, \ T-\De, \ T \right)^{1-\al} \ dz,   
 \end{multline}
where as in (\ref{CM4}) we may take $F(y,z)$ to be given by
\begin{multline} \label{CO4}
-\log F(y,z) = \frac {A^2(1-\al)}{\al \ve} T \left[ \left\{ y + \int^T_0 b(\xi, s)ds - \xi\right\}^2 + z^2 \right] \\
+ \frac{3A^2(1-\al)}{\al\ve} \; T[AT\xi]^2  + \frac{A^2(1-\al)}{12\al} T^2. 
\end{multline}
Let $F(\cdot,\cdot)$ be the function defined from (\ref{K1}). Then the function 
\[	v(z,t) = G\left( z +F(\xi,t) ,  \ \xi, \ t, \ T\right), \quad t < T,  \]
satisfies the terminal value problem
\begin{eqnarray} \label{CP4}
0 &=& \frac {\pa v}{\pa t} + \big[ b(z + F(\xi,t), t) - b(F(\xi,t),t) \big] \frac {\pa v}{\pa z}  + \frac \ve 2 \; \frac {\pa^2 v}{\pa z^2} , \quad  t < T,  \ z\in\R \\
 \del(z-\xi)&=& \lim_{t\ra T} v(z,t) ,   \quad z\in\R. \nn
\end{eqnarray}
From (\ref{CP4}) we see that we may proceed now exactly as in Lemma 5.2 by replacing the Green's function on the RHS of (\ref{CN4}) by the solution to (\ref{CP4}) on the interval $|z| < \eta$ with Dirichlet boundary conditions on $|z|=\eta$. Using the fact that
$$|F(\xi,t)-\{\xi - \int^T_t \; b(\xi,s)ds\}|\le C[A(T-t)]^2|\xi|, $$
for some universal constant $C$, we conclude that  the inequality (\ref{CI4}) holds.  The proof of the upper bound (\ref{CH4}) on the Green's function is obtained in a similar way, following the argument of Lemma 5.2. 
\end{proof}
\begin{corollary} Suppose $b(\cd,\cd)$ satisfies (\ref{A1}) and $b(0,\cdot)\equiv 0$.  Then there exist positive universal constants $\eta,C_3,\ga_3,C_4,\ga_4$ such that (\ref{AU4}) holds provided $AT \le \eta$ and $\del = T/2$.
\end{corollary}
\begin{proof}  To show the first inequality in (\ref{AU4}) we consider
\begin{multline*}
P\left( Y_\ve(T/2) < C_3y/2 \  \big|  \ Y_\ve(0) = y, \ Y_\ve(T) = 0 \right) \\
= G(y,0,0,T)^{-1} \int^{C_3y/2}_{-\infty} d\xi \; G(y, \xi, 0, T/2) \  G(\xi, 0, T/2, T).
\end{multline*}
It is easy to see now by using Lemma 5.2 how to bound $G(y,0,0,T)$ from below and $G(\xi, 0, T/2, T)$ from above. Using also Lemma 5.3 to bound $G(y,\xi, 0, T/2)$ from above,  we conclude that the first inequality in (\ref{AU4}) holds for $\del = T/2$ provided $\eta > 0$ is sufficiently small.
To show the second inequality of (\ref{AU4}) we write
\begin{multline*}
P\left( Y_\ve(T/2) >C_4\;y/2 \  \big| \  Y_\ve(0) = y, \ Y_\ve(T) = 0 \right) \\
= G(y,0,0,T)^{-1} \int^\infty_{C_4\;y/2} G(y, \xi, 0, T/2)  \ G(\xi, 0, T/2, T),
\end{multline*}
and argue as in the previous paragraph. 
\end{proof}
In order to show that (\ref{AU4}) continues to hold when $\del/T < < 1/2$ we need to obtain some further estimates on Green's functions.  Towards that goal we strengthen Corollary 5.1 as follows:
\begin{lem} Suppose $b(\cd, \cd)$ satisfies (\ref{A1}) and $b(0,\cdot)\equiv 0$.  Then there exist positive universal constants $\eta, C_1,C_2$ such that if $AT \le \eta$, 
\be \label{CQ4}
P\left( \sup_{0\le s\le T} | Y_\ve(s) | > \rho \  \big| \  Y_\ve(0) = y, \ Y_\ve(T) = 0 \right) \le \exp \left[ - C_1\rho^2/2\ve T\right],
\ee
provided $|\rho| \ge C_2\Big[ |y| + \sqrt{\ve T} \Big]$.
\end{lem}
\begin{proof}  We do a dyadic decomposition of the interval $0 \le s \le T$.  Thus let $S_n, n = 0, 1,2,...,$ be defined by  $S_n = \left\{  {jT}/{2^n} : 0 \le j \le 2^n \right\}$.
It is evident from the continuity of $Y_\ve(\cd)$ that
\begin{multline} \label{CR4}
P\left( \sup_{0\le s\le T}| Y_\ve(s) | > \rho  \ \big| \  Y_\ve(0) = y,  \ Y_\ve(T) = 0 \right) 
\le \\ 
 \sum^\infty_{n=1} P\bigg( \sup_{s\in S_n-S_{n-1}} | Y_\ve(s)| > \rho(1-\mu^{n+1}), \ \sup_{s\in S_{n-1}} |Y_\ve(s)|  
\le \rho(1 - \mu^n) \  \big| \ Y_\ve(0) = y, \ Y_\ve(T) = 0 \bigg), 
\end{multline}
provided $\mu \in (0,1)$ satisfies $\rho(1-\mu) > |y|$.  Observe next that
\begin{multline} \label{CS4}
P\left( \sup_{s\in S_n-S_{n-1}} | Y_\ve(s)| > \rho(1-\mu^{n+1}), \ \sup_{s\in S_{n-1}} |Y_\ve(s)| \le \rho(1 - \mu^n) \  \big| \ Y_\ve(0) = y, \ Y_\ve(T) = 0 \right), 
\\
\le \sum_{s \in S_n - S_{n-1}} P\left( |Y_\ve(s) - Y_\ve(s + T/2^m)| > \rho\mu^n(1-\mu) \  \big| \ Y_\ve(0)  = y,   \ Y_\ve(T) = 0 \right). 
\end{multline}
The probability in the sum on the RHS of (\ref{CS4}) can be expressed in terms of the Green's function (\ref{D1}) as follows:
\begin{multline} \label{CT4}
P \left( |Y_\ve(s) - Y_\ve(s + T/2^n)| > \rho\mu^n(1-\mu) \  \big| \ Y_\ve(0)  = y,  Y_\ve(T) = 0 \right)= \\
G(y,0,0,T)^{-1} \int^\infty_{-\infty} \int^\infty_{-\infty}  d\xi  \ d\zeta \  G(y,\xi,0,s) 
H \left(|\xi - \zeta| - \rho\mu^n(1-\mu) \right) \\
G(\xi, \zeta, s, s + T/2^N) \ G(\zeta, 0, s+ T/2^n, T),  
\end{multline}
where $H(z), \ z \in \R$, is the Heaviside function.  We may estimate the integral on the RHS of (\ref{CT4}) by using Lemma 5.2 and Lemma 5.3.  We first consider the integral with respect to $\xi$ in (\ref{CT4}) for a fixed $\xi \in \R$.  From Lemma 5.3 we have that
\begin{multline} \label{CU4}
G(y, \xi,0,s) \  G(\xi, \zeta, s, s+T/2^n) \le   \frac 1{2\pi\ve\sqrt{sT/2^n}} \exp \bigg[ -\frac{\{ y+\int^s_0 b(\xi,s')ds'-\xi\}^2}{2\ve s} (1-CAs) \\
+ \frac{C(As)^3\xi^2}{2\ve s} + CAs - \frac {\{\xi+ \int^{s+T/2^n}_s b(\zeta,s')ds' - \zeta\}^2}{2\ve T/2^n} 
(1-CAT/2^n) + \frac{C(AT/2^n)^3}{2\ve T/2^n} \; \zeta^2 + CAT/2^n \bigg]  \ .
\end{multline}
Setting $z = y + \int^s_0 \; b(\zeta, s')ds' - \zeta$ we see from (\ref{CU4}) that
\begin{multline} \label{CV4}
G(y, \xi,0,s) \  G(\xi, \zeta, s, s+T/2^n) \le \\
 \frac1 {2\pi \ve \sqrt{sT/2^n}} 
\exp \bigg[- \frac{\{ z^2 + 2(\zeta - \xi)z + (\zeta - \xi)^2 \}}{2\ve s} 
- \frac{(\zeta-\xi)^2}{2\ve T/2^n}  
+ \frac{CA}{\ve} \left[ (\zeta-\xi)^2 + z^2 + \zeta^2 \right] + CAT \bigg],  
\end{multline}
for some universal constant $C$.  Integrating the RHS of (\ref{CV4}) over the region $|\xi - \zeta| > \rho\mu^n(1-\mu)$ we conclude that
\begin{multline} \label{CW4}
\int^\infty_{-\infty} d\xi \  G(y, \xi, 0, s) \  H \left(|\xi - \zeta| - \rho\mu^n(1-\mu) \right) G(\xi, \zeta, s, s + T/2^n) 
\le \\
 \exp \left[ - \frac{\rho^2\mu^{2n}(1-\mu)^2} {4\ve T/2^n} \right]
\frac 1 {\sqrt{2\pi\ve \tau_n}} \exp \left[ -\frac{z^2}{2\ve\tau_n} \left\{ \frac 1 2 - \frac{2CAT}{2^n} \right\} + \frac{CA}\ve \left\{ z^2 + \zeta^2\right\} + CAT \right] \ , 
\end{multline}
where  $\tau_n = T/2^n + s/2 - 2CAsT/2^n$.
Hence if we use the inequalities
\[   (y-\zeta)^2 [1 - As] - As \zeta^2 \le z^2 \le (y-\zeta)^2 [1+As] +2As\zeta^2,  \]
which are valid for $AT \le 1$, and substitute the RHS of (\ref{CW4}) into the RHS of (\ref{CT4}), we may conclude from Lemma 5.2  that the LHS of (\ref{CT4}) is bounded by a Gaussian integral in $\zeta$.  Evaluating this integral we have then that
\begin{multline}  \label{CX4}
P\left( |Y_\ve(s) - Y_\ve(s + T/2^n)| > \rho \mu^n(1-\mu)  \  \big| \ Y_\ve(0) = y,  \ Y_\ve(T) = 0 \right) \le \\
 G(y,0,0,T)^{-1} \   \frac{\sqrt{2}} {[2\pi\ve(T+T/2^n)]^{1/2}}
 \exp \left[- \frac{\rho^2\mu^{2n}(1-\mu)^2}{4\ve T/2^n}- \frac{y^2}{2\ve(T+T/2^n)} + \frac{CAy^2}{\ve} + CAT \right] 
 \end{multline}
for some universal constant $C$.  Choosing now $\mu$ in (\ref{CX4}) to satisfy $1 /\sqrt{2} < \mu < 1$ and using the lower bound for $G(y,0,0,T)$ in Lemma 5.2 we conclude from (\ref{CX4}) that
\begin{multline}  \label{CY4}
P\left( |Y_\ve(s) - Y_\ve(s + T/2^n)| > \rho \mu^n(1-\mu)  \  \big| \ Y_\ve(0) = y,  \ Y_\ve(T) = 0 \right) \le \\
\exp \left[ \frac{-\rho^2\mu^2(1-\mu)^{2n}}{8\ve T/2^n} \right]  \quad {\rm if \ }\rho \ge C_2[|y| + \sqrt{\ve T}],
\end{multline} 
provided $C_2$ is a sufficiently large universal constant.  Hence (\ref{CR4}), (\ref{CS4}) imply that
\begin{multline}  \label{CZ4}
P \left( \sup_{0\le s \le T} |Y_\ve(s)| > \rho \  \big| \ Y_\ve(0) = y, \ Y_\ve(T) = 0\right) \le \\
\sum^\infty_{n=1} 2^{n-1} \exp \left[ \frac{-\rho^2\mu^{2n}(1-\mu)}{8\ve T/2^n}\right] \le \exp \left[ - \frac{C_1\rho^2}{2\ve T} \right],  
\end{multline}
for some universal constant $C_1 > 0$. 
\end{proof}
To prove (\ref{AU4}) under the assumptions (\ref{D1}) and $b(0,\cdot)\equiv 0$ we actually need versions of Lemma 5.3 and Lemma 5.4 which hold in the situation when $b(0,\cd) \not\equiv 0$.  A slight modification of the proof of Lemma 5.3 yields:
\begin{corollary} Suppose $b(\cd,\cd)$ satisfies (\ref{A1}).  Then there are universal constants $\eta, \ C > 0$ such that the Green's function $G$ defined by (\ref{D1}) satisfies the inequalities
\be \label{DA4}
G(y,\xi,0,T) \le \frac 1{ \sqrt{2\pi\ve T}} \exp \left[- \frac{\{ y+\int^T_0 b(\xi,s)ds-\xi\}^2}{2\ve T(1+ CAT)} + \frac{C(AT)^3\xi^2}{2\ve T} + CAT  + \frac{CA}\ve \left\{ \int^T_0 |b(0,s)| ds\right\}^2 \right], 
\ee
\begin{multline}  \label{DB4}
G(y,\xi,0,T) \ge \frac 1{ \sqrt{2\pi\ve T}} \exp \bigg[ -\frac{\{ y+\int^T_0 b(\xi,s)ds-\xi\}^2}{2\ve T} (1+ CAT)\\
 - \frac{C(AT)^3\xi^2}{2\ve T} 
- CAT - \frac{CA}\ve \left\{ \int^T_0 |b(0,s)| ds\right\}^2 \bigg], 
\end{multline}
provided $AT \le \eta$.
\end{corollary}
We can also slightly modify the proof of Lemma 5.4 to obtain the following:
\begin{corollary} Suppose $b(\cd, \cd)$ satisfies (\ref{A1}).   Then for any $y \in \R$ which satisfies
\be \label{DCC4}
|y| + \sqrt{\ve T} \ge \int^T_0 |b(0,s)|ds.
\ee
the result of Lemma 5.4 holds.
\end{corollary}
\begin{proof}   We simply use the Green's functions bounds of Corollary 5.2 in place of the bounds of Lemma 5.3 in the argument of Lemma 5.4.
\end{proof}
\begin{lem} Suppose $b(\cd,\cd)$ satisfies (\ref{A1}) and $b(0,\cdot)\equiv 0$.  For $\la \in \R$ define $b_\la(\cd,\cd)$ by $b_\la(y,s) = b(y+\la s, s)$, $y \in \R, \ 0 \le s \le T$,  and let $G_\la$ be the Green's function (\ref{D1}) associated with $b_\la$.  Then there are universal constants $\eta, C > 0$ such that the following inequalities hold provided $AT \le \eta$:
\be \label{DC4}
\frac{G_\la(y,0,0,T)}{G_0(y,0,0,T)} \le \exp \left[ \frac{C|\la|AT}\ve \left\{|y| + |\la|AT^2 \right\} + CAT \right],
\ee
\be \label{DD4}
\frac{G_\la(y,0,0,T)}{G_0(y,0,0,T)} \ge \exp \left[- \frac{C|\la|AT}\ve \left\{|y| + |\la|AT^2 \right\} - CAT \right].
\ee
\end{lem}
\begin{proof}  Consider first the situation when $|y| \le |\la| AT^2$.  The result follows from Corollary 5.2 on using the inequality $\int^T_0 |b_\la(0,s)|ds \le |\la| AT^2/2$, whence we need only prove (\ref{DC4}), (\ref{DD4}) for $|y| \ge |\la|AT^2$.  Observe now that if $|y| = O(\sqrt{\ve T \ })$ then $|\la| AT |y|/\ve = O(y^2/\ve T) = O(1)$.  Hence we might expect to prove (\ref{DC4}), (\ref{DD4}) for $|y| = O(\sqrt{\ve T \ })$ by perturbation methods.  To implement this we consider the function $u_\la(z,t), \ z \in \R, \ t < T$, defined by
\be \label{DK4}
u_\la(z,t) = G_\la \left( z - \int^T_t b_\la (0,s)ds,  \ 0, \  t, \  T \right).
\ee
Evidently $u_\la$ is a solution to the terminal value problem
\begin{eqnarray} \label{DE4}
\frac {\pa u_\la}{\pa t} &+& \tilde b_\la(z,t) \frac {\pa u_\la}{\pa z} + \frac \ve 2 \; \frac {\pa^2 u_\la}{\pa z^2} = 0,  \quad z \in \R, \ t < T, \\
& \ & \lim_{t \ra T} \; u_\la (z,t) = \del(z),  \nn
\end{eqnarray}
where  $\tilde b_\la(z,t)$ is given by the formula
\be \label{DEE4}
 \tilde b_\la(z,t) = b_\la \left( z - \int^T_t b_\la(0,s) \ ds, \  t\right) - b_\la(0,t). 
\ee
Following the argument of Lemma 3.4 we see that the terminal value problem (\ref{DE4}) on the interval $|z| < \eta$ with Dirichlet boundary conditions can be solved by perturbation expansion for times $0 \le t < T$ provided $\sup \{|\tilde b_\la(z,t) | : |z| < \eta, 0 < t < T\} (T/\ve)^{1/2} < < 1$.  Assuming now that $y, \eta$ satisfy the inequalities
\be \label{DF4}
|\la|AT^2 \le |y| \le \sqrt{\ve T} / (AT)^\del, \ \eta = \sqrt{\ve T} \big/ (AT)^{2\del},
\ee
it is clear that the perturbation expansion converges provided $\del < 1/2$ and $AT$ is smaller than some constant depending only on $\del$.  In fact, letting $G(z,t), z \in \R, \; t > 0$, be the probability density function for the normal variable with mean 0 and variance $t$ we have that
\begin{multline} \label{DG4}
\Big|u_\la(z,t) - G( z, \ve(T-t)) 
+ \int^T_t ds \int^\eta_{-\eta} d\xi  \ G( z-\xi, \ve(s-t))  \ \tilde b_\la(\xi,s)\frac \pa{\pa\xi}G(\xi,\ve(T-s))\Big|  \\
\le C(AT)^{2-4\del} \ G(z, \; 2\ve(T-t )), 
\end{multline}
provided $|z| \le \sqrt{\ve T} \big/ (AT)^\del$.  Here $AT$ needs to be smaller than some constant depending only on $\del$, and the constant $C$ on the RHS of (\ref{DG4}) also depends on $\del$.  It is easy to see that
\begin{multline} \label{DH4}
 \Big| \int^\eta_{-\eta} d\xi \  G( z-\xi, \ve(s-t)) \  \tilde b_\la(\xi,s)\; \frac \pa{\pa\xi}\ G(\xi,\ve(T-s))\Big| \\
\le  \frac 1 2 \int^\infty_{- \infty} d\xi  \ G( z-\xi, \ve(s-t)) \frac A{\ve(T-s)} \left[ 3\xi^2 + A^2\la^2T^2(T-s)^2\right] G(\xi,\ve(T-s))  \\
=  \frac 1 2 \left\{  \frac{A^3\la^2T^2(T-s)}\ve + \frac{3A(T-s)z^2}{\ve(T-t)^2} + \frac{3A(s-t)}{T-t} \right\} G(z,\ve(T-t)). 
\end{multline}
Substituting the inequality (\ref{DH4}) into (\ref{DG4}) we conclude that $u_\la(z,t)$ satisfies the inequalities
\begin{multline}  \label{DI4}
u_\la(z,t) \le  G(z,\ve (T-t)) +C(AT)^{2-4\del} G(z,2\ve(T-t))+
\\
\bigg\{ \frac{A^3\la^2 T^2(T-t)^2}{4\ve} + \frac{3Az^2}{4\ve} + \frac{3A(T-t)}{4} \bigg\} G(z,\ve (T-t)),  
\end{multline}
\begin{multline}  \label{DJ4}
u_\la(z,t)   \ge   G(z,\ve (T-t)) - C(AT)^{2-4\del} \  G(z,2\ve(T-t))-
\\
\bigg\{ \frac{A^3\la^2T^2(T-t)^2}{4\ve} + \frac{3Az^2}{4\ve} + \frac{3A(T-t)}{4} \bigg\} G(z,\ve(T-t)). 
\end{multline}

We have shown that (\ref{DI4}), (\ref{DJ4}) holds for the function $u_\la(z,t)$ which satisfies (\ref{DE4}) on the intervals $|z| < \eta, \ 0 < t < T$, and with Dirichlet boundary conditions on $|z| = \eta$.  It follows that the function $u_\la(z,t)$ defined by (\ref{DK4}), also satisfies (\ref{DJ4}) for $|z| \le \sqrt{\ve T} / (AT)^\del, \ 0 < t < T$.  From the argument of Lemma 5.2 we  see that the upper bound (\ref{DI4}) continues to hold for the function (\ref{DK4}) when $|z| \le \sqrt{\ve T} / (AT)^\del$.

The inequalities (\ref{DC4}), (\ref{DD4}) can be deduced from (\ref{DI4}) ,(\ref{DJ4}) in the case when $y$ lies in the interval $|\la| AT^2 \le |y| \le K\sqrt{\ve T}$, where $K \ge 1$ is a constant.  The constant $C$ now in (\ref{DC4}), (\ref{DD4}) depends on $K$, and $AT$  must be chosen sufficiently small depending only on $K$.  To obtain the lower bound (\ref{DD4}) we set $\del = 1/8$ and $z = y + \int^T_0 b_\la(0,s)ds,  \ t = 0$ in (\ref{DJ4}).  Thus we obtain the inequality
\be \label{DL4}
G_\la(y,0,0,T) \ge G \left( y + \int^T_0 b_\la(0,s) \ ds,  \ \ve T \right) \exp \left[ - C\left\{ \frac{\la^2(AT)^3 T}{\ve} + AT\; e^{K^2} \right\} \right] 
\ee
for some universal constant $C$.  Lemma 5.2 implies that $G_0(y,0,0,T)$ satisfies the upper bound
\be \label{DM4}
G_0(y,0,0,T) \le G(y, \ve T) \exp \Big[ CAT (1 + K^2) \Big]
\ee
for some universal constant $C$.  Now (\ref{DD4}) follows by estimating from below the ratio of the RHS of (\ref{DL4}) to (\ref{DM4}).  The upper bound (\ref{DC4}) can be similarly obtained from (\ref{DI4}) and Lemma 5.2.

To complete the proof of the lemma we use induction as we did in Lemma 5.2.  We consider the lower bound (\ref{DD4}).  Observe first that the previous arguments imply that the lower bound
\be \label{DN4}
\frac {G_\la(y,0,t,T)}{G_0(y,0,t,T)} \ge \exp \left[ \frac{-C|\la|AT}{\ve} \left\{ |y| + |\la|AT(T-t)\right\} -CA(T-t) \right]
\ee
holds for $0 \le t < T$ if $y$ lies in one of the regions $|y| \le |\la|AT(T-t)$ or $|\la|AT(T-t) \le |y| \le K \sqrt{\ve(T-t)}$.  For the former region the constant $C$ in (\ref{DN4}) can be chosen in a universal way provided $AT$ is smaller than some universal constant.  For the latter region $C$ depends on $K$ and $AT$, and must be taken sufficiently small depending only on $K$.

Suppose now we have proved (\ref{DN4}) for $T-t = T/2^N, \; y \in \R$, where $N$ is some integer $N \ge 1$ with constant $C = C_N$.  We show that (\ref{DN4}) also holds for $T-t = T/2^{N-1}, \; y \in \R$, with a constant $C_{N-1}$ given in terms of $C_N$.  To do this we use the inequality
\begin{multline}  \label{DO4}
G_\la(y,0,t_{N-1},T)^{1-\al} \ge \\
\int^\infty_{-\infty} G_0(y,z,t_{N-1},t_N) 
\exp \left[ - \frac{(1-\al)}{2\al\ve} \ \frac{A^2\la^2T^3}{2^N} \right] G_\la \left(z,0,t_N,T\right)^{1-\al}\; dz,  
\end{multline}
where $T-t_n = T/2^n, \; n = 0,1,2..., \  0 < \al < 1$.  The inequality (\ref{DO4}) is derived similarly to (\ref{BK4}).

We assume $y$ in (\ref{DO4}) satisfies $|y| \ge \max \left[ |\la|AT(T-t_{N-1}), K\sqrt{\ve(T-t_N)}\right]$ and set $\al = |\la|AT^2/|y|2^N \le 1/2$.  Then, on substituting (\ref{DN4}) for $t=t_N$ into (\ref{DO4}) we obtain the inequality
\begin{multline} \label{DP4}
\frac{G_\la(y,0,t_{N-1},T)^{1-\al}} {G_0(y,0,t_{N-1},T)^{1-\al} } \ge \\
 \exp \bigg[ - \frac{(1-\al) |\la|AT|y|} {2\ve} 
\frac{-C_N(1-\al)|\la|^2(AT)^2T}{\ve 2^N} - \frac{C_N(1-\al)AT}{2^N} \bigg] G_0(y,0,t_{N-1},T)^\al  \\
E \bigg\{ G_0(Y_\ve(t_N),0, t_N, T)^{-\al} \exp \left[ - \frac{C_N(1-\al)|\la| AT|Y_\ve(t_N)|}{\ve} \right] 
 \ \Big| \ Y_\ve(t_{N-1})= y, \ Y_\ve(T) =0 \bigg\} ,  
\end{multline}
where $Y_\ve(\cdot)$ is the solution to (\ref{E1}).  From Lemma 5.2 we have that
\be \label{DQ4}
\frac {G_0(y,0,t_{N-1},T)} {G_0(z,0,t_N,T)} \ge \exp \left[ -\frac 1 2 \log 2 - \frac{y^2}{2\ve T/2^N} - \frac{CAT}{2^N} \right], \quad z \in \R,
\ee
for some universal constant $C$.  Since we are assuming that $|y| \ge K\sqrt{\ve(T-t_N)}$, we conclude from (\ref{DQ4}) that for sufficiently large $K$,
\be \label{DR4}
\frac {G_0(y,0,t_{N-1},T)^\al} {G_0(z,0,t_N,T)^\al} \ge \exp \left[ - \frac{|\la|AT|y|}{\ve}\right], \quad z \in \R.
\ee
To get a lower bound for the RHS of (\ref{DP4}) we are therefore left to estimate from below the expectation
\begin{multline}  \label{DS4}
E \left\{  \exp \left[  - \frac{C_N(1-\al)|\la| AT| |Y_\ve(t_N)|}{\ve} \right]  \  \Big| \  Y_\ve(t_{N-1})= y, \ Y_\ve(T) =0 \right\} 
\\
 \ge  \exp \left[ - \frac{C_N(1-\al)|\la| AT}{\ve}  E \Big\{ |Y_\ve(t_N)| \  \Big|  \ Y_\ve(t_{N-1})= y, \ Y_\ve(T) =0 \Big\} \right].   
\end{multline}
We estimate the expectation on the RHS of (\ref{DS4}) by
\begin{multline} \label{DT4}
 E \left[ | Y_\ve(t_N)|  \ \big|  \ Y_\ve(t_{N-1}) = y, \  Y_\ve(T) = 0 \right ] \le \\
  \frac{3|y|}4 
+ \int^\infty_{|y|/4} d\rho \ P\left( | Y_\ve(t_N)| - y/2| > \rho \; \Big|  \ Y_\ve(t_{N-1}) = y, Y_\ve(T) = 0 \right) . 
\end{multline}
We have now that 
\begin{multline}  \label{DU4}
P\left( | Y_\ve(t_N) - y/2| > \rho \; \big|  \ Y_\ve(t_{N-1}) = y,  \ Y_\ve(T) = 0 \right) 
\\
= G(y,0,t_{N-1}, T)^{-1}\int^\infty_{-\infty} \; d\xi \; G( y, \xi, t_{N-1}, t_N) \  H( |\xi - y/2| - \rho)
 \ G(\xi,0, t_N, T ),  
\end{multline}
with $H(\cdot)$ being the Heaviside function.  Now, arguing in the same way as we did to obtain (\ref{CY4}) we conclude from (\ref{DU4}) that
\begin{multline} \label{DV4}
P\left( | Y_\ve(t_N) -y/2| > \rho \; \big|  \ Y_\ve(t_{N-1}) = y, \  Y_\ve(T) = 0 \right)  \le
\\
\exp \left[ - \frac{\rho^2}{2\ve T/2^N} + \frac{CAy^2} \ve \right], \quad {\rm if} \  \rho \ge K_0 \sqrt{\ve(T-t_N)}, 
\end{multline}
where $C$ and $K_0$ are universal constants.  Hence we have that
\begin{multline}  \label{DW4}
\int^\infty_{|y|/4} d\rho \ P\left( | Y_\ve(t_N) - y/2| > \rho \; \big|  \; Y_\ve(t_{N-1}) = y, \  Y_\ve(T) = 0 \right)
\\
\le \frac{\ve T}{|y|2^{N-2}} \exp \left[ - \frac{|y|^2\; 2^{N-4}}{2\ve T} + \frac{CAy^2} \ve \right], 
\end{multline}
provided the constant $K_0$ in (\ref{DV4}) satisfies $K \ge 4K_0$.  Now (\ref{DT4})  and (\ref{DW4})  imply that the expectation on the LHS of (\ref{DT4}) is bounded by $4|y|/5$.  Hence it follows from (\ref{DP4}), (\ref{DR4}), (\ref{DS4}), that we can take $C_{N-1} = 4C_{N}/5 + K_1$ for some universal constant $K_1$.  Thus in order to complete the proof of (\ref{DD4}) we need to show that $C_N$ satisfies $\lim_{N\ra \infty} 4^N \; C_N /5^N = 0$.  To do this we proceed as in Lemma 5.2 by proving that (\ref{DD4}) holds with a constant $C = C(AT)$ which diverges logarithmically in $AT$ as $AT \ra 0$.

We have already observed that (\ref{DD4}) holds for a universal constant $C$ if $|y| \le |\la|AT^2$ and for a constant $C$ depending only on $K$ if $|\la|AT^2 \le |y| \le K\sqrt{\ve T}$.   Hence we shall assume that $|y| \ge \max \left[| \la|AT^2 , K\sqrt{\ve T}\right]$.  Analogously to (\ref{DO4}) there is the inequality
\begin{multline}  \label{DX4}
G_\la(y,0,0,T)^{1-\al} \ge \\
\int^\infty_{-\infty} G_0(y,z,0,T-\De) 
	\exp \left[ - \frac{(1-\al)} {2\al\ve} \ A^2\la^2 T^3 \right] G_\la \left(z,0,T-\De, T\right)^{1-\al}\; dz.  
\end{multline}
We set $\al=|\la|AT^2 /2|y|$ in (\ref{DX4}), whence $0 < \al \le 1/2$ and the exponential on the RHS of (\ref{DX4}) can be absorbed into the RHS of (\ref{DD4}).  As in (\ref{DK4}) we shall obtain a perturbation expansion of $G_\la \left(z,0,T-\De, T\right)$ by considering the function,
\begin{eqnarray} \label{DY4}
u_\la(z,t) &=& G_\la (z + \vp_\la(t), 0, t, T ), \quad {\rm where} \\
\phi '_\la(t) &=& b_\la(\vp_\la(t), t ),  \quad t < T, \ \vp_\la(T) = 0. \nn
\end{eqnarray}
Then $u_\la(z,t)$ is a solution to the terminal value problem (\ref{DE4}) but now with drift $\tilde b_\la(z,t)$ given by
\be \label{DZ4}
\tilde b_\la(z,t) = b_\la(z + \vp_\la(t), t) - b_\la(\vp_\la(t), t).
\ee
Since $|\tilde b_\la(z,t)| \le A|z|, \  z \in \R$, we may expand the solution of the Dirichlet problem (\ref{DE4}) by perturbation theory on the intervals $|z| < \eta, \ T-\De < t < T$, provided $\eta, \De$ satisfy
\be \label{EA4}
\eta = K_1\sqrt{\ve \De}, \ \ (A\eta)^2\De = \nu \ve,
\ee
where $K_1 > > 1$ and $\nu < < 1$.  Thus if $u_{\la,D}(z,t)$ denotes the solution to this Dirichlet problem we have as in (\ref{BO4}) the inequality
\begin{multline} \label{EB4}
u_{\la,D}(z, T-\De) \ge \frac 1{\sqrt{2\pi\ve\De}} \bigg[ \exp \left\{ \frac{-z^2}{2\ve\De} \right\} - \\
C_3 \exp \left[- \frac{K^2_1}{4} \right] - C_4(\rho) \nu^{1/2} \exp \left\{ - \frac{z^2}{2\ve(1+\rho)\De} \right\} \bigg], \quad  |z|<\eta, 
\end{multline}
where $C_3$ is a universal constant, $\rho > 0$ can be arbitrary and $C_4(\rho)$ is a constant depending only on $\rho$.  We choose now $\nu, K_1, \De/T$ by
\begin{eqnarray} \label{EC4}
K_1 &=& (AT)^{-k_1} \exp [C_1Ay^2/\ve], \quad  \nu^{1/2} = (AT)^{k_2} \exp [-C_2Ay^2/\ve], \\
\De/T &=& (AT)^{k_1 + k_2 - 1} \exp \left[ -(C_1 + C_2) Ay^2/\ve \right] , \nn
\end{eqnarray}
where $k_1, k_2, C_1, C_2 > 0$ and $k_1 + k_2 > 1$.  Evidently the choice of $K_1, \nu, \De$ in (\ref{EC4}) is consistent with (\ref{EA4}).

To estimate from below the LHS of (\ref{DD4}) we use the inequality derived from (\ref{DX4}), (\ref{DY4}),
\begin{multline}  \label{ED4}
G_\la(y, 0, 0, T)^{1-\al} \ge \exp \left[ - \frac{(1-\al)|\la| AT|y|}{\ve} \right] 
\\
\int^{\eta + \vp_\la(T-\De)}_{-\eta + \vp_\la(T-\De)} G_0(y, z,0, T-\De) \ u_{\la,D}\big(z-\vp_\la(T-\De), \  T-\De\big)^{1-\al}\; dz.  
\end{multline}
If we substitute now the RHS of (\ref{EB4}) into the RHS of (\ref{ED4}) we obtain an integral which we would like to show is comparable to $G_0(y, 0, 0, T)^{1-\al} $.  To do this we write
\be \label{EE4}
G_0(y,0,0,T) = \int^\infty_{-\infty} G_0(y, z,0, T-\De)  \  G_0(z, 0, T-\De, T)  \ dz,
\ee
and use perturbation analysis to show that $G_0(z,0,T-\De, T)$ is comparable to the RHS of (\ref{EB4}).  Using the upper bound (\ref{CA4}) on $G_0(z,0,T-\De, T)$ in (\ref{EE4}) we obtain an upper bound on $G_0(y,0,0, T)$ which has the same form as the integral on the RHS of (\ref{ED4}).

We compare the principal terms of these integrals.  Thus for the integral on the RHS of (\ref{ED4}) the principal term is 
\be \label{EF4}
(2\pi\ve\De)^{\al/2} \int^{\eta + \vp_\la(T-\De)}_{-\eta + \vp_\la(T-\De)} G_0(y, z,0, T-\De) \frac{1}{\sqrt{2\pi\ve\De}}  \exp \left[ -(1-\al)  \frac{\{z - \vp_\la(T-\De)\}^2} {2\ve\De} \right] \ dz.
\ee
For the integral on the RHS of (\ref{EE4}) the principal term is
\be \label{EG4}
\int^\infty_{-\infty} G_0(y, z,0, T-\De) \ \frac 1{\sqrt{2\pi\ve\De}} \exp \left[ \frac{-z^2}{2\ve\De} \right] dz.
\ee
Observe now that we may assume $|\vp_\la(T-\De)| \le \eta/2$.  To see this we first note from (\ref{DY4}) that $|\vp_\la(T-\De)| \le CA|\la|T\De$ for some universal constant $C$, whence it follows that $|\vp_\la(T-\De)| \le C|y|\De/T$.  Thus from (\ref{EA4}) the inequality will follow if we can show that $2C|y| \le K_1(T/\De)^{1/2}\; \sqrt{\ve T}$, which is equivalent to showing that $4C^2|y|^2/\ve T \le K^2_1(T/\De)$.  Choosing $K_1, \; T/\De$ as in (\ref{EC4}) we see that this inequality holds provided $3k_1 + k_2 > 2$ and $AT$ is sufficiently small, depending only on $C_1,C_2$.  Similarly we have that
\begin{multline} \label{EH4}
\frac{\vp_\la(T-\De)^2}{2\ve \De} + \frac{\eta|\vp_\la(T-\De)|}{\ve \De} \le \frac{C^2A^2|\la|^2\;T^2\De}{2\ve}  + \frac{C\eta\;A|\la|T}{\ve } \le \\
 \frac{C^2A^2|\la|^2\;T^3}{2\ve} + \frac{CK_1|y|}{\sqrt{\ve T}} \; \left( \frac \De T \right)^{1/2} 
\le \frac{C^2A^2|\la|^2\;T^3}{2\ve}  + C'AT, 
\end{multline}
provided the constants in (\ref{EC4}) satisfy $k_2 > k_1 + 5, \ C_2 > C_1$.  In that case the constant $C'$ in (\ref{EH4}) depends only on $k_1,k_2,C_1,C_2$.  We conclude then that the expression in (\ref{EF4}) is bounded below by
\be \label{EI4}
(2\pi\ve \De)^{\al/2} \exp \left[ - \frac{C|\la|^2\;A^2\;T^3}{\ve}  - CAT \right] \int^{\eta/2}_{-\eta/2} G_0(y,z,0,T-\De)
\frac 1{\sqrt{2\pi\ve\De}} \exp \left[ - \frac{z^2}{2\ve\De} \right] dz,
\ee
for a constant $C$ depending only on the constants in (\ref{EC4})

Next we bound the integral in (\ref{EI4}) from below by a constant times the integral in (\ref{EG4}).  To show this we use Lemma 5.3.  Thus the Green's function $G_0(y,z,0,t)$ is bounded above and below by the inequalities
\begin{eqnarray} \label{EJ4}
G_0(y,z,0,t) &\le& \frac 1{\sqrt{2\pi\ve t}} \exp\left[ -\frac{(y-z)^2}{2\ve t} + \frac {CA}\ve \; (y^2+ z^2) + CAt\right], \\
G_0(y,z,0,t) &\ge& \frac 1{\sqrt{2\pi\ve t}} \exp\left[ - \frac{(y-z)^2}{2\ve t} - \frac {CA}\ve \; (y^2+ z^2) - CAt\right]. \nn
\end{eqnarray}
Substituting the lower bound of (\ref{EJ4}) into (\ref{EG4}) we conclude that
\begin{multline} \label{EK4}
\int^\infty_{-\infty} G_0(y,z,0,T-\De) \ \frac 1{\sqrt{2\pi\ve\De}} \exp \left[  -\frac{z^2}{2\ve\De} \right] dz \ge
\\
 \frac 1{\sqrt{2\pi\ve T}} \exp \left[ - \frac{y^2}{2\ve T} - \frac{CA} \ve \ y^2 - CAT \right],  
 \end{multline}
for some universal constant $C$.  Using the upper bound in (\ref{EJ4}) we also have that 
\begin{multline} \label{EL4}
\int_{|z| > \eta/2} G_0(y,z,0,T-\De) \ \frac 1{\sqrt{2\pi\ve\De}} \exp \left[-  \frac{z^2}{2\ve\De} \right] dz \\
\le \exp\left[ -\frac{\eta^2}{16\ve\De} \right] \int^\infty_{-\infty} G_0(y,z,0,T-\De) \ \frac 1{\sqrt{2\pi\ve\De}} \exp \left[ - \frac{z^2}{4\ve\De} \right] dz  \\
 \le  \frac{1}{\sqrt{2\pi\ve T}}   \exp\left[ -\frac{\eta^2}{16\ve \De} - \frac{y^2}{2\ve T} + \frac{CA}{\ve} \; y^2 + CAT + \frac 1 2 \; \log 2 \right], 
\end{multline}
for some universal constant $C$.  Observe that in (\ref{EL4}) we are assuming that the constants in (\ref{EC4}) satisfy $k_1 + k_2 > 2$ so that $\De/T \le AT$.  Now taking $\eta$ to be given by (\ref{EA4}), (\ref{EC4}), we conclude  from (\ref{EK4}), (\ref{EL4}) that
\begin{multline} \label{EM4}
 \int_{|z| > \eta/2} G_0(y,z,0,T-\De) \ \frac 1{\sqrt{2\pi\ve\De}} \exp \left[ - \frac{z^2}{2\ve\De} \right] dz \\
\le \exp\left[- \frac{1}{32(AT)^{2k_1}} \right] \int^\infty_{-\infty} G_0(y,z,0,T-\De) \ \frac 1{\sqrt{2\pi\ve\De}} \exp \left[  -\frac{z^2}{2\ve\De} \right] dz. 
\end{multline}
It follows that (\ref{EF4}) is bounded below by
\begin{multline} \label{EN4}
(2\pi\ve\De)^{\al/2} \exp \left[ - \frac{C^|\la|^2\;A^2T^3}{\ve} - CAT \right] 
\\
 \left\{ 1- \exp\left[- \frac{1}{32(AT)^{2k_1}} \right] \right\} \int^\infty_{-\infty} G_0(y,z,0,T-\De) \ \frac 1{\sqrt{2\pi\ve\De}} \exp \left[  - \frac{z^2}{2\ve\De} \right] dz.  
 \end{multline}
The integral in (\ref{EN4}) is the principle term in the expression (\ref{EE4}) for $G_0(y,0,0,T)$.  Assuming then that we can replace the integral by $G_0(y,0,0,T)$ and that we take into account only the principal term for $u_{\la,D}$ in (\ref{ED4}), we have from (\ref{EN4}) that the inequality
\begin{multline}  \label{EO4}
G_\la(y,0,0,T)^{1-\al} \ge \left[ (2\pi \ve \De)^{1/2} \ G_0(y,0,0,T) \right]^{\al} 
\\
\exp \left[ - \frac{C(1-\al)|\la|AT}{\ve} \left\{ |y| + \la AT^2 \right\} - C(1-\al)AT \right] G_0(y,0,0,T)^{1-\al}  
\end{multline}
holds for some universal constant $C$.  By Lemma 5.2 we have that
\be \label{EP4}
 \left[ (2\pi \ve \De)^{1/2} \ G_0(y,0,0,T) \right]^{\al} \ge \exp\left[ - \frac \al 2 \log (T/\De) - \frac{\al y^2} {2\ve T} \ (1 + CAT) - \al CAT \right] 
\ee
for some universal constant $C$.  Taking $\al$ as before to be given by $\al = |\la|AT^2/2|y| \le 1/2$ and using the fact that $|y| \ge K\sqrt{\ve T}$ we see from (\ref{EC4}), (\ref{EP4})  that
\be \label{EQ4}
\left[ (2\pi \ve \De)^{1/2} \ G_0(y,0,0,T) \right]^{\al} \ge \exp \left[ -(1-\al) CAT \left\{ \frac{|\la y|}\ve |\log(AT)| + 1 \right\}\right],
\ee
where the constant $C$ depends only on the constants $C_1, C_2, k_1, k_2$ of (\ref{EC4}) and also $K$.  Combining then (\ref{EO4}), (\ref{EQ4}) we have obtained a lower bound of the form (\ref{DD4}) with a constant $C=C(AT) = C'|\log (AT)|$.

To complete the proof of (\ref{DD4}) with a constant  $C=C(AT) = C'|\log (AT)|$ we need to estimate the effect of the error terms in  (\ref{CA4}),  (\ref{EB4}).  From (\ref{CA4}) the main error term in  (\ref{EE4}) is given by
\be \label{ER4}
\int^\infty_{-\infty} G_0(y,z,0,T-\De) \  C_2(\rho)\nu^{1/2} \; \frac 1{\sqrt{2\pi\ve\De}} \; \exp \left[ - \frac{z^2}{2\ve(1+\rho)\De}\right] dz,
\ee
where $\nu$ is given by (\ref{EC4}).  It is evident that by choosing $k_2, C_2$ sufficiently large in a universal way in  (\ref{EC4}) that the integral of  (\ref{ER4}) is bounded above by $AT$ times the integral on the LHS of (\ref{EK4}).  We can similarly estimate the error terms in (\ref{ED4})  of $u_{\la,D}$.  If we use the inequality (\ref{BQQ4}) then from  (\ref{EB4}) we obtain a term like  (\ref{ER4}).   Hence (\ref{DD4}) with a constant $C=C(AT) = C'|\log (AT)|$ holds.  By previous argument it follows then that (\ref{DD4}) holds with some universal constant $C$ provided $AT \le \eta$, where $\eta$ may also be chosen in a universal way.

The completion of the proof of the upper bound (\ref{DC4}) can be carried out in a similar way to the method we used to prove the lower bound. 
\end{proof}
\begin{lem}  Suppose $b(\cd, \cd)$ satisfies (\ref{A1}) and $b(0,\cdot)\equiv 0$.  If $G$ is the Green's function defined by (\ref{D1}), then there are universal constants $\eta, C > 0$ such that $G$ satisfies the inequalities,
\be \label{ES4}
\frac{G(y,\xi,0,T)}{G(y,0,0,T)} \le \exp \left[ -\frac{\xi^2}{2\ve T}(1 - CAT) + \frac{\xi y}{\ve T} \left[1 + CAT \ {\rm sgn}(\xi y)\right] + CAT \right],
\ee
\be \label{ET4}
\frac{G(y,\xi,0,T)}{G(y,0,0,T)} \ge \exp \left[- \frac{\xi^2}{2\ve T}(1 + CAT) + \frac{\xi y}{\ve T} \left[1 - CAT \ {\rm sgn}(\xi y)\right] - CAT \right],
\ee
for all $y,\xi \in \R$, provided $AT \le \eta$.
\end{lem}
\begin{proof} The result follows from Lemma 5.3 if $|y| \le |\xi|$, so we shall assume that $|\xi| \le |y|$.  Letting $w(z,t) = G(z,\xi, t, T), \  t < T$, it follows from (\ref{B1}) that the function $w_\la(z,t) $ defined by
\be \label {EU4}
w_\la(z,t) = \exp \left[ - \frac{\la z}{\ve} + \frac{\la^2}{2\ve} \; (T-t) \right] w(z + \la t, t)
\ee
is the solution to the terminal value problem
\begin{eqnarray} \label{EV4}
0&=& \frac{\pa w_\la}{\pa t} + b(z + \la t, t) \; \frac{\pa w_\la}{\pa z} + \frac \ve 2 \; \frac{\pa^2 w_\la}{\pa z^2} + \frac \la \ve \; b(z + \la t, t)w_\la, \\
\lim_{t \ra T} w_\la (z,t) &=& \exp \left[ - \frac{\la z}{\ve}\right] \del (z + \la T - \xi). \nn
\end{eqnarray}
Taking $\la = \xi/T$ in (\ref{EV4}) we see from Lemma 5.5 that
\be  \label{EW4}
w_\la(y,0) = G_\la(y,0,0,T) E \left[ \exp \left\{ \frac \la \ve \; \int^T_0 
b( Y_{\ve,\la}(s) + \la s, s) \ ds \right\}  \  \Big| \  Y_{\ve,\la}(0) = y, \ Y_{\ve,\la}(T) = 0 \right],  
\ee
where $Y_{\ve,\la}(\cdot)$ is the solution to (\ref{D1}) with the drift $b_\la$ of Lemma 5.5 in place of $b$.  Since
\[	\int^T_0 |b_\la(0,s)| \ ds \le AT^2|\la|/2 \le |\xi| \le |y|   \]
if $AT \le 2$, we can use Corollary 5.3 to estimate the expectation in (\ref{EW4}).  To see this first observe that the expectation is bounded above by
\be \label{EX4}
\exp \left[ \frac{\la^2 AT^2}{2\ve} \right] E\left[ \exp \left\{ \frac{AT|\la|}\ve \; \sup_{0 \le s \le T}  |Y_{\ve,\la}(s)| \right\} \; \Big| \; Y_{\ve,\la}(0) = y, \; Y_{\ve,\la}(T) = 0 \right].
\ee
To bound the expectation in (\ref{EX4}) we use the identity,
\be \label{EY4}
E\Big [e^X\Big] = \ \ 1 + \int^1_0 \ E\Big[Xe^{kX}\Big] dk 
\ee
for any random variable $X$. 
From (\ref{CQ4}) we have that
\begin{multline}  \label{EZ4}
E \left[  \sup_{0 \le s \le T} |Y_{\ve,\la}(s)| \exp \left\{ r  \sup_{0 \le s \le T}|Y_{\ve,\la}(s)| \right\} \Big| 
\; Y_{\ve,\la}(0) = y, \; Y_{\ve,\la}(T) = 0 \right]    
\\
\le C_2 \left[ |y| + \sqrt{\ve T} \right] \ \exp \left\{ r C_2 \left[ |y| + \sqrt{\ve T} \right] \right\} \\
\sum^\infty_{n=0} (n+1) \exp\left\{nr C_2 \left[ |y| + \sqrt{\ve T} \right] - C_1C^2_2n^2 \left[ |y| + \sqrt{\ve T} \right]^2/2\ve T \right\} , 
\end{multline}
for any $r \ge 0$.  Assuming $r \le AT |\la|/\ve$ and using the fact that $|\la| \le |y| /T$, we see that there is an integer $n_0 \ge 1$, depending only on $AT$ and $C_1,C_2$, such that $2r \le C_1C_2 n_0[|y| + \sqrt{\ve T} ]/2\ve T$.  Hence (\ref{EZ4}) implies that there is a constant $C$ depending only on $AT$ such that
\begin{multline}  \label{FA4}
E \left[  \sup_{0 \le s \le T} |Y_{\ve,\la}(s)| \exp \left\{ r  \sup_{0 \le s \le T} |Y_{\ve,\la}(s)| \right\} \Big| 
\; Y_{\ve,\la}(0) = y, \; Y_{\ve,\la}(T) = 0 \right] 
\\
\le \ C \left[ |y| + \sqrt{\ve T}  \ \right] \ \exp \left\{  Cr \left[ |y| + \sqrt{\ve T}  \ \right] \right\}, \quad  0 \le r \le AT |\la|/\ve\; . 
\end{multline}
It follow now from (\ref{FA4}), on using the inequality $2\sqrt{\ve T} \le \ve/|\la| + |\la|T$, that there is a constant $C$ depending only on $AT$ such that
\begin{multline} \label{FB4}
E \left[ \exp \left\{ \frac \la \ve \ \int^T_0 \; b \left( Y_{\ve,\la}(s) + \la s, s \right)ds \right\} \; \Big| \; Y_{\ve,\la}(0) = y, Y_{\ve,\la} (T) = 0 \right]
\\
\le \exp \left[ \frac{C|\la|AT}\ve \Big\{ |y| + |\la|T \Big\} + CAT \right].	
\end{multline}
Substituting (\ref{FB4}) into (\ref{EW4}) and using inequality (\ref{DC4}) of Lemma 5.5, we conclude that upper bound (\ref{ES4}) holds.  The lower bound (\ref{ET4}) can be established by a similar argument. \ \ \end{proof}

\begin{proof}[Proof of Theorem 1.2]  Since in Corollary 5.1 we already proved the result for $\del \sim T/2$ we shall be concerned here with the situation where $\del/T < < 1$.  We have now with $y < 0$, the identity
\begin{multline}  \label{FC4}
P \left( Y_\ve(T-\del) < {C_3\del y}/{T} \; \Big| \; Y_\ve(0) = y, \ Y_\ve(T) = 0 \right) = 
\\
G(y,0,0,T)^{-1} \int^{C_3\del y/T}_{-\infty} \; d\xi  \ G(y, \xi, 0, T-\del) \  G(\xi, 0, T-\del, T). 
\end{multline}
There is also the identity,
\be \label{FD4}
G(y,0,0,T)  = \int^{\infty}_{-\infty} \; d\xi  \ G(y, \xi, 0, T-\del)  \ G(\xi, 0, T-\del, T) .
\ee
From Lemma 5.2 and Lemma 5.6 one obtains from (\ref{FD4}) the inequality, 
\begin{multline}  \label{FE4}
G(y, 0, 0, T) \ge \frac{G(y, 0, 0, T-\del)}{\sqrt{2\pi\ve \del}} \int^{\infty}_{-\infty} \ d\xi \exp \bigg[ \frac{-\xi^2}{2\ve(T-\del)}
\\
 \frac{-\xi^2}{2 \ve \del} + \frac{\xi y}{\ve(T-\del)} - \frac{CA\xi^2}{\ve} - \frac{CA|\xi y|}{\ve} - CAT \bigg]  
 \end{multline}
for some universal constant $C$.  Now let $X$ be the normal variable with mean $\del y/T$ and variance $\ve \del(T-\del)/T$.  Then (\ref{FE4}) is equivalent to
\be \label{FF4}
G(y, 0, 0, T) \ge G(y, 0, 0, T-\del) \left( \frac{T-\del}T \right)^{1/2} \exp \left[ \frac{\del y^2}{2\ve T(T-\del)} - CAT \right] E\left[ \exp \left\{ \frac{-CAX^2}{\ve} - \frac{CA|y| |X|}{\ve} \right\} \right] .  
\ee
Applying Jensen's inequality in (\ref{FF4}) and then the Schwarz inequality, we conclude that
\be \label{FG4}
G(y, 0, 0, T) \ge G(y, 0, 0, T-\del) \left( \frac{T-\del}T \right)^{1/2} \exp \left[ \frac{\del y^2}{2\ve T(T-\del)} - CAT- \frac{CA\del y^2}{\ve T}  - CA\left(\frac{\del}{\ve}\right)^{1/2} |y| \right] ,
\ee
for some universal constant $C$.  We similarly have from Lemma 5.2 and Lemma 5.6  that 
\begin{multline}  \label{FH4}
\int^{C_3\del y/T}_{-\infty}  d\xi \  G(y, \xi, 0, T-\del) \  G(\xi, 0, T-\del, T) \le \\
G(y, 0, 0, T-\del) \left( \frac{T-\del}T \right)^{1/2} \exp \bigg[ \frac{\del y^2}{2\ve T(T-\del)} + CAT \bigg] 
E\left[ \exp \left\{ \frac{CAX^2}\ve + \frac{CA|y| |X|}\ve \right\} \; ; \; X < \frac{C_3\del y}{T}\right],   
\end{multline}
for some universal constant $C$.  Assuming now that $C_3 > 1$, we have then
\begin{multline}  \label{FI4}
E\left[ \exp \left\{ \frac{CAX^2}\ve + \frac{CA|y| |X|}\ve \right\} \; ; \; X < \frac{C_3\del y}{T} \right] \le 
\sqrt{2} \exp \left[ - \frac{(C_3-1)^2\del y^2}{4\ve T(T-\del)} \right]
 \\
 \ E \left[ \exp \left\{ \frac{CAX_1^2}\ve + \frac{CA|y| X_1}\ve \right\}+
  \exp \left\{  \frac{CAX_1^2}\ve - \frac{CA|y| X_1}\ve \right\}  \right],
\end{multline}
where $X_1$ is the Gaussian variable with mean $\del y/T$ and variance $2\ve \del(T-\del)/T$.  The expectation on the RHS of (\ref{FI4}) can be explicitly computed.  Hence we conclude that
\begin{multline}  \label{FJ4}
E\left[ \exp \left\{ \frac{CAX^2}\ve +  \frac{CA|y||X|}\ve \right\} \; ; \; X < \frac{C_3\del y}T \right] \le 
\\
	2 \sqrt{2} \ \exp \left[ - \frac{(C_3-1)^2 \del y^2}{4\ve T(T-\del)} +  \frac{CA\del y^2}{\ve T} + CA\del \right], 
\end{multline}
for some universal constant $C$.  The first inequality of (\ref{AU4}) follows from (\ref{FC4}) -- (\ref{FJ4}) upon taking $C_3$ large enough and using the fact that $y < -T \sqrt{\ve/\del}$.

To prove the second inequality of (\ref{AU4}) we consider the identity,
\begin{multline} \label{FK4}
P \left( Y_\ve(T-\del) > \frac{C_4\del y}T \; \big| \; Y_\ve(0) = y, \; Y_\ve(T) = 0 \right) = 
\\
G(y,0,0,T)^{-1} \ \int^\infty_{ {C_4\del y}/T} \; d\xi \; G(y,\xi,0, T-\del)  \ G(\xi,0,T-\del,T).   
\end{multline}
We now choose $C_4$ to satisfy $0 < C_4 < 1$ and proceed as previously. 
\end{proof}

\section{Representation formula for the Stochastic Cost function}
Corollary 4.1 suggests that we may take the limit $\del \ra 0$ in (\ref{BH3}) by setting $\displaystyle{\lim_{\del\ra 0}}\; E[q_\ve(x, y_\ve(T-\del), T-\del)]=0$, but it does not prove it.  In fact Lemma 3.1 shows that $q_\ve(x, y_\ve (T-\del),T-\del)$ becomes arbitrarily large for $y$ close to $x$ with $y<x$ as $\del \ra 0$.  To deal with this problem we need to obtain a sharper lower bound on $-\pa q_\ve(x,y,t)/\pa y$ than in (\ref{M4}), in particular one that does not decay as $y \ra -\infty$.  In the linear approximation $b(y,s) = A(s)y$, one can express $-\pa q_\ve(0,y,0)/\pa y$ for $y < 0$ by the formula,
\be \label{AP4}
-\frac {\pa q_\ve(0,y,0)}{\pa y} = \frac{\ve\La (T)}{\sqrt{2\pi \ve \sig^2(T)}} \exp \left[ -\frac{\La(T)^2y^2}{2\ve \sig^2(T)} \right] \Big/ \Phi \left( \frac{\La(T)y}{\sqrt{ \ve \sig^2(T)}} \right),
\ee
where $\Phi$ is the cumulative distribution function for the standard normal variable and $\La(T), \ \sig^2(T)$ are given by (\ref{AQ4}).
Hence provided $AT < 1$ we see from (\ref{J4}) that
\be \label{AR4}
-\frac{\pa q_\ve}{\pa y} \; (0,y,0) \sim -\frac{\La(T)^2y}{\sig^2(T)}, \quad  y \big/ \sqrt{\ve T} << - 1.
\ee
Comparing (\ref{AR4}) and (\ref{M4}), we see that the exponential factor in (\ref{M4}) may be removable in the case of nonlinear $b(\cdot,\cdot)$. We prove this in the following:
\begin{lem}  Suppose $b(\cd, \cd)$ satisfies (\ref{A1}) and let $q_\ve(x,y,t), \; x,y \in \R, \; t < T$, be defined by (\ref{G1}).  Then there are universal constants $C,\eta > 0$ such that
\be \label{FL4}
-\frac{\pa q_\ve(x,y,t)}{\pa y} \ge \left[ \frac{F(x,t)-y}{T-t} \right] e^{-CA(T-t)},
\ee
provided $0 \le t < T$, $A(T-t) < \eta, \ y < F(x,t)$, where $F(x,t)$ is the function defined from (\ref{K1}).
\end{lem}
\begin{proof} From (\ref{O4}) we see that
\be \label{FM4}
-\frac{\pa q_\ve(x,y,t)}{\pa y} \ge e^{-A(T-t)} \frac{\ve \; G(y,x,t,T)}{\int^\infty_x\; G(y,z,t,T) \ dz} \ .
\ee
Let us assume first that $b(x,s) = 0, \ 0 \le s \le T$, whence $F(x,t) = x, \ 0 \le t < T$.  From Lemma 5.6 we see that provided $A(T-t) < \eta$ and $\eta$ is chosen sufficiently small,
\be \label{FN4}
\frac{\int^\infty_x G(y,z,t,T)dz}{\ve G(y,x,t,T)} \le \frac 1 \ve \ \int^\infty_0 d\xi \exp \left[ -\frac{\xi(x-y)}{ \ve(T-t)\{1 + CA(T-t)\}}\right] 
 = \frac{(T-t)\{1+CA(T-t)\}}{x-y} \ 
\ee
where $C$ is a universal constant.  The inequality (\ref{FL4}) follows now from (\ref{FM4}), (\ref{FN4}).  To deal with the more general case we make the change of variable as in (\ref{DY4}), (\ref{DZ4}), and proceed as above.
\end{proof}
\begin{theorem} Suppose $b(\cd, \cd)$ satisfies (\ref{A1}) and  $q_\ve(x,y,t), \; x,y \in \R, \; t < T$, is defined by (\ref{G1}).  If $\la_\ve(\cd, \cd)$ is the optimal controller defined by (\ref{P1})  then for $0 \le t < T$, $x,y \in \R$, the functions $q_\ve(x,y,t)$, $\pa q_\ve(x,y,t)/\pa x$ and $\pa q_\ve(x,y,t)/\pa y$ have the representations,
\be \label{FO4}
q_\ve(x,y,t) = E \left\{ \frac 1 2 \; \int^T_t \left[ \la_\ve(y_\ve(s), s) - b(y_\ve(s), s)\right]^2\; ds \  \big|  \ y_\ve(t) = y \right\}
\ee
\begin{multline} \label{FP4}
\frac{ \pa q_\ve(x,y,t)}{\pa y} = \\
-\frac{1}{T-t} \; E \left\{  \int^T_t \left[ 1 + (T-s) \frac{\pa b}{\pa y} (y_\ve(s),s) \right] 
\left[ \la_\ve(y_\ve(s), s) - b(y_\ve(s), s)\right] ds \  \big|  \ y_\ve(t) = y \right\},
\end{multline}
\begin{multline} \label{FQ4}
\frac{ \pa q_\ve(x,y,t)}{\pa x} = \\
\frac{1}{T-t} \;E \left\{  \int^T_t \left[ 1 - (s-t) \frac{\pa b}{\pa y} (y_\ve(s),s) \right] 
\left[ \la_\ve(y_\ve(s), s) - b(y_\ve(s), s)\right] ds \  \big|  \ y_\ve(t) = y \right\},
\end{multline}
where $y_\ve(s), \; t \le s < T$, is the solution to the SDE (\ref{N1}) with initial condition $y_\ve(t) = y$.
\end{theorem}
\begin{proof}  From (\ref{BH3}) the representation (\ref{FO4}) for $q_\ve(x,y,t)$ holds provided we can show that
\be \label{FR4}
\lim_{\del\ra 0} E\left[ q_\ve(x,y_\ve(T-\del),T-\del) \  \big| \ y_\ve(t)=y\right] = 0.
\ee
In view of Lemma 3.1 and Corollary 4.1, (\ref{FR4}) will follow if we can show that for $M\ge 1$, 
\begin{multline} \label{FS4}
\limsup_{\del \ra 0} \; E\left[ q_\ve(x,y_\ve(T-\del),T-\del) \; ; \; y_\ve(T-\del) \le x-M\sqrt{\ve \del} \  \big|  \ y_\ve(t) = y \right] \le c(M),
\end{multline}
where the constant $c(M)$ satisfies $\lim_{M \ra \infty}\; c(M) = 0$.  To prove (\ref{FS4}) we use Lemma 6.1.  Thus let $t_0 <T$ be such that $CA(T-t_0) < 1/10$, where $C$ is the constant in (\ref{FL4}).  If in addition $A(T-t_0) < 1/10$ then $y_\ve(s)$ satisfies the differential inequality,
\be \label{FT4}
dy_\ve(s) \ge \left[ \frac 3 4 \ \frac{\{x-y_\ve(s)\}}{T-s} - 2 \sup_{s\le s' \le T} |b(x,s')| \right] ds + \sqrt{\ve}  \ dW(s),
\ee
provided $t_0 \le s < T$, and $y_\ve(s) < x$.  If $t \ge t_0$ then we see from (\ref{FT4}), by following the argument of Lemma 4.2, that for $\del/(T-t) < 1/K$,
\be \label{FU4}
P \left( y_\ve(T-\del)  < x-\rho \  \big| \  y_\ve(t) = y \right) \le \exp \left[ {-\rho^2}/{20\ve \del} \right],
\ee
provided $\rho \ge K\sqrt{\ve \del}$ and the constant $K$ depends only on $x,y$.  Evidently (\ref{FS4}) follows from (\ref{FU4}) on using Lemma 3.1.  If $t < t_0$ then one can argue as in Theorem 4.1 equation (\ref{AK4}) that the probability of $y_\ve(t_0)$ conditioned on $y_\ve(t) = y$ being very negative is extremely small.  Then one applies (\ref{FT4}) for $t_0 \le s < T$ to show that (\ref{FS4}) holds in this case also.  We have obtained the representation (\ref{FO4}).

To prove (\ref{FP4}) we proceed in a similar way to how we obtained the analogous representation (\ref{AC2}) in the classical case.  Thus let $y_\ve(s), t \le s < T$, with $y_\ve(t) = y$ be as before and for $\De y \in \R$ define $y_{\ve,\De y}(s)$ by
\be \label{FV4}
y_{\ve,\De y}(s) = y_\ve(s) + (T-s)\De y / (T-t), \quad  t \le s < T,
\ee
so that $y_{\ve,\De y}(t) = y+\De y$ and $y_{\ve,\De y}(s)$ satisfies the SDE
\be \label{FW4}
dy_{\ve,\De y}(s) = \left[ \la_\ve\left( y_{\ve,\De y}(s) - (T-s)\De y/(T-t), s \right) - \De y /(T-t) \right]ds + \sqrt{\ve}  \ dW(s), \;  \ 
t\le s < T.
\ee
Then by Lemma 3.2 there is the inequality
\begin{multline} \label{FX4}
q_\ve(x,y+ \De y,t) \le  E\left\{ q_\ve(x, y_\ve(T-\del) + \del \De y/(T-t), T-\del ) \  \big|  \ y_\ve(t) = y \right\}
\\
+ E \Big\{ \frac 1 2 \int^{T-\del}_t \left[ \la_\ve(y_\ve(s), s) - \De y/(T-t) - b(y_\ve(s) + (T-s)\De y / (T-t),s)\right]^2 ds \  \big| \  y_\ve(t) = y \Big\} \ , 
\end{multline}
where we have used the fact that (\ref{FV4}) gives the solution to (\ref{FW4}).  Since by the argument we used to establish (\ref{FO4}) one has that
\[ \lim_{\del \ra 0} \; E\left\{ q_\ve(x, y_\ve(T-\del) + \del \De y/(T-t), T-\del ) \  \big|  \ y_\ve(t) = y \right\}=0, \]
we conclude from (\ref{FX4}) that
\begin{multline} \label{FY4}
q_\ve(x,y+ \De y,t) \le \\
E \Big\{ \frac 1 2 \int^T_t \left[ \la_\ve(y_\ve(s), s) - \De y/(T-t) - b(y_\ve(s) + (T-s)\De y / (T-t),s)\right]^2 ds \  \big| \  y_\ve(t) = y \Big\}.
\end{multline}
To see that the RHS of (\ref{FY4}) is finite, it will be sufficient to show that
\be \label{FZ4}
E \left\{ \int^T_t \; \la_\ve (y_\ve(s), s)^2 \; ds \; | \; y_\ve(t) = y \right\} < \infty.
\ee
Observe that the inequality (\ref{FZ4}) does not follow in a straightforward way from the fact that $y_\ve(s)$ is a solution to (\ref{N1}), where $\la_\ve(\cd,\cd)$ is given by (\ref{P1}) and $-\pa q_\ve(x,y,t) /\pa y$ satisfies (\ref{FL4}).  In fact for $Z_\ve(s)$ the solution to (\ref{AB4}), it is easy to see that
\[
E \left\{ \int^T_t \ \frac{Z_\ve(s)^2}{(T-s)^2} \ ds \ \big|  \ Z_\ve(t) = z \right\} = \infty,
\]
for all $\mu > 0$.  To prove (\ref{FZ4}) we use the fact that the LHS of (\ref{FO4}) is finite.  Hence (\ref{FZ4}) follows if we can show that
\be \label{GA4}
E \left\{ \int^T_t \; y_\ve(s)^2 \; ds \; \big| \; y_\ve(t) = y \right\} < \infty.
\ee
It is easy to see that (\ref{GA4}) is a consequence of the fact that $\la_\ve(y,s) \ge b(y,s),  \ y \in \R,  \ t \le s < T$, and Lemma 3.4.  Here we use the fact that Lemma 3.4 implies that for any $\eta > 0$, $\la_\ve(y,s)$ is uniformly Lipschitz in $y$ in any region $y \ge x + \eta,  \ t \le s < T$.  Having established (\ref{FZ4}), we obtain from (\ref{FO4}), (\ref{FY4}) the inequality
\begin{multline} \label{GB4}
\limsup_{\De y \ra 0} [q_\ve(x,y + \De y, t) - q_\ve(x,y,t)] /\De y \le 
\\
-\frac{1}{T-t}E \left\{  \; \int^T_t \left[ 1 + (T-s) \frac{\pa b}{\pa y} (y_\ve(s),s) \right] 
\left[ \la_\ve(y_\ve(s), s) - b(y_\ve(s), s)\right] ds  \ \big| \  y_\ve(t) = y \right\}.
\end{multline}
Next in analogy to (\ref{FY4}) we have that
\begin{multline} \label{GC4}
q_\ve(x,y,t) \le \\
E \Big\{ \frac 1 2 \int^T_t \left[ \la_\ve(y_\ve(s), s) + \De y/(T-t) - b(y_\ve(s) - (T-s)\De y / (T-t),s)\right]^2 ds \  \big| \  y_\ve(t) = y + \De y \Big\}.
\end{multline}
Using now (\ref{FO4}) with $y$ replaced by $y + \De y$ we conclude from (\ref{GC4}) that
\begin{multline} \label{GD4}
\liminf_{\De y \ra 0} \  [q_\ve(x,y + \De y, t) - q_\ve(x,y,t)] /\De y \ge 
\\
- \frac{1}{T-t} E \left\{\; \int^T_t \left[ 1 + (T-s) \frac{\pa b}{\pa y} (y_\ve(s),s) \right] 
\left[ \la_\ve(y_\ve(s), s) - b(y_\ve(s), s)\right] ds \  \big| \  y_\ve(t) = y \right\},
\end{multline}
provided we show that 
\be \label{GE4}
\lim_{\eta \ra 0} \; E \left[ \int^T_t \; |y_{\ve, \eta}(s) - y_{\ve,0}(s)| \ ds \right]= 0,
\ee
\be \label{GF4}
\lim_{\eta \ra 0} \; E \left[ \int^T_t \; |\la_\ve(y_{\ve, \eta}(s),s) - \la_\ve(y_{\ve, 0}(s),s)| \ ds \right]= 0,
\ee
where $y_{\ve, \eta}(s), \; t \le s < T$, is the solution to (\ref{N1}) with initial condition $y_{\ve, \eta}(t) = y + \eta$.  To prove (\ref{GE4}) we use the uniform Lipschitz continuity of $\la_\ve(z, s)$ in any region $z \ge z_0, \  t \le s \le T-\del$, where $\del > 0,  \ z _0 \in \R$ can be arbitrary.  Thus by introducing a stopping time and using the fact that the probability of $y_{\ve, \eta}(s)$ being large and negative is very small we see that
\be \label{GG4}
\lim_{\eta \ra 0} \; E \left[ \int^{T-\del}_t \; |y_{\ve, \eta}(s) - y_{\ve,0}(s)| \ ds \right]= 0.
\ee
Now (\ref{GE4}) follows from (\ref{GA4}), (\ref{GG4})  using the fact that one can obtain a bound in (\ref{GA4}) which is uniform in $\eta$ for small $\eta$.  To prove (\ref{GF4})  first observe
that (\ref{FO4}) implies that
\be \label{GH4}
\sup_{|\eta| \le \eta_0} \; E \left\{ \int^T_t \; \la_\ve(y_{\ve, \eta}(s),s)^2 \ ds \right\} < \infty,
\ee
for any $\eta_0 > 0$.  Thus it is sufficient to show that 
\be \label{GI4}
\lim_{\eta \ra 0} \; E \left[ \int^{T-\del}_t \; |\la_\ve(y_{\ve, \eta}(s),s) - \la_\ve(y_{\ve,0}(s),s)| \ ds \right]= 0
\ee
for any $\del > 0$.  For any $z_0 \in \R$ we introduce a stopping time $\tau_\eta(z_0) = \inf\{ s \ge t : y_{\ve,\eta}(s) = z_0\}$.  From the uniform Lipschitz continuity of $\la_\ve(z,s)$ in $z \ge z_0, \ t \le s \le T-\del$, and (\ref{GG4}) we have that 
\be \label{GJ4}
\lim_{\eta \ra 0} \; E \left[ \int^{(T-\del) \wedge \tau_\eta(z_0) \wedge \tau_0(z_0)}_t \; |\la_\ve(y_{\ve, \eta}(s),s) - \la_\ve(y_{\ve,0}(s),s)| \ ds \right]= 0.
\ee
The expectation in (\ref{GI4}) exceeds the expectation in (\ref{GJ4}) by at most
\be \label{GK4}
2P\big( \tau_\eta(z_0) \wedge \tau_0(z_0) < T-\del \big) \sup_{t \le t' < T-\del} 
 E \left[ \int^{T-\del}_{t'} \; |\la_\ve(y_{\ve}(s),s) | \ ds \; \big| \  y_\ve(t') = z_0  \right].
\ee
Since $\la_\ve(z,s) \ge b(z,s), \; z \in \R, \; t \le s < T$, the probability in (\ref{GK4}) is decaying exponentially fast in $z_0$ as $z_0 \ra -\infty$.  In contrast the expectation in (\ref{GK4}) is increasing at most linearly in $|z_0|$ as $z_0 \ra -\infty$.  This follows from the representation  (\ref{FO4}) for $q_\ve$ and Lemma 3.1.  Hence the expression in (\ref{GK4}) converges to 0 as $z_0 \ra -\infty$, whence we conclude that (\ref{GI4}) follows from (\ref{GK4}).  We have proved (\ref{GF4}).  Now (\ref{FP4}) follows from (\ref{GB4}), (\ref{GD4}).  The proof of (\ref{FQ4}) is similar to the proof of (\ref{FP4}). 
\end{proof}
Once we have the representations in Theorem 6.1 for $q_\ve(x,y,t)$ and its first derivatives, the inequality (\ref{V1}) easily follows.
\begin{corollary} Suppose the function $b(\cdot,\cdot)$ satisfies the Lipschitz condition (\ref{A1}).  Then for $x,y\in\R, \ t<T$, the following inequalities hold:
\begin{eqnarray} \label{A5}
\left| \frac{\pa q_\ve}{\pa x} \; (x,y,t)\right| &\le& \Big[ 1 + (T-t)A\Big] \Big[2q_\ve(x,y,t) / (T-t)\Big]^{1/2}, \\
\left| \frac{\pa q_\ve}{\pa y} \; (x,y,t)\right| &\le& \Big[ 1 + (T-t)A\Big] \Big[2q_\ve(x,y,t) / (T-t)\Big]^{1/2}. \nn
\end{eqnarray}
\end{corollary}
\begin{proof} This follows from Theorem 6.1 on using the representations (\ref{FO4}), (\ref{FP4}), (\ref{FQ4}) and applying the Schwarz inequality in (\ref{FP4}), (\ref{FQ4}). 
\end{proof}

\section{Proof of Theorem 1.3}
In order to prove convergence of first derivatives in $x$ and $y$  of the function $q_\ve(x,y,t)$ defined by (\ref{G1}) to the corresponding derivatives of the function $q(x,y,t)$ defined by (\ref{J1}) as $\ve\ra 0$, it will generally be necessary to assume the concavity in $y$ of the function $b(y,t)$  in (\ref{B1}). Recall however that $q(x,y,t)=0$ if $y\ge F(x,t)$, where $F(\cdot,\cdot)$ is the function defined from (\ref{K1}). Thus for $y\ge F(x,t)$ the derivatives of $q(x,y,t)$ are $0$. In this case it easily follows from Corollary 6.1 that the derivatives in $x$ or $y$ of $q_\ve(x,y,t)$ converge to $0$ as $\ve\ra 0$, without making any further assumptions on the function $b(\cd,\cd)$ beyond the Lipschitz condition (\ref{A1}). 
\begin{corollary} Suppose $b(\cd, \cd)$ satisfies (\ref{A1}) and the function $F(\cdot,\cdot)$ is defined from (\ref{K1}).  Then for $0 < \ve \le 1$ there is a constant $C(x,y,t,T)$ such that
\be \label{B5}
\left| \frac{\pa q_\ve}{\pa x} \; (x,y,t)\right| + \left| \frac{\pa q_\ve}{\pa y} \; (x,y,t)\right| \le C(x,y,t,T)\ve^{1/4},
\ee
provided $y \ge F(x,t)$.
\end{corollary}
\begin{proof}   The inequality (\ref{B5}) follows from Theorem 1.1 and Corollary 6.1 since $q(x,y,t) = 0$ for $y \ge F(x,t)$.
\end{proof}
In order to show convergence when $y < F(x,t)$ we shall need to assume $b(\cdot,\cdot)$ is concave as well as that (\ref{A1}) holds.  We first prove a result about the classical problem.
\begin{lem}  For $\al \ge 0$ let $y_\al(s), \; 0 \le s \le T$, be the solution to the equation
\be \label{C5}
\frac{dy_\al}{ds} = b(y_\al(s), s) - \al \; \frac{\pa q}{\pa y} \; (x,y_\al(s), s), \quad  0 \le s < T, \ y_\al(0) = y,
\ee
where $q(\cdot,\cdot,\cdot)$ is the classical cost function (\ref{J1}). 
There there is a constant $C(AT)$ depending only on $AT$ such that
\be \label{D5}
0 \le y_\al(s) - y_0(s) \le \max [1,\al]  \ C(AT) \; \sqrt{Tq(x,y,0)}, \quad  0 \le s \le T.
\ee
\end{lem}
\begin{proof}   We first consider the case $\al = 1$ since $y_1(\cd)$ is the optimal trajectory for the variational problem (\ref{B2}).  From (\ref{B2}), (\ref{X2}) we see that
\be \label{E5}
C_1(AT) \left[ q(x,y,0)/T\right]^{1/2} \le \frac{dy_1(s)}{ds} - b(y_1(s), s) \le C_2(AT) \left[ q(x,y,0)/T\right]^{1/2}, \quad 0 \le s < T,
\ee
for some positive constants $C_1,C_2$ depending only on $AT$.  Setting $\vp_1(s) = y_1(s) - y_0(s)$ it follows from (\ref{C5}), (\ref{E5}) that
\be \label{F5}
|\vp'_1(s)| \le A\vp_1(s)  + C_2(AT) \left[q(x,y,0)/T\right]^{1/2}, \  0 \le s \le T, \quad \vp_1(0) = 0.
\ee
Applying Gronwall's inequality to (\ref{F5}) we conclude that (\ref{D5}) holds for $\al = 1$ and a-fortiori for $0 \le \al \le 1$.

Suppose now that $\al > 1$ in which case $y_\al(s) > y_1(s), \; 0 \le s \le T$.  Using the fact that $q(x,y,t)$ is convex in $y$, we see from (\ref{C5}) that
\[	\frac{dy_\al(s)}{ds} \le b( y_\al(s), s) - \al \ \frac{\pa q}{\pa y}( x, y_1(s), s), \quad  0 \le s < T.  \]
Thus if $\vp_\al(s) = y_\al(s) - y_0(s)$ we have that
\[   |\vp'_\al (s) | \le A\; \vp_\al (s) + \al \ C_2(AT) \left[ q(x,y,0)/T\right]^{1/2}, \]
whence (\ref{D5}) follows for $\al > 1$ as before. 
\end{proof}
We can use the method of Lemma 7.1 to find a region where the paths $y_\ve(s), \  0 \le s < T$, for the stochastic control problem (\ref{N1}), (\ref{O1}) are most likely to be found.
\begin{lem} Let $y_\ve(s), \; 0 \le s < T$, be the solution to the stochastic equation (\ref{N1}) with $y_\ve(0) = y$, where $\la_\ve(\cd,\cd)$ is given by (\ref{P1}).  Then there is a universal constant $M$ and a constant $C(AT)$ depending only on $AT$ such that
\be \label{G5}
P \left[ \inf_{0\le s < T} \big[ y_\ve(s) - y_0(s)\big] < -\rho \right] \le \exp \left[ -\rho^2/\ve T\;C(AT) \right],
\ee
provided $\rho^2 \ge M\ve T\;C(AT)$.  There is a further constant $C_1(x,y,A,T)$ depending only on $x,y,A,T$ such that
\be \label{H5}
P \left[ \sup_{0\le s < T} [ y_\ve(s) - y_0(s)] > \rho + C(AT)\sqrt{Tq(x,y,0)} + C_1(x,y,A,T)\ve^{1/4}\right]
\le \exp \left[ -\rho^2/\ve T\;C(AT)\right],
\ee
provided $\rho^2 \ge M \ve T \; C(AT)$.
\end{lem}
\begin{proof}   The inequality (\ref{G5}) is obtained by using the fact that $y_\ve(s) \ge Y_\ve(s), \; 0 \le s < T$, where $Y_\ve(0) = y$ and $Y_\ve(\cd)$ satisfies (\ref{E1}).  Then one compares solutions of (\ref{E1}) to solutions of the deterministic equation (\ref{C5}) with $\al=0$, using the Lipschitz property (\ref{A1}) of $b(\cd,\cd)$ and applying Gronwall's inequality.

To obtain the inequality (\ref{H5}) we need to use the convexity of the function $q_\ve(x,y,s)$ in $y$, which is established in the Appendix (Theorem A1).  Let $y_c(s),  \ 0 \le s < T$, be the optimal trajectory $y(\cdot)$ for the variational problem (\ref{J1}) with $y(0) = y$.  Then if $y > y_c(s)$ we have from Corollary 6.1 that
\be \label{I5}
0 \le -\frac{\pa q_\ve}{\pa y} (x,y,s) \le -\frac{\pa q_\ve}{\pa y} (x,y_c(s),s) \le (1 + AT)\left[2q_\ve(x,y_c(s),s) / (T-s)\right]^{1/2}.
\ee
From Lemma 3.3 we see that there is a constant $C_2(x,y,A,T)$ depending only on $x,y,A,T$ such that
\be \label{J5}
q_\ve(x,y_c(s),s)  \le q(x,y_c(s),s) + C_2(x, y, A, T) \sqrt{\ve}, \quad  0 \le s < T.
\ee
Putting (\ref{I5}), (\ref{J5}) together and using the fact that (\ref{E5}) holds for $y_c(\cd)$, we conclude that
\begin{multline} \label{K5}
0 \le - \frac{\pa q_\ve}{\pa y} (x,y,s) \le C_1(AT) \left[ q(x,y,0)/T\right]^{1/2} \\
+ C_3(x,y,A,T) \ve^{1/4}\big/ \sqrt{T-s}, 
\quad  0 \le s < T, \; y > y_c(s).
\end{multline}
Consider now the diffusion process $Z_\ve(\cd)$ defined as a solution to the stochastic equation
\be \label{L5}
dZ_\ve(s) = \mu_\ve( Z_\ve(s),s) ds + \sqrt{\ve} \ dW(s), \quad 0 \le s < T,
\ee
where $\mu_\ve(\cd, \cd)$ is given by the formula
\begin{eqnarray} \label{M5}
\mu_\ve(z,s) &=& b(z,s) - \frac{\pa q_\ve}{\pa y} (x,z,s), \quad z < y_c(s), \\
\mu_\ve(z,s) &=& b(z,s) + C_1(AT)\left[q(x,y,0)/T\right]^{1/2} \nn \\
    & \qquad & \qquad + C_3(x,y,A,T) \ \ve^{1/4}\big/\sqrt{T-s}, \quad  z > y_c(s). \nn
\end{eqnarray}
Then if $Z_\ve(0) \ge y_\ve(0)$, it follows from (\ref{K5}) that $Z_\ve(s) \ge y_\ve(s),  \ 0 \le s < T$, with probability 1.

For any $t, \; 0 \le t < T$, suppose that $z_0 > y_c(t)$ and consider the solution $z(s)$ to the initial value problem 
\be \label{N5}
dz(s) = \mu_\ve( z(s), s)ds, \quad t \le s < T, \ z(t) = z_0.
\ee
By letting $\ve \ra 0$ in (\ref{K5}) we see that $z(s) > y_c(s), \ t < s \le T$.  Hence on setting $\phi(s) = z(s) - y_c(s)$ we have from (\ref{N5}) and the Lipschitz property of $b(\cd, \cd)$ that
\be \label{O5}
-A\phi (s) \le \phi'(s) \le A\phi(s) +  C_1(AT)\left[q(x,y,0)/T\right]^{1/2} + C_3(x,y,A,T) \ \ve^{1/4}\big/\sqrt{T-s}, \quad  t\le s < T.
\ee
Integrating (\ref{O5}) we conclude that
\begin{multline} \label{P5}
[ z_0 - y_c(t)]e^{-AT} \le z(s) - y_c(s) \le \\
e^{AT} \left\{ \left[ z_0 - y_c(t)\right] +  C_1(AT)\left[Tq(x,y,0)\right]^{1/2} + 2\sqrt{T}\; C_3(x,y,A,T) \  \ve^{1/4} \right\}, \quad t\le s<T. 
\end{multline}

We can compare the solution of (\ref{N5}) to the solution of the stochastic equation (\ref{L5}) with initial condition $Z_\ve(t) = z_0 > y_c(t)$.  Arguing as in Lemma 3.1 we see that
\be \label{Q5}
P \left( \sup_{t \le s < T} |Z_\ve(s) - z(s) | > \del \right) \le \exp \Big[-\del^2 / \ve T \ C_2 (AT)\Big],
\ee
where the constant $C_2(AT)$ depends only on $AT$.  Also $\del$ must satisfy the inequalities
\be \label{R5}
\del < \left[z_0 - y_c(t)\right]e^{-AT}, \quad  \del^2 \ge M\ve T \  C_2(AT),
\ee
where $M$ is a universal constant.  The first inequality in (\ref{R5}) ensures by (\ref{P5}) that if $|Z_\ve(s) - z(s)| < \del$ then $Z_\ve(s) > y_c(s)$.  Hence to estimate the probability (\ref{Q5}) we can assume the drift $\mu_\ve(\cdot,\cdot)$ of (\ref{L5}) is given by the second formula in (\ref{M5}).

To prove (\ref{H5}) first observe that the probability in (\ref{H5}) is bounded above by the probability
\be \label{S5}
\sup_{0 \le t < T} P \left[ \sup_{t \le s < T} [y_\ve(s) - y_0(s)] > \frac{1+2e^{AT}}{2(1+e^{AT})} \; \eta
 \ \Big| \  y_\ve(t) - y_0(t) = \frac{\eta}{2(1 + e^{AT})} \right],
\ee
where $\eta$ is given by the formula
\be \label{T5}
\eta = \rho + C(AT) \sqrt{Tq(x,y,0)} + C_1(x,y,A,T) \  \ve^{1/4}.
\ee
The probability in (\ref{S5}) is in turn bounded above by the same probability with $y_\ve(s)$ replaced by $Z_\ve(s)$.  Observe next that
\be \label{U5}
Z_\ve(s) - y_0(s) \ > \ \frac{1+2e^{AT}}{2(1+e^{AT})}\ \eta \ \Longrightarrow \ Z_\ve(s)-z(s) > \eta/4 \ ,
\ee
where we have used the fact that $z_0 - y_0(t) = \eta/2(1+e^{AT})$ and the inequalities (\ref{D5}), (\ref{P5}).  The constants $C(AT)$ and $C_1(x,y,A,T)$ in (\ref{T5}) must also be chosen sufficiently large.  Hence the probability in (\ref{S5}) is bounded above by the probability
\be \label{V5}
P \left( \sup_{t\le s < T} | Z_\ve(s)-z(s)| > \frac{\eta e^{-AT}}{4(1+e^{AT})}  \ \Big| \; Z_\ve(t) - y_0(t) = \frac{\eta}{2(1+e^{AT})} \right).
\ee
It is clear from (\ref{D5}) that if the constant $C(AT)$ in (\ref{T5}) is chosen sufficiently large then we may apply (\ref{Q5}) to estimate (\ref{V5}), since for $C(AT)$ large enough the first inequality in (\ref{R5})  is satisfied. Now (\ref{H5}) follows from (\ref{Q5}) since the condition on $\rho$ implies the second inequality in (\ref{R5}). 
\end{proof}
\begin{lem}  Let $y_\ve(s), \; 0 \le s < T$, be as in Lemma 7.2 and $y_c(s), \; 0 \le s < T$, be the solution to the corresponding classical problem (\ref{J1}) which has optimal controller $\la_c(s),  \ 0 \le s < T$.  Then there is a constant $C(x,y,A,T)$ such that
\be \label{W5}
E \left\{ \int^T_0 \left[ \la_\ve( y_\ve(s),s) - b( y_\ve(s),s) - \la_c(s) + b(y_c(s), s) \right]^2\; ds \right\} \le
C(x,y,A,T)\ve^{1/4}.
\ee
\end{lem}
\begin{proof} Following the argument of Lemma 3.5 we define a classical path $y_{\ve,c}(\cd)$ which corresponds to the stochastic path $y_\ve(\cd)$ by
\be \label{X5}
\frac{dy_{\ve,c}(s)}{ds}= \la_\ve( y_\ve(s),s) + k/T, \quad 0 \le s < T,
\ee
where $y_{\ve,c}(0) = y$ and $k$ is defined by
\be \label{Y5}
k = \max \left[ x - y - \int^T_0 \; \la_\ve( y_\ve(s),s)ds, \  0 \right].
\ee
Observe from Lemma 4.1 and Theorem 4.1 that the integral on the RHS of (\ref{Y5}) exists with probability 1.  Letting $\al$ be an arbitrary number, $0 < \al < 1$, and using the fact that $y_{\ve,c}(T) \ge x$, we have that
\begin{multline} \label{Z5}
q(x,y,0) \le {\mathcal F} \left[ \al \; y_{\ve,c}(\cd) + (1-\al)y_c(\cd) \right] = \frac 1 2 \; \int^T_0 \Big[ \al \{ \la_\ve( y_\ve(s),s) - \; b( y_\ve(s),s)\} \\
+ (1-\al)\{ \la_c(s) - b( y_c(s),s)\} + g_\ve(s) - h( y_\ve(s),s) \Big]^2 \; ds ,  
\end{multline} 
where the deterministic function $h(z,s)$ is given by the formula
\be \label{AA5}
h(z,s) = b( \al z + (1-\al) y_c(s), s) - \al b (z,s) - (1-\al) b(  y_c(s), s) ,
\ee
and the random function $g_\ve(s)$ by the formula,
\be \label{AB5}
g_\ve(s) = \al k/T +  b( \al y_\ve(s)  + (1-\al) y_c(s), s) -  b( \al y_{\ve,c}(s) + (1-\al) y_c(s), s).
\ee
We expand out $g_\ve(\cd)$ in the quadratic expression in (\ref{Z5}) to obtain the inequality
\begin{multline} \label{AC5}
q(x,y,0) \le \frac 1 2 \int^T_0 g_\ve(s)^2 \; ds +  \int^T_0 |g_\ve(s)| \; |h( y_\ve(s),s)|  \ ds  \ + \\
 \int^T_0 |g_\ve(s)| \; | \la_c(s) - b( y_c(s),s) | \ ds +  \int^T_0 |g_\ve(s)| \; | \la_\ve(y_\ve(s),s) - b( y_\ve(s),s )| \ ds \  + \\
\frac 1 2\; \int^T_0 \Big[ \al\{ \la_\ve( y_\ve(s),s) -  b( y_\ve(s),s)\} 
+ (1-\al)\{ \la_c(s) - b( y_c(s),s)\} - h( y_\ve(s),s) \Big]^2 \; ds  .
\end{multline}
Since $b(\cd,s)$ is concave for $0 \le s < T$, it follows that the function $h$ is non-negative.  Thus since 
$ [\la_\ve( y_\ve(s),s) -  b( y_\ve(s),s)]$ and $[\la_c(s) - b( y_c(s),s)]$ are both non-negative, one has the inequality
\begin{multline} \label{AD5}
 \frac 1 2\; \int^T_0 \Big[ \al\{ \la_\ve( y_\ve(s),s) -  b( y_\ve(s),s)\} + (1-\al)\{ \la_c(s) - b( y_c(s),s)\} 
 - h( y_\ve(s),s) \Big]^2 \; ds \\
  \le \frac \al 2  \int^T_0 \Big[ \la_\ve(y_\ve(s),s) - b( y_\ve(s),s)\Big]^2 \;ds 
+  \frac{1-\al} 2 \int^T_0 \Big[ \la_c(s) - b( y_c(s),s)\Big]^2 \ ds  \\
- \frac {\al(1-\al)} 2  \int^T_0 \Big[ \la_\ve(y_\ve(s),s) - b( y_\ve(s),s) -  \la_c(s) + b(y_c(s),s)\Big]^2 \ ds  ,
\end{multline}
provided that $h(y_\ve(s),s) \le 2 (1-\al)[  \la_c(s) - b(y_c(s),s)]$ for $0 \le s < T$.  Since we also have that $h(z,s) \le 2A\al \ (1-\al)|z - y_c(s)|$, we conclude that (\ref{AD5}) holds provided $y_\ve(\cd)$ satisfies the inequality
\be \label{AE5}
A\al \  | y_\ve(s)  -  y_c(s)| \le [\la_c(s) - b( y_c(s),s)], \quad  0 \le s < T.
\ee
If we now use (\ref{D5}) and the lower bound in (\ref{E5}) we see that (\ref{AE5}) is implied by the inequality
\be \label{AF5}
| y_\ve(s)  -  y_0(s)| \le \left[ \al^{-1} \; C_1(AT) - C_2(AT) \right] \sqrt{Tq(x,y,0)}, \quad  0 \le s < T,
\ee
for some positive universal constants $C_1(AT)$, $C_2(AT)$ depending only on $AT$.

Observe now that from Theorem 1.1 the inequality (\ref{W5}) holds if $Tq(x,y,0) \le \ve^{1/4}$, whence we may assume $Tq(x,y,0) > \ve^{1/4}$.  It follows then from Lemma 7.2 that, for $\al$ sufficiently small depending only on $AT$ and $\ve$ sufficiently small depending only on $x,y,A,T$ the inequality (\ref{AF5}) holds with probability close to 1.

We estimate the expectation of the terms in $g_\ve(\cd)$ on the RHS of (\ref{AC5}).  From Theorem 4.1 it follows that the quantity $k$ in (\ref{Y5}) satisfies the inequality
\be \label{AG5}
0 \le k \le \sqrt{\ve} \ \max[W(T), 0],
\ee
where $W(\cd)$ is Brownian motion.  We also have from (\ref{N1}),(\ref{X5}) that
\be \label{AH5}
\sup_{0 \le s < T} | y_\ve(s) - y_{\ve,c}(s) | \le  \sqrt{\ve} \sup_{0 \le s < T} |W(s)| + k.
\ee
We may bound the random function $g_\ve(\cd)$ of (\ref{AB5}) using (\ref{AG5}), (\ref{AH5}) to obtain
\be \label{AI5}
\sup_{0 \le s < T}  |g_\ve(s)| \le \frac{2\al\sqrt{\ve}} T \ [1 + AT] \sup_{0 \le s < T} |W(s)| .
\ee
Evidently (\ref{AI5}) implies that
\be \label{AJ5}
E \left[ \int^T_0 \; g_\ve(s)^2 ds \right] \ \le \ \al^2  \ve  \  C_3(AT)
\ee
for a constant $C_3(AT)$ depending only on $AT$.  The inequality (\ref{AJ5}) in turn implies by the Schwarz inequality that
\be \label{AK5}
E \left[ \int^T_0 \; |g_\ve(s)|  |\la_c(s) - b(y_c(s),s)|  \ ds \right] \ \le \ \ve^{1/2} \al \  C_4(AT) [1 + q(x,y,0)]
\ee
for a constant $C_4(AT)$ depending only on $AT$.  Similarly one has by Theorem 1.1 that
\begin{multline} \label{AL5}
E \left[ \int^T_0 |g_\ve(s)| \; | \la_\ve( y_\ve(s),s) - b( y_\ve(s),s) | \ ds \right] 
\\
\le \sqrt{\ve} \; \al \; C_5(AT) \Big[ 1 + q(x,y,0) + C_6(x,y,A,T)\sqrt{\ve} \big], 
\end{multline}
for constants $C_5(AT)$ depending only on $AT$ and $C_6(x,y,A,T)$ on $x,y,A,T$.  The final term involving $g_\ve(\cd)$ can be estimated by using Lemma 7.2.  Thus
\begin{multline} \label{AM5}
E \left[ \int^T_0 |g_\ve(s)| \ |h( y_\ve(s), s)| \ ds \right] \le \al \sqrt{\ve} \ C_3(AT)^{1/2}
\\
2A\al(1-\al) E  \left[ \int^T_0 | y_\ve(s) - y_c(s)|^2 \; ds \right]^{1/2},  
\end{multline}
and the expectation on the RHS of (\ref{AM5}) is bounded as
\be \label{AN5}
E  \left[ \int^T_0 | y_\ve(s) - y_c(s)|^2 \; ds \right] \le C_7(AT) \left\{ T^2q(x,y,0) + C_8(x,y,A,T)\ve^{1/2} \right\}.
\ee

Let us define now $p_\ve$ as the probability that the inequality (\ref{AF5}) is violated, and take the expectation of (\ref{AC5}) over the event (\ref{AF5}).  Thus from (\ref{AC5}), (\ref{AD5}) and (\ref{AJ5}) - (\ref{AN5}) we conclude that
\begin{multline}  \label{AO5}
\frac {\al(1-\al)} 2 E\left[ \int^T_0  \left[ \la_\ve( y_\ve(s),s) -  b( y_\ve(s),s) - \la_c(s) + b( y_c(s),s) \right]^2 \; ds \ {\rm; \  (\ref{AF5}) \ holds} \right]
\\
\le \left[ p_\ve + \ve^{1/2} \al \; C_9(AT) \right] q(x,y,0) + \al C_{10}(x,y,A,T)\ve^{1/4}.  
\end{multline}
Since we can estimate $p_\ve$ from Lemma 7.2, we can conclude (\ref{W5}) from (\ref{AO5}) provided we can estimate the expectation
\be \label{AP5}
E\left[ \int^T_0  \left[ \la_\ve( y_\ve(s),s) -  b( y_\ve(s),s) \right]^2 \; ds  \ {\rm; \  (\ref{AF5}) \ does \ not \  hold} \right]
\ee
appropriately.  We have now from Corollary 6.1 that
\begin{multline} \label{AQ5}
E\left[ \int^{T-\del}_0  \left[ \la_\ve( y_\ve(s),s) -  b( y_\ve(s),s) \right]^2 \; ds  \ {\rm; \  (\ref{AF5}) \ does \ not \  hold} \right]
\\
 \le 4p^{1/2}_\ve (1 + AT)^4 \ \int^{T-\del}_0 \frac{ds}{T-s} \ E\left[ q_\ve(x, y_\ve(s),s)^2 \right]^{1/2} .  
 \end{multline}
Let  $Y_\ve(s), \ s\ge 0$, be  the solution to (\ref{E1}) with $Y_\ve(0) = y$.  Recall that since $\pa q_\ve(x,z,s) / \pa z \le 0$ we have that $y_\ve(s) \ge Y_\ve(s), \ s\ge 0$ . Using Lemma 3.1 then, we conclude that
\be \label{AT5}
 E\left[ q_\ve(x, y_\ve(s),s)^2 \right] \le \frac 1{(T-s)^2} \big[ C_3(AT) E\left[ \{ x-Y_\ve(s)\}^4\right] + C_4(x,y,A,T) \big].
\ee

We are left now to estimate
\be \label{AU5}
E\left[ \int^T_{T-\del} \left[ \la_\ve( y_\ve(s),s) -  b( y_\ve(s),s) \right]^2 \; ds  \ {\rm; \  (\ref{AF5}) \ does \ not \  hold} \right]
\ee
for some $\del > 0$.  Instead of attempting to show that the expectation (\ref{AU5}) is small, we consider as in (\ref{AD5}) under what circumstances the inequality
\begin{multline} \label{AV5}
\Big[ \al \{ \la( y_\ve(s),s) -  b( y_\ve(s),s) \} + (1-\al) \{ \la_c(s) - b( y_c(s),s) \} - h( y_\ve(s),s) \Big]^2  \\
\le \frac \al 2  \big[ \la_\ve(y_\ve(s),s) - b( y_\ve(s),s)\big]^2  
+  \frac{(1-\al)} 2 \big[ \la_c(s) - b( y_c(s),s)\big]^2   \\
- \frac {\al(1-\al)} 2  \big[ \la_\ve(y_\ve(s),s) - b( y_\ve(s),s) -  \la_c(s) + b(y_c(s),s)\big]^2 
\end{multline}
holds if $s$ lies in the region $T-\del < s < T$.  From Theorem 4.1 we see that if $\del > 0$ is sufficiently small depending only on $A$, then (\ref{AV5}) holds if $y_\ve(s)$ satisfies the one sided inequality
\be \label{AW5}
y_\ve(s) - y_0(s) \le \big[ \al^{-1} C_1(AT) - C_2(AT) \big] \sqrt{Tq(x,y,0)}
\ee
similar to (\ref{AF5}).  Thus instead of estimating (\ref{AU5}) it will be sufficient to estimate
\be \label{AX5}
E \left[ \int^T_{T-\del} \big[ \la_\ve(y_\ve(s),s) - b( y_\ve(s),s) \big]^2 \chi(y_\ve(s),s)ds \right],
\ee
where
\begin{eqnarray} \label{AY5}
\chi(z,s) &=& 1 \ \ {\rm if} \ z > y_0(s) + \left[ \al^{-1}C_1(AT) - C_2(AT)\right] \sqrt{Tq(x,y,0)}, \\
\chi(z,s) &=& 0, \ \ {\rm otherwise}. \nn
\end{eqnarray}
Using (\ref{I5}), (\ref{J5}) we see that if $\chi( y_\ve(s),s) = 1$ then $\left[ \la_\ve( y_\ve(s),s) - b( y_\ve(s),s )\right]^2$ is bounded, whence we conclude that the expectation  (\ref{AX5}) is bounded by $C(x,y,A,T)$ for a constant $C(x,y,A,T)$.  The result follows from Lemma 7.2. 
\end{proof}
\begin{proof}[Proof of Theorem 1.3]  We use the representation for $\pa q(x,y,t)/\pa y$ given by (\ref{AC2}) and for $\pa q_\ve(x,y,t)/\pa y$ by (\ref{FP4}).  Thus we have that 
\begin{multline} \label{BA5}
 \frac{\pa q}{\pa y}(x,y,0) - \frac{\pa q_\ve}{\pa y}(x,y,0) = \\
  \frac{1}{T} E \left\{ \int^T_0 \left[ 1 + (T-s) \frac{\pa b}{\pa y}( y_\ve(s),s)\right] 
 \left[  \la_\ve( y_\ve(s),s) -  b( y_\ve(s),s) - \la_c(s) + b( y_c(s),s) \right] ds  \right\} \\
+ \frac 1 T \; E\left\{ \int^T_0 (T-s) 
\left[ \frac{\pa b}{\pa y}( y_\ve(s),s) - \frac{\pa b}{\pa y}( y_c(s),s) \right] \left[ \la_c(s) - b( y_c(s),s) \right] ds \right\} . 
\end{multline}
In view of Lemma 7.3 the second identity of (\ref{Y1}) follows if we can show that 
\be \label{BB5}
\lim_{\ve \ra 0} \ E \left\{ \int^T_0 \Big| \frac{\pa b}{\pa y}( y_\ve(s),s) - \frac{\pa b}{\pa y}( y_\ve(s),s) \Big| ds \right\} = 0.
\ee
We put $\phi_\ve(s) = y_\ve(s) - y_c(s), \ 0 \le s < T$, and observe that $\phi_\ve(s)$ satisfies the equation,
\be \label{BC5}
d\phi_\ve = \left[\la_\ve( y_\ve(s),s) -  b( y_\ve(s),s) - \la_c(s) + b( y_\ve(s),s) \right] ds + \sqrt{\ve} \; dW(s), \quad \phi(0) = 0.
\ee
It follows from (\ref{BC5}) that
\be \label{BD5}
\sup_{0\le s < T}|\phi_\ve(s)| \le \int^T_0 \big| \la_\ve( y_\ve(s),s) -  b( y_\ve(s),s) - \la_c(s) + b( y_\ve(s),s) \big| ds + \sqrt{\ve}  \  \sup_{0 \le s < T} |W(s)|.
\ee
One easily sees from (\ref{BD5}) and Lemma 7.3 that (\ref{BB5}) holds.  We have proved the second identity of (\ref{Y1}).  The first identity follows in a similar way. 
\end{proof}

\appendix
\section{Log Concavity of Solutions to Linear Diffusion Equations }
Our goal in this appendix is to establish convexity properties of the function $q_\ve(x,y,t)$ defined by (\ref{G1}).  We shall first show convexity in $y$ for fixed $x \in \R$, $t < T$, since showing joint convexity in $(x,y)$ is considerably more difficult.  We consider the terminal-boundary value problem
\be \label{A6}
\frac{\pa w}{\pa t} + b(y,t) \ \frac{\pa w}{\pa y} + \frac \ve 2 \ \frac{\pa^2 w}{\pa y^2} = 0, \quad y > 0, \ t < T,
\ee
\[   w(y, T) = w_0(y), \  \ y > 0; \qquad  \ w(0,t) = 0, \ \ t < T.	\]
\begin{proposition}  Assume $b(\cd,\cd)$ satisfies (\ref{A1}) and the terminal function $w_0(y)$ is $C^2$ for $y > 0$ and $C^1$ for $y \ge 0$ with $w_0(0) = 0$.  Assume further that
\be \label{B6}
\sup_{y>0} \left\{ |w_0(y)| + |dw_0(y)/dy| + |d^2w_0(y)/dy^2| \right\} < \infty.
\ee
Then there is a unique solution $w(y,t), \; y > 0, \; t < T$, to the terminal-boundary value problem (\ref{A6}) which has the property that $w(y,t)$ is $C^2$ in $y$, $C^1$ in $t$, and satisfies the inequality
\be \label{C6}
\sup_{y>0, T_0 < t < T} \left\{ |w(y,t)| + |\pa w(y,t)/\pa y| + |\pa^2 w(y,t)/\pa y^2|\right\} < \infty
\ee
for any $T_0 < T$.  In addition, the functions $w(y,t)$ and $\pa w(y,t)/\pa y$ are continuous for $y \ge 0, \; t \le T$.
\end{proposition}
\begin{proof} 
We first observe that the result holds when $b \equiv 0$.  In this case the solution is given by the method of images as
\be \label{D6}
w(y,t) = \int^\infty_0 \left[ G(y - y', \ve(T-t)) - G(y+y', \ve(T-t)) \right] w_0(y') \ dy',
\ee
where $G(\cdot,s)$ is the pdf of the Gaussian variable with mean $0$ and variance $s$.
Thus on using integration by parts we have
\be    \label{E6}
\frac{\pa w}{\pa y} (y,t) = \int^\infty_0 \left[ G(y - y', \ve(T-t)) + G(y+y', \ve(T-t)) \right] \frac{dw_0(y')}{dy'}\;dy',
\ee
where we have used the fact that $w_0(0) = 0$ in deriving (\ref{E6}).  On using a further integration by parts we have that
\be \label{F6}
\frac{\pa^2 w}{\pa y^2} (y,t) = \int^\infty_0 \left[ G(y - y', \ve(T-t)) - G(y+y', \ve(T-t)) \right] \frac{d^2w_0(y')}{dy'^2}\;dy'.
\ee
It follows easily from (\ref{D6}), (\ref{E6}), (\ref{F6}) that (\ref{C6}) holds.  In addition $w(y,t)$ and $\pa w(y,t)/\pa y$ are continuous for $y \ge 0, t \le T$.  We also have that $\pa^2 w(y,t)/\pa y^2$ is continuous for $y > 0, \; t \le T$, provided $d^2 w_0(y)/dy^2$ is continuous in $y > 0$.

To prove the result for general $b(\cd,\cd)$ satisfying (\ref{A1}) it will be sufficient to establish it for $t$ restricted to a small interval $[T - \De, \; T]$.  We proceed as in Lemma 3.4.  Taking $y_1 = \eta$ in (\ref{V3}) we see from (\ref{X3}) that $w(y,t)$ is given by the formula
\be \label{G6}
w(y,t) = \int^{2\eta}_0 G(y, y',t,T)w_0(y')dy' - \ve \; \int^T_t ds \  w_+(s) \; \frac{\pa G}{\pa y'} \; (y, 2\eta, t, s),
\ee
provided $0 < y < 2\eta$.  The Green's function $G(y,y',t,T)$ is defined by the perturbation expansion (\ref{AA3}).  Since $w_+(\cd)$ is bounded by virtue of (\ref{B6}), we see that if $\De$ satisfies (\ref{Y3}) then $\sup\{|w(y,t)|: 0 < y \le \eta$, $T-\De \le t < T\} < \infty$ and $w(y,t)$ is continuous for $0 \le y \le \eta, \ T-\De \le t \le T$, with $w(0,t) = 0$.

We consider next the first derivative $\pa w(y,t) / \pa y$, which from (\ref{G6}) is given by the formula
\be \label{H6}
\frac{\pa w}{\pa y}(y,t) = \int^{2\eta}_0 \frac{\pa G}{\pa y} (y, y',t,T)w_0(y')dy' - \ve \; \int^T_t ds \  w_+(s) \; \frac{\pa^2 G}{\pa y\pa y'} \; (y, 2\eta, t, s).
\ee
It is evident from (\ref{AT3}) that the second integral on the RHS of (\ref{H6}) is uniformly bounded in the set $\{(y,t):0 < y \le \eta,  \ T-\De \le t < T\}$ and that the integral converges to 0 as $t \ra T$, uniformly for $0 < y \le \eta$.  To estimate the first integral on the RHS of (\ref{H6}) we do an integration by parts for the first term in the perturbation expansion (\ref{AA3}) for $G(y,y',t,T)$.  Just as in (\ref{E6}) we see that this term is uniformly bounded in the set $\{(y,t):0 < y \le \eta,  \ T-\De \le t < T\}$, and converges uniformly to $dw_0(y)/dy$ as $t \ra T$.  We can estimate the higher order terms
\be \label{J6}
\int^{2\eta}_0 \ \frac{\pa v_n}{\pa y} (y, y',t,T) \ w_0(y') \ dy' ,
\ee
for $n \ge 0$ simply by using (\ref{AB3}).  Thus we see that the sum of the higher order terms is uniformly bounded in the set $\{(y,t):0 < y \le \eta,  \ T-\De \le t < T\}$.  To prove continuity of $\pa w(y,t)/\pa y$ as $t \ra T$ we need to show that the integral in (\ref{J6}) converges uniformly to 0 as $t \ra T$ in the interval $0 < y \le \eta$.  This follows from (\ref{AB3}) when $n \ge 1$.  To prove it for $n = 0$ we again need to make use of integration by parts.  Thus we see that 
\begin{multline} \label{K6}
\big| \int^{2\eta}_0 g_0(z,y',s,T)  \ w_0(y') \ dy' \  \big| \le  \\
C|b(z,s)| \left[ \sup_{0 < y \le 2\eta}|dw_0(y)/dy | +G(z-2\eta, 2\ve(T-s)) |w_0(2\eta)| \right] ,  
\end{multline}
for some universal constant $C$.  It follows from ({\ref{K6}) and the representation (\ref{AA3}) for $v_0$ that the integral (\ref{J6}) also converges to 0 as $t \ra T$ when $n=0$.  We have shown that $\sup \{ |\pa w(y,t)/\pa y| : 0 < y \le \eta,  \ T-\De \le t < T\} < \infty$ and $\pa w(y,t)/\pa y$ is continuous for $0 \le y \le \eta$, $T-\De \le t \le T$.

To estimate the second derivative $\pa^2w(y,t)/\pa y^2$ we proceed in a similar manner.  Thus we have that
\be \label{L6}
\frac{\pa^2 w}{\pa y^2}(y,t) = \int^{2\eta}_0 \frac{\pa^2 G}{\pa y^2} (y, y',t,T)w_0(y')dy' - \ve \; \int^T_t ds \  w_+(s) \; \frac{\pa^3G}{\pa y^2\pa y'} \; (y, 2\eta, t, s).
\ee
We wish to show that $\sup \{ |\pa^2 w(y,t)/\pa y^2| : 0 < y \le \eta, \  T-\De \le t < T\} < \infty$.  In view of (\ref{AY3}) it is sufficient to consider only the first integral on the RHS of (\ref{L6}).  We estimate the first term in the perturbation expansion (\ref{AA3}) for $G(y,y',t,T)$ using integration by parts as in (\ref{F6}).  The higher order terms, corresponding to $v_n(y,y',t,T)$ with $n \ge 1$, can be estimated using (\ref{AR3}), so we are only left to deal with the term corresponding to $v_0(y,y',t,T)$.  We can estimate this by using (\ref{K6}) and the corresponding inequality for the derivative of $g_0$, 
\begin{multline}  \label{MM6}
\left| \int^{2\eta}_0 \frac{\pa g_0}{\pa z} (z,y',s,T) w_0(y')dy'\right|  \le \\
C\left[ A + \left\{\frac \nu{\De(T-t)}\right\}^{1/2} \right] 
\left[ \sup_{0 < y \le 2\eta} |dw_0(y)/ dy| +G(z-2\eta, 2\ve(T-s)) |w_0(2\eta)| \right] ,  
\end{multline}
for some universal constant $C$.  We have shown that $\sup \{ |\pa^2 w(y,t)/\pa y^2| : 0 < y \le \eta,  \ T-\De \le t < T\} < \infty$.

We can easily extend the estimates we have made on $w(y,t)$ and its $y$ derivatives in the set $\{(y,t) : 0 < y \le \eta, \  T-\De \le t < T\}$ to all of $y > 0$ by observing that the function $v(z,t)$ defined by $v(z,t) = w(z + y(t), t)$, where $y(s), s \le T$, is a solution to (\ref{K1}) with $y(T) = y_1$, satisfies the PDE
\be \label{NN6}
\frac{\pa v}{\pa t} + \left[ b(z + y(t), t) - b(y(t), t)) \right] \frac{\pa v}{\pa z} + \frac \ve 2 \; \frac{\pa^2v}{\pa z^2} =0.
\ee
Then we represent $v(z,t)$ by a formula similar to (\ref{G6}) and use perturbation theory as before, observing that the perturbation series for the Green's function converges in a region $\{ |z|< \eta,  \ T-\De \le t< T\}$, where $\eta,\De$ can be taken independent of $y_1$. 
}
\end{proof}
\begin{theorem} Suppose $b(\cd,\cd)$ satisfies (\ref{A1}), and in addition the function $b(y,t)$ is concave in $y$ for $y \in \R,  \ t \le T$.  Then for any fixed $x \in \R,  \ 0 \le t < T$, the function $q_\ve(x,y,t)$  of (\ref{G1}) is a convex function of $y \in \R$.
\end{theorem}
\begin{proof} We shall take wlog $x=0$.  For $\del$ satisfying $0 < \del < 1$ we define a function $g_\del(z)$ with domain $\{z\in \R : z > -1\}$ by
\begin{eqnarray} \label{OO6}
g''_\del(z) &=& 1/(1+z)^2,  \ \ \quad {\rm if \ }-1 < z < -(1-\del), \\
g''_\del(z) &=& -z/\del^2(1-\del)   \quad {\rm if \ }\ -(1-\del) < z < 0, \nn \\
g_\del(0) &=& g'_\del(0) = 0, \nn \\
g''_\del(z) &=& 0, \quad {\rm if \ }  z > 0. \nn 
\end{eqnarray}
Evidently $g_\del(z)$ is a $C^2$ convex decreasing function which has the property that $g_\del(z) = 0$ for $z>0$ and $g_\del(z) \sim  K_\del - \log(1+z)$ as $z \ra -1$, where $K_\del$ is a constant depending on $\del$.  For $\La > 0$ and $y > -\La$ let $\tau_{\La,y,t}$ be the first hitting time at $-\La$ for the diffusion $Y_\ve(s), \ s\ge t$, of (\ref{E1}) with $Y_\ve(t) = y$.  We define a function $u_{\ve,\La,\del}(y,t)$ by
\be \label{PP6}
u_{\ve,\La,\del}(y,t) = E\left\{ \exp \left[ -g_\del \left( {Y_\ve(T)}/ \La \right)\right] \ ;  \  \tau_{\La,y,t}> T \right\}.
\ee
Letting $\del \ra 0$ in (\ref{PP6}) we conclude from (\ref{OO6}) that
\be \label{QQ6}
P\left( Y_\ve(T) > 0 \; ; \; \tau_{\La,y,t} > T \  \big| \  Y_\ve(t) = y \right) = \lim_{\del \ra 0} u_{\ve,\La,\del}(y,t).
\ee
It is also clear from (\ref{F1}) that
\be \label{RR6}
u_\ve(0,y,t) = \lim_{ \La \ra \infty } \ P\left( Y_\ve(T) > 0 \; ; \; \tau_{\La,y,t} > T \  \big| \  Y_\ve(t) = y \right) .
\ee
We conclude from (\ref{QQ6}), (\ref{RR6}) that the convexity of $q_\ve(0,y,t)$ in $y$ follows from the log concavity of the function $u_{\ve,\La,\del}(y,t)$ in $y$.

To prove log concavity we first observe that $u_{\ve,\La,\del}(y,t)$ satisfies the PDE (\ref{B1}) for $y > -\La$, $t<T$, with Dirichlet boundary condition $u_{\ve,\La,\del}(y,t) = 0$ at $y = -\La$, and terminal data
\be \label{SS6}
u_{\ve,\La,\del}(y,T) = \exp \left[ -g_\del \left(  y/\La \right) \right], \quad  y > -\La.
\ee
Since the function (\ref{SS6}) is increasing in $y$, it follows from the maximum principle that for $t < T$ the function $u_{\ve,\La,\del}(y,t) $ is also an increasing function of $y$.  From (\ref{OO6}) we see that $u_{\ve,\La,\del}(y,T)$ is $C^2$ for $y \ge -\La$ and $u_{\ve,\La,\del}(-\La,T) = 0$, $\pa u_{\ve,\La,\del}(-\La,T) /\pa y > 0$.  We may therefore apply the regularity result of Proposition A.1.  It follows from this and the Hopf maximum principle \cite{pw}  that
\be \label{TT6}
\pa u_{\ve,\La,\del}(-\La,t)/\pa y > 0, \quad  t \le T.
\ee

Next as in (\ref{G1}) we put $u_{\ve,\La,\del}(y,t)  = \exp [-q_{\ve,\La,\del}(y,t)/\ve]$, and observe that $q_{\ve,\La,\del}(y,t) $ satisfies the PDE (\ref{H1}).  Since $u_{\ve,\La,\del}(y,t) $ is an increasing function of $y$, it follows that $q_{\ve,\La,\del}(y,t) $ is a decreasing function of $y$. Hence $q_{\ve,\La,\del}(y,t) $ is a solution to the PDE 
\be \label{UU6}
\frac{\pa q_{\ve,\La,\del}}{\pa t} + \frac \ve 2 \ \frac{\pa^2 q_{\ve,\La,\del}}{\pa y^2} - B\left(y,t, \frac{\pa q_{\ve,\La,\del}}{\pa y} \right) = 0,
\ee
where the function $B(y,t,p)$ is defined by
\be \label{VV6}
B(y,t,p) = b(y,t)|p| + p^2/2.
\ee
Observe that the function $B(y,t,p)$ is concave in $y$ for all $p \in \R,  \ t \le T$.  Applying Theorem 4.1 of \cite{gk}  to (\ref{UU6}) we see that $q_{\ve,\La,\del}(y,t) $ is convex in $y$ for $y > -\La, \  t < T$, provided we can show that the expression
\be \label{WW6}
q_{\ve,\La,\del}((y + y')/2, t) - \left[q_{\ve,\La,\del}(y,t) + q_{\ve,\La,\del}(y',t)\right]/2
\ee
is less than or equal to 0 as $(y, y', t)$ approaches $(y_\infty, y'_\infty, t_\infty)$ with $t_\infty \le T$ finite, and $(y_\infty, y'_\infty)$ on the boundary of $(-\La, \infty)^2 \subset \R^2$ if $t_\infty < T$, and an arbitrary point in the closure of $(-\La, \infty)^2$ if $t_\infty = T$.

Suppose now that $y_\infty = -\La$ and $-\La < y'_\infty \le \infty$.  From (\ref{PP6}) we see that $u_{\ve,\La,\del}(y,t) > 0$ for $y > -\La$, $t \le T$, whence the limits of the first and third terms in (\ref{WW6}) are finite as $(y,y',t) \ra (y_\infty, y'_\infty, t_\infty)$, whereas the second term converges to $-\infty$.  Thus we may assume $y'_\infty = y_\infty = -\La $.  In that case we observe that the exponential of $\ve^{-1}$ times the expression (\ref{WW6}) is the same as 
\be \label{XX6}
\left[ u_{\ve,\La,\del}(y,t) \  u_{\ve,\La,\del}(y',t)\right]^{1/2} \big/u_{\ve,\La,\del}(\{y + y'\}/2,t).
\ee
From (\ref{TT6}) we may write (\ref{XX6}) as
\be \label{YY6}
p(z,z',t) \ (zz')^{1/2} \big/  \ [(z+z')/2],
\ee
where $z = y + \La, \  z' = y' + \La$, and $\lim\{ p(z,z',t): z,z' \ra 0, \  t \ra t_\infty \} = 1$.  Thus since arithmetic mean exceeds geometric mean, it follows from (\ref{YY6}) that the limit of (\ref{WW6}) as $(y,y',t) \ra (-\La, -\La, t_\infty)$ is less than or equal to 0.

For $t_\infty = T$ we need to show non-positivity of (\ref{WW6}) for any $(y_\infty, y'_\infty)$ in the closure of $(-\La,\infty)^2$.  This follows from Proposition A1 and the convexity of $g_\del(\cd)$. 
\end{proof}
\begin{theorem}  Suppose $b(\cd, \cd)$ satisfies (\ref{A1}).  Then $\pa^2q_\ve(x,y,t)/\pa x \pa y \le 0$ for $x,y \in \R, \  0 \le t < T$.
\end{theorem}
\begin{proof} It will be sufficient to show that for any $h > 0$ the function \ $q_\ve(x + h,y,t) - q_\ve(x,y,t)$ is a decreasing function of $y$.  Letting $g(z)$ be the function
\be \label{M6}
g(z) = z^2, \ \ z < 0; \quad g(z) = 0, \ \ z \ge 0,
\ee
we define $u_{\ve,\del}(x,y,t)$ similarly to (\ref{PP6}) by
\be \label{N6}
u_{\ve,\del}(x,y,t) = E \left\{ \exp \left[ - g \left( \frac{Y_\ve(T) - x}{\del}\right)\right]  \ \Big| \  Y_\ve(t) = y \right\}.
\ee
Evidently $\lim_{\del \ra 0} u_{\ve,\del}(x,y,t) = u_{\ve}(x,y,t)$ and hence the function $q_{\ve,\del}(x,y,t) = -\ve \log u_{\ve,\del}(x,y,t)$ satisfies $\lim_{\del\ra 0}q_{\ve,\del}(x,y,t) = q_{\ve}(x,y,t)$. Arguing as in Lemma 3.1, we also see that $q_{\ve,\del}(x,y,t)$ satisfies the inequality
\be \label{O6}
0 \le q_{\ve,\del}(x,y,t) \le C\left[ (x-y)^2 H(x-y) + 1 \right], \quad  y \in \R, \ 0 \le t < T,
\ee
where $H(\cd)$ is the Heaviside function and $C$ a constant.

In order to prove that $q_\ve(x + h, y, t) - q_\ve(x,y,t)$ is decreasing in $y$ it will be sufficient to show that the function $v_{\ve,\del}(y,t) = q_{\ve,\del}(x + h, y, t) - q_{\ve,\del}(x,y,t)$ is decreasing in $y$ for any $\del > 0$.  To see this we note that $v_{\ve,\del}$ satisfies a PDE
\be \label{P6}
\frac {\pa v_{\ve,\del}}{\pa t} + b_{\ve,\del}(y,t) \; \frac {\pa v_{\ve,\del}}{\pa y} + \frac \ve 2 \; \frac {\pa^2 v_{\ve,\del}}{\pa y^2} = 0, \quad  y \in \R, \  t < T,
\ee
where the drift $b_{\ve,\del}(\cd,\cd)$ is given by the formula 
\be \label{Q6}
b_{\ve,\del}(y,t) = b(y,t) - \frac 1 2 \ \frac{\pa q_{\ve,\del}(x + h, y, t)}{\pa y} - \frac 1 2 \ \frac{\pa q_{\ve,\del}(x,y,t)}{\pa y} \ .
\ee
The terminal data for $v_{\ve,\del}$ is given by
\begin{eqnarray} \label{R6}
v_{\ve,\del}(y,T) = &h[2(x-y)+h]/\del,& \quad  {\rm if \ }y < x, \\
&[x+h - y]^2/\del,& \quad   {\rm if \ }x < y < x+h, \nn \\
&0,& \quad   {\rm if \ } y > x + h. \nn
\end{eqnarray}
Consider now the diffusion process $Y_{\ve,\del}(s)$ defined by
\be \label{S6}
dY_{\ve,\del}(s) = b_{\ve,\del}(Y_{\ve,\del}(s),s)ds + \sqrt{\ve}  \ dW(s).
\ee
From Lemma 3.4 we see that the drift $b_{\ve,\del}(y,s)$ is uniformly Lipschitz in $y$ in any region $y \ge y_0,  \ 0\le t \le T-\eta$, where $y_0 \in \R$ and $\eta >0$ can be arbitrary.  Let $\tau_{y,t}$ be the first hitting time at $y_0$ for $Y_{\ve,\del}(\cd)$ with $Y_{\ve,\del}(t) = y > y_0$.  Then we have the representation,
\begin{multline} \label{T6}
v_{\ve,\del}(y,t) = E \left[v_{\ve,\del}(Y_{\ve,\del}(T-\eta), T-\eta) \; ; \; \tau_{y,t} > T-\eta \right]
\\
+ E \left[v_{\ve,\del}(Y_{\ve,\del}(\tau_{y,t}), \tau_{y,t}) \; ; \; \tau_{y,t} < T-\eta \right].  
\end{multline}
Observe that from (\ref{Q6}) we have that $b_{\ve,\del}(y,s) \ge b(y,s)$, $t \le s < T$.  Hence using (\ref{O6}) we may take the limit $y_0 \ra -\infty$ in (\ref{T6}) to conclude that
\be \label{U6}
v_{\ve,\del}(y,t) = E \left[v_{\ve,\del}(Y_{\ve,\del}(T-\eta), T-\eta) \  \big| \  Y_{\ve,\del}(t) = y \right].
\ee
If $v_{\ve,\del}(z,T-\eta)$ were known to be a decreasing function of $z$ then it would follow from (\ref{U6}) that $v_{\ve,\del}(y,t)$ is a decreasing function of $y$.  Since $u_{\ve,\del}(x,y,T-\eta)$ converges uniformly on any finite interval $a \le y \le b$ as $\eta \ra 0$ to the function $\exp[-g(y-x)/\del]$, we see that $v_{\ve,\del}(z,T-\eta)$ converges uniformly on any finite interval as $\eta \ra 0$ to the decreasing function (\ref{R6}).  Thus we can still conclude from (\ref{U6}) that $ v_{\ve,\del}(y,t)$ is a decreasing function of $y$.  The result follows. 
\end{proof}
It appears that one cannot prove the convexity of $q_\ve(x,y,t)$ as a function of $x$ for fixed $y$ directly, in analogy to Theorem A1, so we shall proceed to showing that $q_\ve(x,y,t)$ is convex jointly in $(x,y)$.  To do this we consider solutions $v(x,y,t)$ to the semi-linear equation
\be \label{V6}
\frac{\pa v}{\pa t} +   b(y,t) \left| \frac{\pa v}{\pa y}\right| + \frac \ve 2 \; \frac{\pa^2v}{\pa y^2} + \frac {\ve'} 2 \; \frac{\pa^2v}{\pa x^2}=0, \quad t<T,
\ee
in the disk $D_R = \{(x,y) : x^2 +y^2 < R^2\} $, with Dirichlet boundary condition and given terminal data.  Thus we wish to solve (\ref{V6}) subject to the conditions
\be \label{W6}
v(x,y,T) = v_0(x,y),  \ (x,y) \in D_R \; ; \quad v(x,y,t) = 0, \  (x,y) \in \pa D_R, \; t < T.
\ee
Using classical techniques \cite{fried, lieb} for proving regularity of solutions to semi-linear parabolic equations, we can establish the following result:
\begin{proposition}:  Assume $b(\cd,\cd)$ satisfies (\ref{A1}) and the terminal function $v_0(x,y)$ is $C^2$ for $(x,y)$ in the closure $\bar{D}_R$ of $D_R$, with $v_0(x,y) = 0$ for $(x,y) \in \pa D_R$.  Then there is a unique solution $v(x,y,t), \; (x,y) \in D_R, \; t<T$, to the terminal value problem (\ref{V6}), (\ref{W6}) which has the property $v(x,y,t)$ is $C^2$ in $(x,y)$, $C^1$ in $t$, and satisfies the inequality
\be \label{Y6}
\sup_{T_0 < t < T} \left\{ |v(x,y,t)| + | Dv(x,y,t)| +|D^2v(x,y,t)| \ : \ (x,y)\in D_R \right\} < \infty
\ee
for any $T_0 < T$. In (\ref{Y6})  $Dv(x,y,t)$ denotes the gradient of $v(x,y,t)$ with respect to $(x,y)$, and $D^2v(x,y,t)$ the Hessian with respect to $(x,y)$.  Additionally, the functions $v(x,y,t)$, $Dv(x,y,t)$ are continuous for $(x,y) \in \bar D_R, \; t \le T$.  The tangential second derivative $\left( y\; \frac \pa{\pa x} - x \; \frac \pa{\pa y} \right)Dv(x,y,t)$ is also continuous.
\end{proposition}
Next we need to establish a Hopf maximum principle (\ref{TT6}) for solutions to (\ref{V6}), (\ref{W6}).
\begin{lem}  Suppose $v_0(x,y), (x,y) \in \bar D_R$, satisfies the conditions of Proposition A2, and in addition $0 \le v_0(x,y) \le 1$, $(x,y) \in \bar D_R$.  Then if $v_0 \not\equiv 0$, the solution $v(x,y,t)$ of (\ref{V6}), (\ref{W6}) satisfies the inequalities
\be \label{Z6}
0 < v(x,y,t) < 1, \quad (x,y) \in D_R, \ \ t < T,
\ee
\be \label{AA6}
 x\; \frac {\pa v}{\pa x} (x,y,t) + y\;\frac {\pa v}{\pa y} (x,y,t) < 0, \quad  (x,y) \in \pa D_R, \ \ t < T.
\ee
\end{lem}
\begin{proof}  The fact that $0\le v(x,y,t) \le 1, \; (x,y) \in D_R, \ t < T$, follows by applying the argument for the weak maximum principle, Theorem 1 of Chapter 3 of \cite{pw}, to the quasilinear equation (\ref{V6}).  Similarly one sees that the argument for the strong maximum principle, Theorem 2 of Chapter 3 in  \cite{pw} applies to (\ref{V6}).  We conclude that (\ref{Z6}) holds.  Finally (\ref{AA6}) follows by applying the argument of Theorem 3 of Chapter 3 in \cite{pw} to (\ref{V6}). 
\end{proof}
The final result we need in order to apply Korevaar's method \cite{kor} to prove convexity in $(x,y)$ of $q_\ve(x,y,t)$ is in effect a comparison principle for solutions of the quasilinear equation (\ref{V6}) to solutions of the linear equation
\be \label{AB6}
\frac{\pa v}{\pa t} +  b(y,t) \frac{\pa v}{\pa y} + \frac \ve 2 \; \frac{\pa^2v}{\pa y^2} + \frac {\ve'} 2 \; \frac{\pa^2v}{\pa x^2}=0.
\ee
\begin{lem}  Assume $b(\cd,\cd)$ satisfies (\ref{A1}) and let $v(x,y,t)$, $(x,y) \in D_R, \ \ t < T,$ be a solution of (\ref{V6}) which is $C^2$ in $(x,y)$ and $C^1$ in $t$.  Assume further that $v(x,y,t)$ extends to a continuous function on $\bar D_R \times \{ t \le T\}$.  Let $w(x,y,t)$ be a second solution to (\ref{V6}) with similar properties to those of $v(x,y,t)$.  Then if for some constant $M$ the inequality
\be \label{AC6}
\int^T_t |b(0,s)|ds + A(T-t) + \sqrt{\ve(T-t)} \le M
\ee
holds, there is a constant $C$ depending only on $M$ such that
\begin{multline} \label{AD6}
 |v(0,0,t) - w(0,0,t)| \le  \\
 \exp\left[-R^2 / C\ve(T-t) \right]  \sup\big\{ |v(x,y,s) - w(x,y,s)| : t \le s < T, \ (x,y) \in \pa D_R \big\} + \\
 \sum_{k \ge 0} \exp \left[-k^2 / C\ve(T-t) \right] \sup  \big\{ |v(x,y,T) -  w(x,y,T) | : (x,y) \in D_{k+M} \cap D_R \big\}, 
\end{multline}
provided $0 < \ve ' \le \ve$.
\end{lem}
\begin{proof} We set $u(x,y,t) = v(x,y,t) - w(x,y,t)$, and observe from (\ref{V6}) that $u(x,y,t)$  satisfies the differential inequality
\be \label{AE6}
\frac{\pa u}{\pa t} -  |b(y,t)| \left| \frac{\pa u}{\pa y}\right| + \frac \ve 2 \; \frac{\pa^2u}{\pa y^2} + \frac {\ve'} 2 \; \frac{\pa^2u}{\pa x^2} \le 0 \ .
\ee
Suppose now that $C(x,y,t)$ satisfies
\be \label{AF6}
\frac{\pa C}{\pa t} -  |b(y,t)| \left| \frac{\pa C}{\pa y}\right| + \frac \ve 2 \; \frac{\pa^2C}{\pa y^2}  + \frac {\ve'} 2  \frac{\pa^2C}{\pa x^2} = 0,  \quad (x,y) \in D_R, \; t < T,
\ee
with boundary and terminal data given by
\be \label{AG6}
C(x,y,T) = u(x,y,T),  \ (x,y) \in D_R \ ; \quad  C(x,y,t) = u(x,y,t),  \ (x,y) \in \pa D_R , \  t < T.
\ee
Then by the maximum principle we have that
$u(x,y,t) \ge C(x,y,t)$ for $(x,y) \in D_R , \   t < T$.  Observe next that $C(x,y,t)$ is the cost function for an optimal control problem.  Thus
\be \label{AH6}
C(x,y,t) = \inf_{\la(\cd,\cd)} \Big\{ E\big[ u(X(T), Y(T), T ) \; ; \; \tau_{x,y,t} > T \big]
+ E \big[ u(X(\tau_{x,y,t}), Y(\tau_{x,y,t}), t) \; ; \; \tau_{x,y,t} < T \big] \Big\},  
\ee
where the stochastic process $[X(s), Y(s)]$ satisfies the SDE
\be \label{AI6}
dY(s) = \la( Y(s), s) ds + \sqrt{\ve} \ dW(s), \quad  dX(s)  = \sqrt{\ve'} \ dW'(s), 
\ee
and  $W(\cd),  \ W'(\cd)$ are independent copies of Brownian motion.  The controller $\la(y,s)$ satisfies the constraints $|\la(y,s)| \le |b(y,s)|, \; y \in \R, \; s \le T$.  The stopping time $\tau_{x,y,t}$ is the first hitting time on $\pa D_R$ for the process (\ref{AI6}) with $X(t) = x, Y(t) = y$.

If we argue now as we did in Lemma 3.1 we can see that $C(0,0,t)$ is bounded below by the negative of the RHS of (\ref{AD6}).  Thus we obtain a lower bound on $v(0,0,t) - w(0,0,t)$.  Since we can repeat the previous argument with $v$ and $w$ interchanged, we also get an upper bound on $v(0,0,t) - w(0,0,t)$, whence (\ref{AD6}) follows. 
\end{proof}
\begin{proposition}   Assume $b(\cd,\cd)$ satisfies (\ref{A1}), and the terminal function $v_0(x,y)$ in Proposition A2 is log concave and satisfies the boundary condition $|Dv_0(x,y)| \not= 0$ for $(x,y) \in \pa D_R$.  If in addition the function $b(y,t)$ is concave in $y$ for $y \in \R, \; t \le T$, then the solution $v(x,y,t)$ of (\ref{V6}), (\ref{W6}) is also log concave.
\end{proposition}
\begin{proof} We again follow the method of Korevaar \cite{kor} as given in \cite{gk} (see also \cite{gp}).  Thus on setting $w(x,y,t)= - \log v(x,y,t)$ we see from (\ref{V6}) that $w(x,y,t)$ satisfies the PDE
\be \label{AJ6}
\frac{\pa w}{\pa t}  + \frac \ve 2 \; \frac{\pa^2w}{\pa y^2} + \frac {\ve'} 2 \; \frac{\pa^2w}{\pa x^2} - B(y,t,Dw) = 0,
\ee
where the function $B(y,t,p)$ is given by the formula
\be \label{AK6}
B(y,t,p) = b(y,t) |p_y| + {\ve p^2_y}/{2} + {\ve' p^2_x}/{2}.
\ee
Since $B(y,t,p)$ satisfies the conditions of Theorem 4.1 of \cite{gk}, the result follows provided we can show that $w(x,y,t)$ is convex for $(x,y,t)$ close to the boundary of $D_R \times \{ t < T\}$.  To see this we argue as in Lemma 2.4 of \cite{kor}.  Observe that it is sufficient to assume $D^2v(x,y,t)$ is bounded as in (\ref{Y6}), and not necessarily continuous as $(x,y,t)$ approaches a boundary point, provided the tangential derivative of $Dv(x,y,t)$ remains continuous.  To see why this is the case consider a 
non-negative $C^2$ function $f$ on the half plane $H = \{(x,z) \in \R^2 : z > 0\}$.  We assume that $f$ extends to a $C^1$ function on the closure $\bar H$ of $H$ and that $f\equiv 0$ on $\pa H$.  In addition we assume the boundary behavior at (0,0) of the second derivatives of $f$ is given by 
\be \label{AL6}
\limsup_{(x,z)\ra (0,0)} \left| \frac{\pa^2 f}{\pa x \pa z} (x,z) \right| < \infty, \ 
\lim_{(x,z)\ra (0,0)} \frac{\pa^2 f}{\pa x^2} (x,z) =0,
\limsup_{(x,z)\ra (0,0)} \left| \frac{\pa^2 f}{ \pa z^2} (x,z) \right| < \infty.
\ee
Now define a function $w(x,y)$ on the domain $U = \{ (x,y) \in \R^2 : y > x^2/2\}$ by $\exp[-w(x,y)]$ $= f(x,y-x^2/2)$.  Then we can see that if $\pa f(0,0)/\pa z > 0$, there exists $\del > 0$ such that the Hessian of $w$ is strictly positive definite for $(x,y) \in U \cap D_\del$.  The convexity of $w(x,y,t)$ close to the boundary of $D_R\times \{ t<T\}$ follows from the regularity result Proposition A2 and Lemma A1 by analogous argument. 
\end{proof}
\begin{theorem}   Assume $b(\cd,\cd)$ satisfies (\ref{A1}) and in addition the function $b(y,t)$ is concave in $y$ for $y \in \R, \  t \le T$.  Then for $t < T$ the function $q_\ve(x,y,t)$ is convex in $(x,y)$ for $(x,y) \in \R^2$.
\end{theorem}
\begin{proof}  Similarly to the proof of Theorem A1, we  approximate $q_\ve(x,y,t)$ by functions defined on finite domains $D_R$ which are convex by virtue of Proposition A3.  To specify the terminal function $v_0(x,y)$, we define a function $f(z)$ for $z< 1$ by
\begin{eqnarray} \label{AM6}
f(z) &=& 0 \ \ {\rm for} \ \ z < 1/2, \qquad  f( 1/2) = f'(1/2) = 0, \\
f''(z) &=& \frac{\exp\left[ -(1-z)^2/(2z-1)\right]}{(1-z)^2}, \quad {\rm for \ }1/2 < z < 1. \nn
\end{eqnarray}
Evidently $f(\cd)$ is a non-negative increasing $C^\infty$ convex function which has the property that $f(z) + \log(1-z)$ has a converging Taylor expansion about $z=1$.  Next let $g : \R \ra \R$ be defined by
\be \label{AN6}
g(z) = z^4, \ \ z < 0\ ; \quad g(z) =0, \ \ z \ge 0,
\ee
whence $g$ is a non-negative decreasing $C^3$ convex function.  It follows from (\ref{AM6}), (\ref{AN6}) that the function $v_0$ with domain $D_R$ defined by
\be \label{AO6}
v_0(x,y) = \exp \left[ -f\big( \sqrt{x^2+y^2} /R \big) - g([y-x]/\del) \right],
\ee
is $C^2$ for $(x,y) \in \bar D_R$ with $v_0(x,y) = 0$ if $(x,y) \in \pa D_R$.  In addition $v_0(x,y)$ is log concave for $(x,y) \in D_R$ and satisfies the non-degenerate boundary condition $|Dv_0(x,y)| \not= 0$ if $(x,y) \in \pa D_R$.  Hence by Proposition A3 the corresponding solution $v_{\del,R}(x,y,t)$ of (\ref{V6}), (\ref{W6}) is log concave in $(x,y)$.

Next we compare the function $v_{\del,R}(x,y,t)$ to a solution of the linear equation (\ref{AB6}).  Thus let $v_{\del}(x,y,t)$ be the unique bounded solution to (\ref{AB6}) in the domain $\{(x,y,t) : (x,y) \in \R^2, t < T\}$ with terminal condition
\be \label{AP6}
v_\del(x,y,t) = \exp \big[ -g([y-x]/\del) \big], \quad  (x,y) \in \R^2.
\ee
From (\ref{AN6}) one sees that $v_\del(x,y,T)$ is an increasing function of $y$ for every $x \in \R$.  The maximum principle implies then that $v_\del(x,y,t) $ is also an increasing function of $y$ for every $x\in \R, \  t < T$.  Thus $v_\del(x,y,t) $ is also a solution to (\ref{V6}).  We may therefore use Lemma A2 to compare the functions $v_\del$ and $v_{\del,R}$.  In view of the fact that $0 \le v_\del \le 1$ and the properties of the function $f$ of (\ref{AM6}), we conclude from (\ref{AO6}), (\ref{AP6}) that
\be \label{AQ6}
\limsup_{R\ra \infty}\big\{ | v_\del(x,y,t) - v_{\del,R}(x,y,t) | : (x,y) \in D_{R_0},  \ T_0 \le t < T \big\} = 0,
\ee
for any $R_0 > 0, \ T_0 < T$.

We conclude from (\ref{AQ6}) and the log concavity of $v_{\del,R}$ that the function $v_\del(x,y,t) $ is also log concave in $(x,y)$ for $(x,y) \in \R^2, \ t < T$.  Observe here that we are using the strong maximum principle to conclude that $v_\del(x,y,t) > 0$, $(x,y) \in \R^2, \  t <T$.  Next we see that the function $v(x,y,t) = \lim_{\del \ra 0} v_\del(x,y,t)$ is the unique bounded solution of (\ref{AB6}) which has terminal data $v(x,y,T) = 0$ if $y < x$, $ v(x,y,T) = 1$ if $y > x$.  Thus $v(x,y,t) = v_{\ve,\ve'}(x,y,t)$ is log concave for $(x,y) \in \R^2$ and $t < T$.  Finally we conclude the convexity of $q_\ve(x,y,t) $ in $(x,y)$  by noting that the function $u_\ve(x,y,t) $ of (\ref{B1}), (\ref{C1}) satisfies $u_\ve(x,y,t) = \lim_{\ve' \ra 0}  v_{\ve,\ve'}(x,y,t) $. 
\end{proof}

\thanks{ {\bf Acknowledgement:}  This research was partially supported by NSF
under grants DMS-0500608 and DMS-0553487.

\end{document}